\documentclass[eqthmnum,nocolour,skinny]{jt-calcs}

\usepackage[euler-digits,euler-hat-accent]{eulervm}
\usepackage{mathrsfs}

\usepackage{tikz-cd}

\DeclareMathOperator{\CS}{CSh}
\DeclareMathOperator{\Cusp}{Cusp}
\DeclareMathOperator{\Fun}{Fun}

\DeclareMathOperator{\Fam}{Fam}

\newcommand{\GU}{\ensuremath{\mathrm{GU}}}

\newcommand{\GC}{\mathbf {G}}
\newcommand{\UC}{\mathbf {U}}
\newcommand{\PC}{\mathbf {P}}
\newcommand{\N}{\mathbb {N}}
\newcommand{\R}{\mathbb {R}}
\newcommand{\lam}{\lambda}
\newcommand{\al}{\alpha}
\newcommand{\CB}{\mathrm {C}}
\newcommand{\KK}{\mathbb {F}}
\newcommand{\IBR}{\mathrm {IBr}}
\newcommand{\LC}{\mathbf {L}}

\newcommand{\lmod}[1]{\ensuremath{#1\text{--{\bf mod}}}}

\newcommand\sd{
	\mathchoice{\mkern1.5mu}{\mkern1.5mu}{}{}%
	{:}%
	\mathchoice{\mkern1.5mu}{\mkern1.5mu}{}{}
}

\setlist[enumerate,1]{label=\textnormal{(\alph*)}}

\title{Lusztig Induction, Unipotent Supports, and Character Bounds}
\author{Jay Taylor and Pham Huu Tiep}
\address{J. Taylor, Department of Mathematics, University of Southern California, Los Angeles,
CA 90089, USA}
\email{jayt@usc.edu\\}
\address{P.~H. Tiep, Department of Mathematics,  Rutgers University, Piscataway, NJ 08854, USA}
\email{tiep@math.rutgers.edu}
\mscno{2010}{20C30}{20C15}

\begin{document}
\begin{abstract}
Recently, a strong exponential character bound has been established in 
\cite{bezrukavnikov-liebeck-shalev-tiep:2017:character-bounds-grps-Lie-type} for all elements $g \in \bG^F$ of a finite reductive group $\bG^F$ which satisfy the condition that the centraliser $C_\bG(g)$ is contained in a $(\bG,F)$-split Levi subgroup $\bM$ of $\bG$ and that $\bG$ is defined over a field of good characteristic. In this paper, assuming a weak version of Lusztig's conjecture relating irreducible characters and characteristic functions of character sheaves holds, we considerably generalize this result by removing the condition that $\bM$ is split. This assumption is known to hold whenever $Z(\bG)$ is connected or when $\bG$ is a special linear or symplectic group and $\bG$ is defined over a sufficiently large finite field.
\end{abstract}

\section{Introduction}

\begin{pa}
Assume $\bG$ is a connected reductive algebraic group, defined over an algebraic closure $\mathbb{F} = \overline{\mathbb{F}}_p$ of the finite field $\mathbb{F}_p$ of prime order $p$, and let $F : \bG \to \bG$ be a Frobenius endomorphism corresponding to an $\mathbb{F}_q$-rational structure on $\bG$. The purpose of this article is to contribute to the problem of bounding the {\it character ratios} $|\chi(g)/\chi(1)|$, where $\chi \in \Irr(\bG^F)$ is an irreducible character of the finite fixed point group $\bG^F$ and $g \in \bG^F$.
\end{pa}

\begin{pa}
Upper bounds for absolute values of character values and character ratios in finite groups have long been of interest, particularly because of a number of
applications, including to random generation, covering numbers, mixing times of random walks, the study of word maps, representation varieties and other areas. Many of these applications are connected with the well-known formula
\begin{equation*}
\frac{\prod_{i=1}^k |C_i|}{|G|}\sum_{\chi\in \Irr(G)} \frac{\chi(c_1)\cdots \chi(c_k)\chi(g^{-1})}{\chi(1)^{k-1}}
\end{equation*}
expressing the number of ways of writing an element $g \in G$, of a finite group $G$, as a product $x_1x_2\cdots x_k$ of elements $x_i\in C_i$, where $C_i = c_i^G$ are $G$-conjugacy classes of elements $c_i$, $1 \leq i \leq k$, and the sum is over the set $\Irr(G)$ of all irreducible characters of $G$, see \cite[10.1]{AH}.
\end{pa}

\begin{pa}
We are particularly interested in so-called {\it exponential character bounds}, namely bounds of the form
\begin{equation}\label{eq:char-bound}
  |\chi(g)| \le \chi(1)^{\alpha_g},
\end{equation}
sometimes with a multiplicative constant, holding for {\it all} characters $\chi \in \Irr(G)$, where $0 \le \alpha_g \le 1$ depends on the group element $g \in G$.
\end{pa}

\begin{pa}
Let us denote by $\mathcal{U}(\bG) \subseteq \bG$ the variety of unipotent elements. Following \cite{bezrukavnikov-liebeck-shalev-tiep:2017:character-bounds-grps-Lie-type} we define for any $F$-stable Levi subgroup $\bM \leqslant \bG$ a constant $\alpha_{\bG}(\bM,F)$ as follows. If $\bM$ is a torus then $\alpha_{\bG}(\bM,F) = 0$ otherwise we have
\begin{equation*}
\alpha_{\bG}(\bM,F) = \max_{1\neq u \in \mathcal{U}(\bM)^F} \frac{\dim u^{\bM}}{\dim u^{\bG}}
\end{equation*}
where $u^{\bG} \subseteq \mathcal{U}(\bG)$ is the $\bG$-conjugacy class of $u$, and similarly for $\bM$. Note the maximum is taken over all \emph{non-identity} unipotent elements.
\end{pa}

\begin{pa}
In \cite[Thm.~1.1]{bezrukavnikov-liebeck-shalev-tiep:2017:character-bounds-grps-Lie-type}, Bezrukavnikov, Liebeck, Shalev, and Tiep were able to obtain a bound of the form \cref{eq:char-bound}, in terms of $\alpha_{\bG}(\bM,F)$, assuming that the centraliser $C_{\bG}(g)$ was contained in $\bM$ and $\bM$ was a proper {\it $(\bG,F)$-split} Levi subgroup, i.e., $\bM$ is the Levi complement of an $F$-stable parabolic subgroup of $\bG$. For the elements that the bound holds, the bound has lead to a number of interesting applications, see \cite[\S5]{bezrukavnikov-liebeck-shalev-tiep:2017:character-bounds-grps-Lie-type}, as well as the sequel \cite{LST}.
\end{pa}

\begin{pa}
The main result of this paper, \cref{thm:character-bound} below, generalises the result of \cite{bezrukavnikov-liebeck-shalev-tiep:2017:character-bounds-grps-Lie-type} to the case of a non-split Levi subgroup $\bM \leqslant \bG$. However, to obtain this generalisation we require two assumptions that are known to hold in a wide number of cases, such as when $Z(\bG)$ is connected, but which remain open in general. To give some idea of these assumptions, let us recall that the space of class functions has two orthonormal bases. Namely, one given by the set of irreducible characters and one given by the characteristic functions of character sheaves.
\end{pa}

\begin{pa}
Lusztig has stated a conjecture that gives, up to multiplication by a diagonal matrix of finite order, the change of basis matrix relating the irreducible characters and characteristic functions of character sheaves, see \cite{lusztig:2018:on-the-definitions-of-almost-characters}. We will require a weak version of this conjecture to hold, which states that the change of basis matrix has a certain block diagonal shape. We call this conjecture the \emph{weak Lusztig conjecture} and we denote it by ($\mathcal{W}_{\bG,F}$), see \cref{pa:weak-Lusztig-conj} for a precise definition. We note that this consequence of Lusztig's conjecture is explicitly mentioned by Lusztig in \cite[2.12(c)]{lusztig:2018:on-the-definitions-of-almost-characters}.
\end{pa}

\begin{pa}
Finally, as is typical in the subject, we will need to assume that one of the main results from Lusztig's work \cite{lusztig:1990:green-functions-and-character-sheaves} on Green functions holds. Specifically, we require that parabolically inducing a cuspidal character sheaf then taking characteristic functions is the same as first taking the characteristic function and then applying Deligne--Lusztig induction. We denote this by ($\mathcal{R}_{\bG,F}$), see \cref{pa:property-DL-ind} for a precise definition. After \cite{lusztig:1990:green-functions-and-character-sheaves} this property is known to hold if $q$ is large enough.
\end{pa}

\begin{thm}\label{thm:character-bound}
There exists a function $f : \mathbb{N} \to \mathbb{N}$ such that the following statement holds. Assume $\bG$ is a connected reductive algebraic group of semisimple rank $r$, $F : \bG \to \bG$ is a Frobenius endomorphism, and $p$ is a good prime for $\bG$. Let $n_0 \geqslant 1$ be an integer such that every $F$-stable Levi subgroup of $\bG$ is $(\bG,F^{n_0})$-split. If ($\mathcal{R}_{\bG,F}$) and ($\mathcal{W}_{\bG,F^n}$) hold, with $n \in \{1,n_0\}$, then for any $F$-stable Levi subgroup $\bM \leqslant \bG$ and any element $g \in \bM^F$ with $C_{\bG}^{\circ}(g) \leqslant \bM$ we have that
\begin{equation*}
|\chi(g)| \leqslant f(r)\cdot\chi(1)^{\alpha_{\bG}(\bM,F)}
\end{equation*}
for any irreducible character $\chi \in \Irr(\bG^F)$.
\end{thm}

\begin{rem}
Let $g = su = us \in \bG^F$ be the Jordan decomposition of $g$. We have $u \in C_{\bG}^{\circ}(s)$ because $C_{\bG}(s)/C_{\bG}^{\circ}(s)$ is a $p'$-group, so $g \in C_{\bG}^{\circ}(s)$. Hence, the assumption that $g \in \bM^F$ and $C_{\bG}^{\circ}(g) \leqslant \bM$ is equivalent to the assumption that $g \in \bG^F$ and $C_{\bG}^{\circ}(s) \leqslant \bM$, see \cref{centralisers}.
\end{rem}

\begin{pa}
We now consider exactly when the assumptions of \cref{thm:character-bound} are known to hold. After work of Shoji the properties ($\mathcal{R}_{\bG,F}$) and ($\mathcal{W}_{\bG,F}$) are known to hold whenever $Z(\bG)$ is connected, see \cite{shoji:1995:character-sheaves-and-almost-characters} and \cite[Thm.~4.2]{shoji:1996:on-the-computation}. Moreover, if $\bG$ is $\SL_n(\mathbb{F})$ or $\Sp_{2n}(\mathbb{F})$ then Bonnaf\'e \cite{bonnafe:2006:sln} and Waldspurger \cite{waldspurger:2004:lusztigs-conjecture} respectively have shown Lusztig's conjecture holds assuming $q$ is large enough, so in these cases ($\mathcal{W}_{\bG,F}$) holds. Here, large enough means that the results of \cite{lusztig:1990:green-functions-and-character-sheaves} hold, so in these cases we may assume ($\mathcal{R}_{\bG,F}$) holds whenever ($\mathcal{W}_{\bG,F}$) holds.
\end{pa}


\begin{pa}\label{types A/C}
As mentioned by Lusztig \cite{lusztig:1990:green-functions-and-character-sheaves} there are functions $g_{\A},g_{\C} : \mathbb{N} \to \mathbb{N}$ such that the results of \cite{lusztig:1990:green-functions-and-character-sheaves} hold for $(\bG,F)$ if $\bG = \SL_n(\mathbb{F})$ and $q > g_{\A}(n)$ and if $\bG = \Sp_{2n}(\mathbb{F})$ and $q > g_{\C}(n)$. It is remarked in \cite[3.2]{waldspurger:2004:lusztigs-conjecture} that one may take $g_{\C}(n) = 2n$. However no justification for this is given. We hope to consider the problem of providing explicit bounds in a future work. Combining these remarks with \cref{thm:character-bound} we obtain the following result.
\end{pa}


\begin{thm}\label{thm:character-bound2}
There exists a function $f : \mathbb{N} \to \mathbb{N}$ such that the following statement holds. Assume $\bG$ is a connected reductive algebraic group of semisimple rank $r$, $F : \bG \to \bG$ is a Frobenius endomorphism, and $p$ is a good prime for $\bG$. Assume in addition that at least one of the following conditions holds:
\begin{enumerate}
         \item $Z(\bG)$ is connected.
	 \item $\bG = \SL_n(\mathbb{F})$ and $q > g_{\A}(n)$.
	\item $\bG = \Sp_{2n}(\mathbb{F})$ and $q > g_{\C}(n)$. 
\end{enumerate}
Then for any $F$-stable Levi subgroup $\bM \leqslant \bG$ and any element $g \in \bM^F$ with $C_{\bG}^{\circ}(g) \leqslant \bM$ we have that
\begin{equation*}
|\chi(g)| \leqslant f(r)\cdot\chi(1)^{\alpha_{\bG}(\bM,F)}
\end{equation*}
for any irreducible character $\chi \in \Irr(\bG^F)$.
\end{thm}

\begin{pa}
Let us consider \cref{thm:character-bound} in the case that $\bM \leqslant \bG$ is a $(\bG,F)$-split Levi subgroup. In \cite{bezrukavnikov-liebeck-shalev-tiep:2017:character-bounds-grps-Lie-type} it is shown that if $p$ is a good prime for $\bG$, the derived subgroup of $\bG$ is simple, and $g \in \bG^F$ is an element such that $C_{\bG}(g)^F \leqslant \bM^F$ then the conclusion of \cref{thm:character-bound} holds. However, for such an element we have $g \in \bM^F$ and $C_{\bG}^{\circ}(g) \leqslant \bM$ because $C_{\bG}^{\circ}(g)^F \leqslant \bM^F$ and $\bM$ is $(\bG,F)$-split, see \cref{centralisers}. Hence, our theorem is a generalization of the result in \cite{bezrukavnikov-liebeck-shalev-tiep:2017:character-bounds-grps-Lie-type}. The exact same argument used in (ii) of the proof of \cite[Thm.~1.1]{bezrukavnikov-liebeck-shalev-tiep:2017:character-bounds-grps-Lie-type} shows that, when $\bM$ is $(\bG,F)$-split, it is sufficient to consider the case when $Z(\bG)$ is connected, which is covered by \cref{thm:character-bound2}. Thus, the assumption that the derived subgroup of $\bG$ is simple in \cite{bezrukavnikov-liebeck-shalev-tiep:2017:character-bounds-grps-Lie-type} is not necessary.
\end{pa}

\begin{pa}
After \cref{thm:character-bound2} one could attempt to prove \cref{thm:character-bound} in general by using the standard technique of regular embeddings. However, this does not seem to be effective in all cases. The following, which applies to any finite reductive group in good characteristic, constitutes what can be achieved with this method. This will be helpful for various applications, especially for the following reason. In a number of applications, one can usually rule out the characters of quasisimple groups of Lie type of large degree by various {\it ad hoc} arguments. On the other hand, in the notation of Corollary \ref{cor:bound}, and under the assumptions that $\bG$ is simple, simply connected, and that $\chi$ is not $\widetilde{\bG}^F$-invariant, one can usually show that $\chi(1)$ is very large (of order of magnitude $|\bG^F|^{\frac{1}{4}}$), making it possible to rule $\chi$ out. 
\end{pa}

\begin{cor}\label{cor:bound}
Let $\bG$ be a connected reductive algebraic group of semisimple rank $r$, and let $F : \bG \to \bG$ be a Frobenius endomorphism. Assume that $p$ is a good prime for $\bG$. Then for any $F$-stable Levi subgroup $\bM \leqslant \bG$ and any element $g \in \bM^F$ with $C_{\bG}^{\circ}(g) \leqslant \bM$ we have that
\begin{equation*}
|\chi(g)| \leqslant f(r)\cdot\chi(1)^{\alpha_{\bG}(\bM,F)}
\end{equation*}
for any irreducible character $\chi \in \Irr(\bG^F)$, with $f$ as defined in \cref{thm:character-bound}, provided that at least one of the following holds:
\begin{enumerate}
\item there is a regular embedding $\bG \to \widetilde{\bG}$, with a Frobenius endomorphism $F: \widetilde{\bG} \to \widetilde{\bG}$ extending $F: \bG \to \bG$, such that either $\chi$ is $\widetilde{\bG}^F$-invariant, or $g^{\widetilde{\bG}^F} = g^{\bG^F}$,
\item $C_\bG(g)$ is connected, or
\item $[\bG,\bG]$ is simply connected and $g$ is semisimple.
\end{enumerate}
\end{cor}

As a consequence of our main result, and combining with results of \cite{bezrukavnikov-liebeck-shalev-tiep:2017:character-bounds-grps-Lie-type}, we obtain the following, explicit and asymptotically optimal, exponential character bound for general and special linear groups:

\begin{thm}\label{thm:glsl}
There is a function $h: \mathbb{N} \to \mathbb{N}$ such that the following statement holds for any $n \geq 5$ and for any prime power $q$. If $G = \GL_n(q)$ or $G=\SL_n(q)$ and $g \in G \smallsetminus Z(G)$, then
\begin{equation*} 
|\chi(g)| \leq h(n)\chi(1)^{(n-2)/(n-1)}
\end{equation*}
for all $\chi \in \Irr(G)$. In fact, if $q > 3n^2$, then one can take $h(n) = 3(f(n-1)+1)$, where $f$ is the function in \cref{thm:character-bound}.
\end{thm}

Our next result yields an asymptotically optimal, exponential character bound for general and special unitary groups:

\begin{thm}\label{thm:gusu}
There are functions $h^*,C^*: \mathbb{N} \to \mathbb{N}$ such that the following statement holds for any $n \geq 10$ and for any prime power $q \geq C^*(n)$.
If $G = \GU_n(q)$ or $G = \SU_n(q)$ and $g \in G \smallsetminus Z(G)$, then
\begin{equation*}
|\chi(g)| \leq h^*(n)\chi(1)^{(n-2)/(n-1)}
\end{equation*}
for all $\chi \in \Irr(G)$. 
\end{thm}

Our next consequence improves on Corollary 1.14 and Theorem 1.15 of \cite{bezrukavnikov-liebeck-shalev-tiep:2017:character-bounds-grps-Lie-type}, 
and extends the main result of \cite{Hil}:
 
\begin{cor}\label{cor:slsu} Let $G= \SL_n(q)$ or $G = \SU_n(q)$,  $x$ be an arbitrary non-central element of $G$, and let $C = x^G$. For $t > 0$ an integer, let $C^t = \{x_1\cdots x_t \mid x_i \in C\}$.
\begin{enumerate}[label={\normalfont (\roman*)}]
	\item If $t>2n$, then $C^t = G$ almost uniformly pointwise as $q\rightarrow \infty$.
	\item The mixing time $T(G,x)$ of the random walk on the Cayley graph $\Gamma(G,C)$ is at most $n$ for large $q$.
\end{enumerate}
\end{cor}

\subsection*{Outline of the Paper}

\begin{pa}
We now give a brief outline of the paper. First, in \cref{1.10} we establish the first main result of the paper, \cref{thm:character-bound},
assuming \cref{thm:main-characters,thm:main-char-sheaves} which will be proved in \cref{sec:proof-of-closure}. 
In \cref{sec:proj-G-sets} we introduce the convenient language of projective $G$-sets which is used in \cref{sec:coset-multiplicities} where we make remarks regarding induction of projective representations from cosets of a finite group. We recall, in \cref{sec:char-sheaves,sec:parabolic-ind-char-sheaves,sec:ind-cusp-char-sheaves}, various constructions and results we need concerning parabolic induction of character sheaves. Property ($\mathcal{R}_{\bG,F}$) is stated in \cref{pa:property-DL-ind}.
\end{pa}

\begin{pa}
In \cref{sec:harish-chandra-param-char-sheaves} we recall the Harish-Chandra parameterization of character sheaves. Our main aim in this section is to give a realisation of characteristic functions of character sheaves that allows us to relate Deligne--Lusztig induction to induction of projective representations on cosets of a finite group, see \cref{prop:bumper-DL-ind-char-sheaves}. This involves a careful analysis of the endomorphism algebra of an induced cuspidal character sheaf.
\end{pa}

\begin{pa}
The main result of \cref{sec:harish-chandra-param-char-sheaves} may be seen as a shadow at the level of functions of a comparison theorem for induction of character sheaves, which we state in \cref{sec:comparison-thms}. In the case of unipotently supported character sheaves these results have been previously established by Lusztig \cite[2.4(d)]{lusztig:1986:on-the-character-values} and Digne--Lehrer--Michel \cite[3.3]{digne-lehrer-michel:1997:gelfand-grave-characters-disconnected}.
\end{pa}

\begin{pa}
Assuming $\bM \leqslant \bG$ is $(\bG,F)$-split, we show in \cref{sec:split-levi-case} that \cref{thm:main-characters} holds following Bezrukavnikov--Liebeck--Shalev--Tiep \cite[2.6]{bezrukavnikov-liebeck-shalev-tiep:2017:character-bounds-grps-Lie-type}. In particular, we remove the assumption in \cite{bezrukavnikov-liebeck-shalev-tiep:2017:character-bounds-grps-Lie-type} that the derived subgroup of $\bG$ is simple. We relate the unipotent supports of characters and character sheaves in \cref{sec:unip-supp}. We are then able to prove \cref{thm:main-characters,thm:main-char-sheaves} in \cref{sec:proof-of-closure}. In \cref{sec:multiplicities} we bound the multiplicities in Deligne--Lusztig induction and prove \cref{cor:bound}. Finally, in \cref{sec:type-A,sec:twisted-type-A} we prove our main results \cref{thm:glsl,thm:gusu} and \cref{cor:slsu} on groups of type $\A$.
\end{pa}

\begin{pa}
At the end of the paper we include an appendix recalling how the simple summands of a semisimple object $A$ of an abelian category $\mathscr{A}$ can be parameterised in terms of the simple modules of the endomorphism algebra $\End_{\mathscr{A}}(A)$. We then study the effect of a linear functor under this parameterisation.
\end{pa}

\begin{pa}[Notation]\label{pa:notation}
For any group $G$ we set ${}^gx := gxg^{-1}$ and $x^g := g^{-1}xg$ for any $g,x \in G$. Moreover, we denote by $\boldsymbol{\iota}_g : G \to G$ the function defined by $\boldsymbol{\iota}_g(x) = gxg^{-1}$. For any category $\mathscr{A}$ we denote by $\Irr(\mathscr{A})$ the isomorphism classes of simple objects in $\mathscr{A}$. The isomorphism class of an object $A \in \mathscr{A}$ will be denoted by $[A] \in \Irr(\mathscr{A})$. If $K \in \mathscr{A}$ is a semisimple object then we denote by $\Irr(\mathscr{A} \mid K)$ the set of isomorphism classes of the simple summands of $K$. If $\mathcal{A}$ is a $k$-algebra, with $k$ a field, then we denote by $\lmod{\mathcal{A}}$ the category of finite dimensional left $\mathcal{A}$-modules.

Throughout all varieties are assumed to be over $\mathbb{F} = \overline{\mathbb{F}}_p$. Moreover, $\bG$ denotes a fixed connected reductive algebraic group and $F : \bG \to \bG$ is a Frobenius endomorphism. We assume $\bT_0 \leqslant \bG$ is a maximal torus and $(\bG^{\star},\bT_0^{\star},F)$ is a triple dual to $(\bG,\bT_0,F)$. We denote by $W_{\bG}(\bT_0)$ the Weyl group $N_{\bG}(\bT_0)/\bT_0$ of $\bG$. If the choice of torus is irrelevant, or implicit, then we simply write $W_{\bG}$ for $W_{\bG}(\bT_0)$.
\end{pa}

\begin{acknowledgments}
The authors thank Gunter Malle for his comments on a preliminary version of this work, C\'edric Bonnaf\'e for useful discussions on the topic of character sheaves, and Roman Bezrukavnikov for helpful discussions on the subject of the paper. Part of the work took place while the authors were in residence at the Mathematical Sciences Research Institute in Berkeley, California, during Spring 2018. The authors gratefully acknowledge the support of the National Science Foundation under grant DMS-1440140. The second author was partially supported by the NSF grant DMS-1840702 and the Joshua Barlaz Chair in Mathematics. Finally, we thank the referee for their careful reading of and 
helpful comments on the paper.
\end{acknowledgments}

\section{Proof of \cref{thm:character-bound}}\label{1.10}
\begin{pa}
Our proof of \cref{thm:character-bound} follows the approach used in \cite{bezrukavnikov-liebeck-shalev-tiep:2017:character-bounds-grps-Lie-type}, although each step is considerably more difficult than in the split case. Here we outline the main steps of the proof with the technical results left to the remaining sections of the paper. Associated to the Levi subgroup $\bM$ we have corresponding Deligne--Lusztig induction and restriction maps $R_{\bM}^{\bG} : \Class(\bM^F) \to \Class(\bG^F)$ and ${}^*R_{\bM}^{\bG} : \Class(\bG^F) \to \Class(\bM^F)$, which are linear maps between the spaces of class functions taking irreducible characters to virtual characters. If $\bM$ is $(\bG,F)$-split then $R_{\bM}^{\bG}$ and ${}^*R_{\bM}^{\bG}$ are simply Harish-Chandra induction and restriction which take irreducible characters to characters.
\end{pa}

\begin{pa}
It is well known that under the assumptions of \cref{thm:character-bound} we have
\begin{equation}\label{eq:lusztig-res-value}
\chi(g) = {}^*R_{\bM}^{\bG}(\chi)(g),
\end{equation}
see \cref{lem:restrict-equal-val}. For any irreducible character $\eta \in \Irr(\bM^F)$ and $\chi \in \Irr(\bG^F)$ we denote by
\begin{equation*}
m(\eta,\chi) = \langle \eta, {}^*R_{\bM}^{\bG}(\chi)\rangle_{\bM^F} = \langle \chi, R_{\bM}^{\bG}(\eta)\rangle_{\bG^F} \in \mathbb{Z}
\end{equation*}
the multiplicity of $\eta$ in ${}^*R_{\bM}^{\bG}(\chi)$. Expanding ${}^*R_{\bM}^{\bG}(\chi)$ in terms of $\Irr(\bM^F)$ we get from \cref{eq:lusztig-res-value} that
\begin{equation}\label{eq:Lusztig-res-bound}
|\chi(g)| \leqslant \sum_{\eta \in \Irr(\bM^F)} |m(\eta,\chi)|\cdot\eta(1),
\end{equation}
where we use the trivial bound $|\eta(g)| \leqslant \eta(1)$. This bound provides our approach to proving \cref{thm:character-bound}.
\end{pa}

\begin{pa}\label{pa:bound}
Let $\Irr(\bM^F \mid {}^*R_{\bM}^{\bG}(\chi))$ denote the set of irreducible constituents of ${}^*R_{\bM}^{\bG}(\chi)$. We will show that there exist three integers $f_1(\bG),f_2(\bG),f_3(\bG) \in \mathbb{N}$ such that all the following inequalities hold:
\begin{enumerate}
	\item $|\Irr(\bM^F \mid {}^*R_{\bM}^{\bG}(\chi))| \leqslant f_1(\bG)$,
	\item $|m(\eta,\chi)| \leqslant f_2(\bG)$ for any $\eta \in \Irr(\bM^F \mid {}^*R_{\bM}^{\bG}(\chi))$,
	\item $\eta(1) \leqslant f_3(\bG)\cdot\chi(1)^{\alpha_{\bG}(\bM,F)}$ for any $\eta \in \Irr(\bM^F \mid {}^*R_{\bM}^{\bG}(\chi))$.
\end{enumerate}
If such integers exist then from \cref{eq:Lusztig-res-bound} we have that
\begin{equation*}
|\chi(g)| \leqslant f(\bG) \cdot \chi(1)^{\alpha_{\bG}(\bM,F)}.
\end{equation*}
where $f(\bG) = f_1(\bG)\cdot f_2(\bG)\cdot f_3(\bG)$.
\end{pa}


\begin{pa}\label{pa:bound_RLG}
Finding an integer $f_1(\bG)$ satisfying \ref{pa:bound}(a) is achieved easily using classic results of Deligne--Lusztig, namely the Mackey formula for tori. In particular, if $W_{\bG}$ is the Weyl group of $\bG$ defined with respect to some (any) maximal torus of $\bG$ then we may take $f_1(\bG) = |W_{\bG}|^2$, see the proof of \cite[Lem.~17]{lusztig:1976:on-the-finiteness}. This requires no assumptions on $(\bG,F)$.
\end{pa}

\begin{pa}
Finding integers satisfying \ref{pa:bound}(b) and \ref{pa:bound}(c) is appreciably more difficult. To tackle both of these problems we rely on our deep assumptions ($\mathcal{R}_{\bG,F}$) and ($\mathcal{W}_{\bG,F}$) which allow us to translate questions concerning Deligne--Lusztig induction $R_{\bM}^{\bG}$ to corresponding questions about parabolic induction of character sheaves.
\end{pa}

\begin{pa}
Our approach to \ref{pa:bound}(b) is to find an integer $f_2'(\bG) \in \mathbb{N}$, depending only on the root system of $\bG$, such that
\begin{equation*}
\langle R_{\bM}^{\bG}(\eta), R_{\bM}^{\bG}(\eta)\rangle_{\bG^F} = \sum_{\chi \in \Irr(\bG^F)} m(\eta,\chi)^2 \leqslant f_2'(\bG)
\end{equation*}
for any irreducible character $\eta \in \Irr(\bM^F)$. We may then take $f_2(\bG) = f_2'(\bG)^{\frac{1}{2}}$ in \ref{pa:bound}(b). Here, a brief comment is in order. As we assume $p$ is good for $\bG$ it is known by work of Bonnaf\'e--Michel \cite{bonnafe-michel:2011:mackey-formula} that the Mackey formula holds. Using the Mackey formula one can obtain a function $f_2'(\bG)$, as above, which is recursively defined. This approach works whenever the Mackey formula holds but the approach we now outline yields an explicit bound, which is significantly better than that achieved with the Mackey formula.
\end{pa}

\begin{pa}
We find $f_2'(\bG)$ in two steps. First, we consider two $F$-stable character sheaves $A_1,A_2 \in \CS(\bM)$ and their corresponding characteristic functions $\mathcal{X}_{A_1}$ and $\mathcal{X}_{A_2}$, defined with respect to some fixed Weil structure. We then obtain a bound
\begin{equation*}
|\langle R_{\bM}^{\bG}(\mathcal{X}_{A_1}), R_{\bM}^{\bG}(\mathcal{X}_{A_2})\rangle_{\bG^F}| \leqslant |W_{\bG}|
\end{equation*}
by expressing the inner product in terms of coset induction in relative Weyl groups, see \cref{prop:inner-prod-char-sheaves}. It is here that we use the property ($\mathcal{R}_{\bG,F}$). Writing $\eta$ as a linear combination of characteristic functions of character sheaves on $\bM$ we are then able to bound $\langle R_{\bM}^{\bG}(\eta), R_{\bM}^{\bG}(\eta)\rangle_{\bG^F}$ once we can bound the number of character sheaves involved in such a decomposition.
\end{pa}

\begin{pa}\label{pa:bound-inner}
For this step we crucially use the weak Lusztig conjecture ($\mathcal{W}_{\bG,F}$) which gives us sharp control over this number. It is at this stage that the following important invariants appear. Namely, if $\bH$ is an algebraic group and $x \in \bH$ is an element then we denote by $A_{\bH}(x)$ the component group $C_{\bH}(x)/C_{\bH}^{\circ}(x)$ of the centraliser of $x$. We set
\begin{equation*}
B(\bG) = |Z(\bG)/Z^{\circ}(\bG)|\cdot \max_{\bH}\max_u |A_{\bH}(u)| \qquad\text{and}\qquad D(\bG) = \max_{\bH}\max_u (\dim u^{\bH})
\end{equation*}
where the maxima are taken over: all semisimple and simply connected groups $\bH$ of rank at most $r$ and all unipotent elements $u \in \mathcal{U}(\bH)$. Using Lusztig's conjecture we are able to show that we may take $f_2'(\bG) = B(\bG)^4 \cdot |W_{\bG}|$, see \cref{prop:multiplicities}.
\end{pa}

\begin{pa}
Now finally let us consider \ref{pa:bound}(c). For each irreducible character $\chi \in \Irr(\bG^F)$ there is a polynomial $\mathbb{D}_{\chi}(t) \in \mathbb{Q}[t]$ such that $\chi(1) = \mathbb{D}_{\chi}(q)$. Moreover, this polynomial has the form
\begin{equation*}
\mathbb{D}_{\chi}(t) = \frac{1}{n_{\chi}}(t^{A_{\chi}} + \cdots \pm t^{a_{\chi}})
\end{equation*}
where the coefficients of all intermediate powers $t^i$, with $a_{\chi} < i < A_{\chi}$ are integers, and $n_{\chi} > 0$ is an integer. By work of Lusztig \cite{lusztig:1992:a-unipotent-support} and Geck--Malle \cite{geck-malle:2000:existence-of-a-unipotent-support} the invariants $n_{\chi}$, $A_{\chi}$, and $a_{\chi}$, occuring in the degree polynomial have geometric interpretations.
\end{pa}

\begin{pa}
Specifically, building on work of Lusztig, Geck--Malle have shown that to each irreducible character $\chi$ one may associate a unique $F$-stable unipotent class $\mathcal{O}_{\chi} = \mathcal{O}_{\chi}^{\bG}$ of $\bG$ called the \emph{unipotent support} of $\chi$ (this can be done without any assumption on $p$). If $p$ is a good prime for $\bG$ and $u \in \mathcal{O}_{\chi}$ then it is known that we have
\begin{equation*}
A_{\chi} = \dim \mathcal{O}_{\chi^*}/2, \qquad a_{\chi} = \dim \mathfrak{B}_u^{\bG},\qquad\text{and}\qquad n_{\chi} \mid |A_{\bG}(u)|,
\end{equation*}
where $\mathfrak{B}_u^{\bG}$ is the variety of Borel subgroups of $\bG$ containing $u$ and $\chi^* = \pm D_{\bG^F}(\chi) \in \Irr(\bG^F)$ is the Alvis--Curtis dual of $\chi$.
\end{pa}

\begin{pa}
This geometric interpretation explains the appearance of the $\alpha_{\bG}(\bM,F)$-bound occurring in \cref{thm:character-bound} and also the occurrence of the term $B(\bG)$ above. To achieve \ref{pa:bound}(c) one can now try to get a relationship between the unipotent support of $\chi$ and the unipotent support of $\eta$ when $\langle \eta, {}^*R_{\bM}^{\bG}(\chi)\rangle_{\bM^F}$ is non-zero. To describe such a relationship let us write $X \leqslant Y$ if $X,Y \subseteq \mathcal{U}(\bG)$ are subsets of the unipotent variety satisfying $X \subseteq \overline{Y}$ (the Zariski closure of $Y$). In \cref{sec:proof-of-closure} we will prove the following, whose conclusion was shown in \cite{bezrukavnikov-liebeck-shalev-tiep:2017:character-bounds-grps-Lie-type} under the assumptions that $p$ is a good prime for $\bG$ and $\bM$ is $(\bG,F)$-split.
\end{pa}

\begin{thm}\label{thm:main-characters}
Assume $p$ is a good prime for $\bG$, ($\mathcal{R}_{\bG,F}$) holds, and that there exists an integer $n_0 \geqslant 1$ such that every $F$-stable Levi subgroup $\bM \leqslant \bG$ is $(\bG,F^{n_0})$-split and the weak Lusztig conjecture ($\mathcal{W}_{\bG,F^{n_0}}$) holds. Then for any $F$-stable Levi subgroup $\bM \leqslant \bG$ and irreducible characters $\chi \in \Irr(\bG^F)$ and $\eta \in \Irr(\bM^F)$ satisfying $\langle \chi, R_{\bM}^{\bG}(\eta)\rangle_{\bG^F} \neq 0$ we have $\mathcal{O}_{\eta} \leqslant \mathcal{O}_{\chi}$ and $\mathcal{O}_{\eta^*} \leqslant \mathcal{O}_{\chi^*}$.
\end{thm}

\begin{proof}[of \cref{thm:character-bound}]
Assume $\chi \in \Irr(\bG^F)$ and $g \in \bG^F$ are as in the statement of \Cref{thm:character-bound}. We will produce a bound on $|\chi(g)|$ following \cref{pa:bound}. Recall that we have already seen that we may take $f_1(\bG) = |W_{\bG}|^2$, in \cref{pa:bound_RLG}, and $f_2(\bG)^2 = B(\bG)^4\cdot |W_{\bG}|$, in \cref{pa:bound-inner}. We now find $f_3(\bG)$. If $\eta\in \Irr(\bM^F \mid {}^*R_{\bM}^{\bG}(\chi))$ then after \cref{thm:main-characters} we get the following numerical relationship between the degree polynomials of $\eta$ and $\chi$:
\begin{equation*}
A_{\eta} = \dim \mathcal{O}_{\eta^*}^{\bM}/2 \leqslant \dim \mathcal{O}_{\chi^*}^{\bG}/2 = A_{\chi}.
\end{equation*}
With this relationship in hand an identical argument to that used in the proof of \cite[Thm.~1.1]{bezrukavnikov-liebeck-shalev-tiep:2017:character-bounds-grps-Lie-type} shows that we may take
\begin{equation*}
f_3(\bG) = 3^{D(\bG)/2} \cdot B(\bG).
\end{equation*}

Putting things together we see that we may choose the integer $f(\bG)$, defined as in \cref{pa:bound}, to be
\begin{equation*}
f(\bG) = 3^{D(\bG)/2} \cdot B(\bG)^3 \cdot |W_{\bG}|^{\frac{5}{2}}.
\end{equation*}
Let $\bG_{\simc}$ be the simply connected cover of the derived subgroup $[\bG,\bG]$ of $\bG$. We have $Z(\bG) = Z([\bG,\bG])Z^{\circ}(\bG)$, so certainly $|Z(\bG)/Z^{\circ}(\bG)| \leqslant |Z([\bG,\bG])| \leqslant |Z(\bG_{\simc})|$. Thus, we have $B(\bG) \leqslant B(\bG_{\simc})$ so $f(\bG) \leqslant f(\bG_{\simc})$.

We now define a function $f : \mathbb{N} \to \mathbb{N}$ by setting $f(r) = \max_{\bG} f(\bG)$ where the maximum is taken over all semisimple and simply connected groups of rank $r$. We note that there are finitely many such groups for a given $r$ so the maximum exists. That this function satisfies the conclusion of \cref{thm:character-bound} is clear.
\end{proof}
%

\begin{pa}
As indicated by our assumptions our proof of \cref{thm:main-characters} is not independent of that given in \cite{bezrukavnikov-liebeck-shalev-tiep:2017:character-bounds-grps-Lie-type} and crucially uses the split case to deal with the non-split case. Indeed, it follows from work of Lusztig \cite{lusztig:1992:a-unipotent-support} and the first author \cite{taylor:2016:GGGRs-small-characteristics} that, if $p$ is a good prime, then to each character sheaf $A$ one may associate a well-defined unipotent class $\mathcal{O}_A \subseteq \mathcal{U}(\bG)$ of $\bG$ which is also called its unipotent support. By relating parabolic induction of character sheaves to Harish-Chandra induction, and using the results in \cite{bezrukavnikov-liebeck-shalev-tiep:2017:character-bounds-grps-Lie-type}, we are able to establish the following result,
whose proof is given in \cref{sec:proof-of-closure}.
\end{pa}

\begin{thm}\label{thm:main-char-sheaves}
Assume $p$ is a good prime for $\bG$. Let $\bM \leqslant \bG$ be a Levi subgroup and let $A \in \CS(\bG)$ and $B \in \CS(\bM)$ be character sheaves satisfying $A \mid \ind_{\bM}^{\bG}(B)$. If there exists a Frobenius endomorphism $F_1 : \bG \to \bG$ such that: $\bM$ is $(\bG,F_1)$-split, $A$ and $B$ are $F_1$-stable, and the weak Lusztig conjecture ($\mathcal{W}_{\bG,F_1}$) holds, then $\mathcal{O}_B \leqslant \mathcal{O}_A$.
\end{thm}

\section{\texorpdfstring{Functions on Projective $G$-sets}{Functions on Projective G-sets}}\label{sec:proj-G-sets}
\begin{pa}\label{pa:G-set-funcs}
We fix an algebraic closure $\Ql$, where $\ell \neq p$ is a prime, and an involutive automorphism $\overline{\phantom{\zeta}} : \Ql \to \Ql$ which satisfies $\overline{\zeta} = \zeta^{-1}$ for any root of unity $\zeta \in \Ql^{\times}$. Let $G$ be a finite group and $X$ a finite (left) $G$-set with action map $\cdot : G \times X \to X$. A function $c : G \times X \to \Ql^{\times}$ is called a \emph{$2$-cocycle} if
\begin{equation}\label{eq:cocyle-cond}
c(gh,x) = c(g,h\cdot x)c(h,x)
\end{equation}
for all $g,h \in G$ and $x \in X$. The pair $(X,c)$ is called a \emph{projective $G$-set}. We say $(X,c)$ is \emph{unital} if $c(g,x) \in \Ql^{\times}$ is a root of unity for all $g \in G$ and $x \in X$. If $H \leqslant G$ is a subgroup and $(Y,d)$ is a projective $H$-set then a function $\psi : Y \to X$ is a \emph{projective $H$-map} if $\psi(h\cdot y) = h\cdot\psi(y)$ and $d(h,y) = c(h,\psi(y))$ for all $h \in H$ and $y \in Y$.
\end{pa}

\begin{pa}\label{pa:G-space-funcs}
If $(X,c)$ is a projective $G$-set then we denote by $\Fun_G(X,c)$ the $\Ql$-vector space of functions $f : X \to \Ql$ satisfying $f(x) = c(g,x)f(g\cdot x)$ for all $g \in G$ and $x \in X$. Note that \cref{eq:cocyle-cond} ensures that $f(g\cdot (h\cdot x)) = f(gh\cdot x)$. The space $\Fun_G(X,c)$ has a natural $\Ql$-valued form $\langle -,-\rangle_X$ defined by
\begin{equation*}
\langle f,f'\rangle_X = \frac{1}{|G|}\sum_{x \in X} f(x)\overline{f'(x)}.
\end{equation*}
We say $x$ is \emph{$c$-regular} if $c(g,x) = 1$ for all $g \in \Stab_G(x)$ (the stabiliser of $x$). If $f(x) \neq 0$, for $f \in \Fun_G(X,c)$ and $x \in X$, then we must have $x$ is $c$-regular.
\end{pa}

\begin{pa}\label{pa:orb-funcs}
A straightforward computation using \cref{eq:cocyle-cond} shows that the subset of $c$-regular elements of $X$ is a union of orbits. If $x \in X$ is $c$-regular then we obtain a well-defined function $\pi_x \in \Fun_G(X,c)$ by setting
\begin{equation*}
\pi_x(y) = \begin{cases}
c(g,x)^{-1}|\Stab_G(x)| &\text{if }y = g\cdot x\text{ for some }g \in G,\\
0 &\text{otherwise}.
\end{cases}
\end{equation*}
If $x_1,\dots,x_k \in X$ are representatives for the orbits of $c$-regular elements then $(\pi_{x_1},\dots,\pi_{x_k})$ is a basis of $\Fun_G(X,c)$. Moreover, if $(X,c)$ is unital then $\langle f,\pi_x\rangle_X = f(x)$.
\end{pa}

\begin{pa}\label{pa:ind-map}
Assume now that $H \leqslant G$ is a subgroup, $(Y,d)$ is a projective $H$-set, and $\psi : Y \to X$ is a projective $H$-map. Then we have a restriction map $\psi^* : \Fun_G(X,c) \to \Fun_H(Y,d)$ defined by $\psi^*(f) = f\circ\psi$. We also have an induction map $\psi_* : \Fun_H(Y,d) \to \Fun_G(X,c)$ defined by
\begin{equation*}
\psi_*(f)(x) = \frac{1}{|H|}\sum_{\substack{(s,y) \in G\times Y\\ s \cdot x = \psi(y)}} c(s,x)f(y).
\end{equation*}
The following is elementary and left to the reader.
\end{pa}

\begin{prop}\label{prop:fun-G-sets-bumper}
Assume $G$ is a finite group and $H \leqslant G$ is a subgroup. Moreover, let $(X,c)$ be a unital projective $G$-set and let $(Y,d)$ be a unital projective $H$-set. If $\psi : Y \to X$ is an $H$-map then the following statements hold:
\begin{enumerate}[label=\textnormal{(\roman*)}]
	\item for any $f \in \Fun_H(Y)$ and $f' \in \Fun_G(X)$ we have $\langle \psi_*(f),f'\rangle_X = \langle f,\psi^*(f')\rangle_Y$,
	\item for any $d$-regular element $y \in Y$ we have $\psi_*(\pi_y) = \pi_{\psi(y)}$,
	\item if $K \leqslant H$ is a subgroup and $(Z,e)$ is a unital projective $K$-set then for any projective $K$-map $\lambda : Z \to Y$ we have $(\psi\circ\lambda)^* = \lambda^*\circ\psi^*$ and $(\psi\circ\lambda)_* = \psi_*\circ\lambda_*$.
\end{enumerate}
\end{prop}

%
%

\section{Multiplicities for Coset Induction}\label{sec:coset-multiplicities}
\begin{pa}\label{pa:conj-proj-G-set}
Assume $G$ is a finite group and let $Z^2(G,\Ql)$ be the set of $2$-cocycles $G \times G \to \Ql^{\times}$ with $G$ acting trivially on $\Ql$. Recall that if $\alpha \in Z^2(G,\Ql)$ then the twisted group algebra $\Ql[G]_{\alpha}$ is a $\Ql$-algebra with basis $(\Theta_g \mid g \in G)$ satisfying $\Theta_g\Theta_h = \alpha(g,h)\Theta_{gh}$. For each $g,x \in G$ there exists $c_{\alpha}(g,x) \in \Ql$ such that $\Theta_g\Theta_x\Theta_g^{-1} = c_{\alpha}(g,x)\Theta_{gxg^{-1}}$. The resulting map $c_{\alpha} : G \times G \to \Ql^{\times}$ makes the pair $(G,c_{\alpha})$ into a projective $G$-set with $G$ acting on itself by conjugation.
\end{pa}

%

\begin{pa}
We will denote by $\Class_{\alpha}(G)$ the space $\Fun_G(G,c_{\alpha})$ defined as in \cref{pa:G-space-funcs}. If $M \in \lmod{\Ql[G]_{\alpha}}$ then $\chi_M : G \to \Ql$, defined by $\chi_M(g) = \Tr(\Theta_g \mid M)$, is the \emph{$\alpha$-character} of $G$ afforded by $M$. The set $\Irr_{\alpha}(G)$ of irreducible $\alpha$-characters, afforded by simple $\Ql[G]_{\alpha}$-modules, is a basis of $\Class_{\alpha}(G)$. If $\alpha \in Z^2(G,\Ql)$ is \emph{unital}, in the sense that $\alpha(g,x) \in \Ql^{\times}$ is a root of unity for all $g,x \in G$, then it is an orthonormal basis. Taking $\alpha(g,x) = 1$ we get the usual space of class functions $\Class(G)$ with basis $\Irr(G)$ given by the irreducible characters.
\end{pa}

\begin{pa}
Let $H \leqslant G$ be a subgroup. Then denoting again by $\alpha$ the restriction $\alpha|_{H\times H} \in Z^2(H,\Ql)$ we have a projective $H$-set $(H,c_{\alpha})$ with $H$ acting on itself by conjugation. Applying \cref{pa:ind-map} to the natural inclusion map $\delta : H \to G$ we get induction $\Ind_H^G = \delta_*$ and restriction maps $\Res_H^G = \delta^*$. The following is a straightforward calculation.
\end{pa}

\begin{lem}\label{lem:char-ind-module}
If $\chi \in \Class_{\alpha}(H)$ is the $\alpha$-character afforded by the $\Ql[H]_{\alpha}$-module $M$ then $\Ind_H^G(\chi)$ is the $\alpha$-character afforded by the $\Ql[G]_{\alpha}$-module $\Ql[G]_{\alpha} \otimes_{\Ql[H]_{\alpha}} M$.
\end{lem}

\begin{pa}
Assume $N \lhd G$ is a normal subgroup. We have a (right) action of $G$ on $\Class_{\alpha}(N)$ given by $f^g(n) = c_{\alpha}(g,n)f({}^gn)$ for all $f \in \Class_{\alpha}(N)$, $g \in G$, and $n \in N$. If $M$ is a $\Ql[N]_{\alpha}$-module and $g \in G$ then we denote by $M_g$ the $\Ql[N]_{\alpha}$-module that is equal to $M$ as a $\Ql$-vector space but where the action is given by
\begin{equation*}
\Theta_n \star m = \Theta_g\Theta_n\Theta_g^{-1}m = c_{\alpha}(g,n)\Theta_{gng^{-1}}m.
\end{equation*}
Clearly $\chi_{M_g} = \chi_M^g$ so the action of $G$ permutes $\Irr_{\alpha}(N)$. In the following discussion we will use that standard results from Clifford Theory hold for this action, see \cite[Props.~1-3]{conlon:1964:twisted-group-algebras}.
\end{pa}

\begin{pa}
Assume now that $\phi \in \Aut(G)$ is an automorphism of $G$. Then we will denote by $G\sd\phi = G \rtimes \langle \phi \rangle$ the semidirect product of $G$ with the cyclic group $\langle \phi \rangle \leqslant \Aut(G)$ defined such that $\phi g\phi^{-1} = \phi(g)$ for any $g \in G$. We denote by $G.\phi$ the coset $\{g\phi \mid g \in G\} \subseteq G\sd\phi$, which is a $G$-set under conjugation. In particular, restricting $c_{\alpha}$ to $G.\phi$ we get a projective $G$-set $(G.\phi,c_{\alpha})$ and we denote by $\Class_{\alpha}(G.\phi)$ the space $\Fun_G(G.\phi,c_{\alpha})$.
\end{pa}

\begin{pa}\label{pa:ortho-basis-coset-funcs}
The natural inclusion map $\delta : G.\phi \to G\sd\phi$ is a projective $G$-map and we obtain, as in \cref{pa:ind-map}, corresponding induction and restriction maps $\Ind_{G.\phi}^{G\sd\phi} = \delta_*$ and $\Res_{G.\phi}^{G\sd\phi} = \delta^*$. Let us denote by $\Irr_{\alpha}(G\sd\phi \downarrow G) \subseteq \Irr_{\alpha}(G\sd\phi)$ those irreducible characters whose restriction to $G$ is irreducible. Then we define the irreducible $\alpha$-characters of the coset $G.\phi$ to be the elements of the set
\begin{equation*}
\Irr_{\alpha}(G.\phi) := \{\Res_{G.\phi}^{G\sd\phi}(\widetilde{\eta}) \mid \widetilde{\eta} \in \Irr_{\alpha}(G\sd\phi \downarrow G)\} \subseteq \Class_{\alpha}(G.\phi).
\end{equation*}
If $\alpha$ is unital then one gets an orthonormal basis for $\Class_{\alpha}(G.\phi)$ by picking for each $\eta \in \Irr_{\alpha}(G)^{\phi}$ an extension $\widetilde{\eta} \in \Irr_{\alpha}(G\sd\phi)$ and taking $\Res_G^{G\sd\phi}(\widetilde{\eta})$.
\end{pa}

\begin{rem}
To avoid cumbersome notation we will identify any irreducible character $\widetilde{\eta} \in \Irr_{\alpha}(G\sd\phi \downarrow G)$ with its restriction $\Res_{G.\phi}^{G\sd\phi}(\widetilde{\eta})$.
\end{rem}

\begin{pa}
Assume $H \leqslant G$ is a subgroup and $g \in G$ is an element such that $\boldsymbol{\iota}_g\phi(H) = H$. Then we have $\boldsymbol{\iota}_g\phi \in \Aut(H)$. We will denote by $H\sd g\phi$ the group $H\sd\boldsymbol{\iota}_g\phi$ and by $H.g\phi$ the coset $H.\boldsymbol{\iota}_g\phi$. Note that we have a surjective homomorphism $\gamma_g : G\sd g\phi \to G\sd \phi$, defined by $(x,(\boldsymbol{\iota}_g\phi)^i) \mapsto (xg\phi(g)\cdots\phi^{i-1}(g),\phi^i)$, which restricts to a bijective $G$-map $G.g\phi \to G.\phi$. If $\alpha \in Z^2(G\sd g\phi,\Ql)$ denotes the $2$-cocycle defined by $\alpha(x,y) = \alpha(\gamma_g(x),\gamma_g(y))$ then $\gamma_g$ is a projective $G$-map and $(\gamma_g)^* : \Class_{\alpha}(G.\phi) \to \Class_{\alpha}(G.g\phi)$ is an isometry. Moreover, we have
\begin{equation*}
(\gamma_g)^*(\Irr_{\alpha}(G\sd\phi\downarrow G)) \subseteq \Irr_{\alpha}(G\sd g\phi\downarrow G) \qquad\text{and} \qquad (\gamma_g)^*(\Irr_{\alpha}(G.\phi)) \subseteq \Irr_{\alpha}(G.g\phi).
\end{equation*}
The restriction of $\gamma_g$ defines an injective $H$-map $\gamma_g : H.g\phi \to G.\phi$. As in \cref{pa:G-set-funcs}, we set $\Ind_{H.g\phi}^{G.\phi} = (\gamma_g)_*$ and $\Res_{H.g\phi}^{G.\phi} = (\gamma_g)^*$. Whilst elementary the following observations concerning induction will form an important ingredient later on.
\end{pa}

\begin{lem}\label{lem:coset-ind-reg-ind}
Assume $\alpha \in Z^2(G\sd\phi,\Ql)$ is unital and $\widetilde{\chi} \in \Irr_{\alpha}(H\sd g\phi \downarrow H)$ and $\widetilde{\rho} \in \Irr_{\alpha}(G\sd\phi\downarrow G)$ have irreducible restrictions $\chi = \Res_H^{H\sd g\phi}(\widetilde{\chi}) \in \Irr_{\alpha}(H)$ and $\rho = \Res_G^{G\sd\phi}(\widetilde{\rho}) \in \Irr_{\alpha}(G)$. Then
\begin{equation*}
|\langle \widetilde{\rho},\Ind_{H.g\phi}^{G.\phi}(\widetilde{\chi})\rangle_{G.\phi}| \leqslant \langle \rho,\Ind_H^G(\chi) \rangle_G.
\end{equation*}
In particular, if $\langle \widetilde{\rho},\Ind_{H.g\phi}^{G.\phi}(\widetilde{\chi})\rangle_{G.\phi} \neq 0$ then $\langle \rho,\Ind_H^G(\chi)\rangle_G \neq 0$.
\end{lem}

\begin{proof}
Let $\widetilde{\psi} = (\gamma_g)^*(\widetilde{\rho}) \in \Irr_{\alpha}(G\sd g\phi)$. Then as $(\gamma_g)^* : \Class_{\alpha}(G.g\phi) \to \Class_{\alpha}(G.\phi)$ is an isometry and $(\gamma_g)^*\circ\Ind_{H.g\phi}^{G.\phi} = \Ind_{H.g\phi}^{G.g\phi}$ we have
\begin{equation*}
\langle \widetilde{\rho},\Ind_{H.g\phi}^{G.\phi}(\widetilde{\chi})\rangle_{G.\phi} = \langle \widetilde{\psi},\Ind_{H.g\phi}^{G.g\phi}(\widetilde{\chi})\rangle_{G.g\phi}.
\end{equation*}
As we identify $\widetilde{\psi}$ with its restriction $\Res_{G.g\phi}^{G\sd g\phi}(\widetilde{\psi})$ we have by Frobenius reciprocity and transitivity of induction, see \cref{prop:fun-G-sets-bumper}, that
\begin{equation*}
\langle \widetilde{\psi},\Ind_{H.g\phi}^{G.g\phi}(\widetilde{\chi})\rangle_{G.g\phi} = \langle \widetilde{\psi}, \Ind_{H\sd g\phi}^{G\sd g\phi}(\Ind_{H.g\phi}^{H\sd g\phi}(\widetilde{\chi}))\rangle_{G\sd g\phi}.
\end{equation*}
The same argument used in \cite[1.3]{bonnafe:2006:sln} shows that $\Ind_{H.g\phi}^{H\sd g\phi}(\widetilde{\chi}) = \sum_{\lambda \in \Irr(H\sd g\phi/H)} \overline{\lambda(g\phi)}(\lambda\otimes\widetilde{\chi})$ so we get that
\begin{equation*}
\langle \widetilde{\rho},\Ind_{H.g\phi}^{G.\phi}(\widetilde{\chi})\rangle_{G.\phi} = \sum_{\lambda \in \Irr(H\sd g\phi/H)} \overline{\lambda(g\phi)}\langle \widetilde{\psi}, \Ind_{H\sd g\phi}^{G\sd g\phi}(\lambda \otimes \widetilde{\chi})\rangle_{G\sd g\phi}.
\end{equation*}

As $\rho = \Res_G^{G\sd g\phi}(\widetilde{\psi})$ an identical argument yields that
\begin{equation*}
\langle \rho,\Ind_H^G(\chi)\rangle_G = \langle \widetilde{\psi},\Ind_H^{G\sd g\phi}(\chi)\rangle_{G\sd g\phi} = \langle \widetilde{\psi},\Ind_{H\sd g\phi}^{G\sd g\phi}(\Ind_H^{H:g\phi}(\chi))\rangle_{G\sd g\phi}.
\end{equation*}
Moreover, a standard consequence of Clifford's Theorem applied to $\Ind_H^{H:g\phi}$ gives us
\begin{equation*}
\langle \rho,\Ind_H^G(\chi)\rangle_G = \sum_{\lambda \in \Irr(H\sd g\phi/H)} \langle \widetilde{\psi},\Ind_{H\sd g\phi}^{G\sd g\phi}(\lambda \otimes \widetilde{\chi}) \rangle_{G\sd g\phi},
\end{equation*}
see \cite[Cor.~6.17]{isaacs:2006:character-theory-of-finite-groups}. Putting things together we get the desired statement as $|\overline{\lambda(g\phi)}| = 1$.
\end{proof}

\begin{cor}\label{lem:triv-bound-inner-prod}
Assume $H_i \leqslant G$ is a subgroup and $g_i \in G$ is an element such that $\boldsymbol{\iota}_{g_i}\phi \in \Aut(H_i)$ for $i \in \{1,2\}$. If $\widetilde{\chi}_i \in \Irr_{\alpha}(H_i\sd g_i\phi \downarrow H_i)$ is an irreducible character then
\begin{equation*}
|\langle \Ind_{H_1.g_1\phi}^{G.\phi}(\widetilde{\chi}_1),\Ind_{H_2.g_2\phi}^{G.\phi}(\widetilde{\chi}_2) \rangle_{G.\phi}| \leqslant |G|.
\end{equation*}
\end{cor}

\begin{proof}
Decomposing in an orthonormal basis of $\Class_{\alpha}(G.\phi)$ we have
\begin{equation*}
\langle \Ind_{H_1.g_1\phi}^{G.\phi}(\widetilde{\chi}_1),\Ind_{H_2.g_2\phi}^{G.\phi}(\widetilde{\chi}_2) \rangle = \sum_{\eta \in \Irr(G)^{\phi}} \langle \widetilde{\eta}, \Ind_{H_1.g_1\phi}^{G.\phi}(\widetilde{\chi}_1)\rangle \cdot \overline{\langle \widetilde{\eta}, \Ind_{H_2.g_2\phi}^{G.\phi}(\widetilde{\chi}_2)\rangle}
\end{equation*}
where $\widetilde{\eta} \in \Irr(G\sd\phi)$ is a fixed extension of $\eta$. If $\chi_i = \Res_{H_i.g_i\phi}^{G.\phi}(\widetilde{\chi}_i)$ is the irreducible restriction of $\widetilde{\chi}_i$ then it follows from \cref{lem:coset-ind-reg-ind} that
\begin{align*}
|\langle \Ind_{H_1.g_1\phi}^{G.\phi}(\widetilde{\chi}_1),\Ind_{H_2.g_2\phi}^{G.\phi}(\widetilde{\chi}_2) \rangle| &\leqslant \sum_{\eta \in \Irr(G)^{\phi}} \langle \eta, \Ind_{H_1}^G(\chi_1)\rangle \cdot \overline{\langle \eta, \Ind_{H_2}^G(\chi_2)\rangle}\\
&\leqslant \langle \Ind_{H_1}^G(\chi_1), \Ind_{H_2}^G(\chi_2)\rangle.
\end{align*}
The statement now follows from the fact that $\Ind_{H_i}^G(\chi_i)$ is necessarily a summand of the character of the regular representation of $\Ql[G]_{\alpha}$.
\end{proof}

\section{Character Sheaves}\label{sec:char-sheaves}
\begin{pa}\label{pa:bound-der-equi-cat}
Assume $\bX$ is a variety equipped with an algebraic action of a connected algebraic group $\bH$. We will denote by $\mathscr{D}_{\bH}(\bX)$ the $\bH$-equivariant bounded derived category of $\Ql$-constructible sheaves on $\bX$, as defined in \cite{bernstein-lunts:1994:equivariant-sheaves}. Let $\bX'$ be another variety equipped with an algebraic action of a connected algebraic group $\bH'$. If $\phi : \bX \to \bX'$ is an equivariant morphism then we obtain (derived) functors $\phi^*,\phi^! : \mathscr{D}_{\bH'}(\bX') \to \mathscr{D}_{\bH}(\bX)$ and $\phi_*,\phi_! : \mathscr{D}_{\bH}(\bX) \to \mathscr{D}_{\bH'}(\bX')$. We will denote by $\mathscr{M}_{\bH}(\bX) \subseteq \mathscr{D}_{\bH}(\bX)$ the full subcategory of $\bH$-equivariant perverse sheaves on $\bX$.
\end{pa}

\begin{pa}
If $\bX = \bH$ in \cref{pa:bound-der-equi-cat} then we will implicitly assume that $\bH$ acts on $\bX$ by conjugation. Recall the pair $(\bG,F)$ fixed in \cref{pa:notation}. In \cite[2.10]{lusztig:1985:character-sheaves} Lusztig has defined the notion of a character sheaf which is a simple object in the category $\mathscr{M}_{\bG}(\bG)$. We will denote by $\CS(\bG) \subseteq \mathscr{M}_{\bG}(\bG)$ the full subcategory whose objects are all finite direct sums of character sheaves. We reserve the term \emph{character sheaf} for a simple object of $\CS(\bG)$.
\end{pa}

\begin{pa}
We will say that a complex $A \in \mathscr{D}_{\bG}(\bG)$ is \emph{$F$-stable} if there exists an isomorphism $\phi : F^*A \to A$. For such an $F$-stable complex $A \in \mathscr{D}_{\bG}(\bG)$ and isomorphism $\phi$ we denote by $\mathcal{X}_{A,\phi} \in \Class(\bG^F)$ the characteristic function of the complex. The map $A \mapsto F^*A$ defines a permutation of the isomorphism classes $\Irr(\mathscr{M}_{\bG}(\bG))$ and $\Irr(\CS(\bG))$ and we denote by $\Irr(\mathscr{M}_{\bG}(\bG))^F \subseteq \Irr(\mathscr{M}_{\bG}(\bG))$ and $\Irr(\CS(\bG))^F \subseteq \Irr(\CS(\bG))$ the corresponding sets of fixed points. The following result of Lusztig is shown under some mild restrictions in \cite[\S25]{lusztig:1985:character-sheaves} and is established in full generality in \cite{lusztig:2012:on-the-cleanness}.
\end{pa}

\begin{thm}[Lusztig]\label{thm:char-funcs-as-basis}
There exists a family of isomorphisms $(\phi_A : F^*A \to A)_{[A] \in \Irr(\CS(\bG))^F}$ such that the set of characteristic functions $\{\mathcal{X}_{A,\phi_A} \mid [A] \in \Irr(\CS(\bG))^F\}$ forms an orthonormal basis of $\Class(\bG^F)$. Each $\phi_A$ is defined uniquely up to multiplication by a root of unity.
\end{thm}

\begin{rem}
If $A \in \mathscr{D}_{\bG}(\bG)$ is an $F$-stable complex then we will often write $\mathcal{X}_A$ instead of $\mathcal{X}_{A,\phi}$ with an isomorphism $\phi : F^*A \to A$ implicitly chosen. If $A \in \CS(\bG)$ is a character sheaf then we will always assume that $\phi$ is chosen to be part of such a family as in \cref{thm:char-funcs-as-basis}.
\end{rem}

\section{Parabolic Induction}\label{sec:parabolic-ind-char-sheaves}
\begin{pa}\label{pa:ind-char-sheaves}
Let $\bP \leqslant \bG$ be a parabolic subgroup with unipotent radical $\bU \leqslant \bP$ and Levi complement $\bL \leqslant \bP$. Associated to this data we have a parabolic induction functor $\ind_{\bL \subseteq \bP}^{\bG} : \mathscr{D}_{\bL}(\bL) \to \mathscr{D}_{\bG}(\bG)$ defined as follows, see \cite[\S4.1]{lusztig:1985:character-sheaves}. First, consider the diagram
\begin{equation*}
\begin{tikzcd}
\bL & \hat{X} \arrow{l}[swap]{\pi}\arrow{r}{\sigma} & \tilde{X} \arrow{r}{\tau} & \bG
\end{tikzcd}
\end{equation*}
where we have
\begin{gather*}
\begin{aligned}
\hat{X} &= \{(g,h) \in \bG \times \bG \mid h^{-1}gh \in \bP\} \quad&\quad \tilde{X} &= \{(g,h\bP) \in \bG \times (\bG/\bP) \mid h^{-1}gh \in \bP\}
\end{aligned}\\
\begin{aligned}
\pi(g,h) &= \bar{\pi}_{\bP}(h^{-1}gh) \quad&\quad \sigma(g,h) &= (g,h\bP) \quad&\quad \tau(g,h\bP) &=g
\end{aligned}
\end{gather*}
and $\bar{\pi}_{\bP} : \bP \to \bL$ is the canonical projection map. Here $\hat{X}$ is a variety where $\bG$ acts on the left via $x\cdot (g,h) = (xgx^{-1},xh)$ and $\bP$ acts on the right by $(g,h)\cdot y = (g,hy)$. Hence, we have an action of $\bG\times\bP^{\mathrm{op}}$ on $\hat{X}$ where $\bP^{\mathrm{op}}$ is the opposite group of $\bP$. Moreover, $\tilde{X}$ is the quotient of $\hat{X}$ by the right $\bP$-action. All the morphisms are equivariant with respect to the stated actions.
\end{pa}

\begin{pa}
The fibres of $\pi$ have dimension $\dim\bG+\dim\bU$ and we set $\tilde{\pi} := \pi^*[\dim\bG+\dim\bU]$. Similarly, the fibres of $\sigma$ have dimension $\dim\bP$ and we set $\tilde{\sigma} := \sigma^*[\dim\bP]$. If $A \in \mathscr{D}_{\bL}(\bL)$ then there exists a canonical complex $D \in \mathscr{D}_{\bG\times\bP^{\mathrm{op}}}(\tilde{X})$ such that $\tilde{\pi}A = \tilde{\sigma}D$. We then define $\ind_{\bL\subseteq\bP}^{\bG}(A) := \tau_!D$. If $f \in \Hom_{\mathscr{D}_{\bL}(\bL)}(A,B)$ is a morphism then we get a morphism $\ind_{\bL\subseteq\bP}^{\bG}(f) \in \Hom_{\mathscr{D}_{\bG}(\bG)}(\ind_{\bL\subseteq\bP}^{\bG}(A),\ind_{\bL\subseteq\bP}^{\bG}(B))$ as follows. We have an induced morphism $\tilde{\pi}f : \tilde{\pi}A \to \tilde{\pi}B$. As $\sigma$ is smooth with connected fibres we have $\widetilde{\sigma}$ is a fully faithful functor so there exists a unique morphism $f'$ such that $\widetilde{\pi}f = \widetilde{\sigma}f'$. We then have $\ind_{\bL\subseteq\bP}^{\bG}(f) = \tau_!f'$. We will need the following concerning induction, which is noted in the proof of \cite[15.2]{lusztig:1985:character-sheaves}. It is a straightforward consequence of the function-sheaf dictionary, see \cite[Thm.~12.1]{kiehl-weissauer:2001:weil-conjectures}.
\end{pa}

\begin{lem}\label{lem:ind-split-Levi}
Assume $\bP$, $\bL$, and $A$ are $F$-stable. Then for any isomorphism $\phi : F^*A \to A$ we have
\begin{equation*}
R_{\bL\subseteq\bP}^{\bG}(\mathcal{X}_{A,\phi}) = \mathcal{X}_{\ind_{\bL\subseteq\bP}^{\bG}(A),\ind_{\bL\subseteq\bP}^{\bG}(\phi)}.
\end{equation*}
\end{lem}

\begin{pa}\label{pa:algebra-auto}
We assume now that $K \in \mathscr{M}_{\bG}(\bG)$ is a semisimple object with endomorphism algebra $\mathcal{A} = \End_{\mathscr{M}_{\bG}(\bG)}(K)$. We then have a functor $\mathfrak{F}_K = \Hom_{\mathscr{M}_{\bG}(\bG)}(-,K) : \mathscr{M}_{\bG}(\bG) \to \lmod{\mathcal{A}}$ as in \cref{sec:decomp-semisimple-obj}. Assume $K$ is $F$-stable, i.e., there exists an isomorphism $\phi : F^*K \to K$. Then we have an algebra automorphism $\sigma : \mathcal{A} \to \mathcal{A}$ given by $\sigma(\theta) = \phi\circ F^*\theta \circ \phi^{-1}$. For any $\mathcal{A}$-module $E \in \lmod{\mathcal{A}}$ we denote by $E_{\sigma}$ the module equal to $E$ as a vector space but with the action defined by $a\cdot e = \sigma^{-1}(a)\cdot e$. The following is a straightforward consequence of \cref{lem:hom-to-end-alg}.
\end{pa}

\begin{lem}\label{lem:iso-correspondence}
For any summand $A \mid K$ we have an isomorphism $\rho : \mathfrak{F}_K(A)_{\sigma} \to \mathfrak{F}_K(F^*A)$ of $\mathcal{A}$-modules defined by $\rho(f) = \phi \circ F^*(f)$. In particular, the assignment $\phi \mapsto \rho^{-1}\circ\mathfrak{F}_K(\phi)$ defines a bijection between the isomorphisms $F^*A \to A$ in $\mathscr{M}_{\bG}(\bG)$ and the $\mathcal{A}$-module isomorphisms $\mathfrak{F}_K(A) \to \mathfrak{F}_K(A)_{\sigma}$.
\end{lem}

\begin{pa}\label{pa:iso-twisted-module}
Note that as $K$ is $F$-stable we have the assignment $A \mapsto F^*A$ defines a permutation of $\Irr(\mathscr{M}_{\bG}(\bG) \mid K)$, c.f., \cref{pa:notation}, and we denote by $\Irr(\mathscr{M}_{\bG}(\bG) \mid K)^F \subseteq \Irr(\mathscr{M}_{\bG}(\bG) \mid K)$ the set of fixed points under this permutation. If $\phi_A : F^*A \to A$ is an isomorphism in $\mathscr{M}_{\bG}(\bG)$ then we have a corresponding isomorphism denoted by $\sigma_A = \mathfrak{F}_K(\phi_A^{-1})\circ\rho : \mathfrak{F}_K(A)_{\sigma} \to \mathfrak{F}_K(A)$. By definition we have $\sigma_A(f) = \phi\circ F^*(f)\circ \phi_A^{-1}$ for any $f \in \mathfrak{F}_K(A)$. An identical argument to that used in \cite[10.4.2]{lusztig:1985:character-sheaves} yields the following.
\end{pa}

\begin{lem}\label{lem:char-func-semisimple}
Assume $K \in \mathscr{M}_{\bG}(\bG)$ is an $F$-stable semisimple perverse sheaf, as above. Then we have
\begin{equation*}
\mathcal{X}_{K,\phi} = \sum_{A \in \Irr(\mathscr{M}_{\bG}(\bG) \mid K)^F} \Tr(\sigma_A,\mathfrak{F}_K(A))\mathcal{X}_{A,\phi_A}.
\end{equation*}
Moreover, we have the trace $\Tr(\sigma_A,\mathfrak{F}_K(A))$ is non-zero for any $A \in \Irr(\mathscr{M}_{\bG}(\bG) \mid K)^F$ as $\sigma_A$ is an automorphism of the vector space $\mathfrak{F}_K(A)$.
\end{lem}

\section{Inducing Cuspidal Character Sheaves}\label{sec:ind-cusp-char-sheaves}

\begin{pa}\label{pa:cusp-pairs}
We denote by $\Cusp(\bG)$ the set of triples $(\bL,\Sigma,[\mathscr{E}])$ where: $\bL \leqslant \bG$ is a Levi subgroup, $\Sigma \subseteq \bL$ is the inverse image of an isolated conjugacy class under the natural projection map $\bL \to \bL/Z^{\circ}(\bL)$, and $[\mathscr{E}]$ is the isomorphism class of an irreducible $\bL$-equivariant local system $\mathscr{E}$ on $\Sigma$ such that
\begin{equation*}
\mathscr{E}^{\sharp} := \IC(\overline{\Sigma},\mathscr{E})[\dim\Sigma] \in \CS(\bL)
\end{equation*}
is a cuspidal character sheaf. For brevity we will write $(\bL,\Sigma,\mathscr{E})$ for the tuple $(\bL,\Sigma,[\mathscr{E}])$ with it implicitly assumed that $\mathscr{E}$ is taken up to isomorphism. Here we use the notation of Lusztig \cite[1.4]{lusztig:1990:green-functions-and-character-sheaves}, except we have shifted the intersection cohomology complex to make it perverse.
\end{pa}

\begin{rem}
We note that if $\bM \leqslant \bG$ is a Levi subgroup of $\bG$ then we have a natural inclusion $\Cusp(\bM) \subseteq \Cusp(\bG)$ of cuspidal triples.
\end{rem}

\begin{pa}\label{pa:cuspidal-induce}
To each tuple $(\bL,\Sigma,\mathscr{E}) \in \Cusp(\bG)$ we associate a perverse sheaf $K_{\bL,\Sigma,\mathscr{E}}^{\bG} \in \mathscr{M}_{\bG}(\bG)$ as follows. Let $\Sigma_{\reg} = \{g \in \Sigma \mid C_{\bG}^{\circ}(g_{\sss}) \leqslant \bL \}$, where $g_{\sss}$ denotes the semisimple part of $g$, and set $Y = \bigcup_{g \in \bG} g\Sigma_{\reg}g^{-1}$. Then we have a diagram
\begin{equation*}
\begin{tikzcd}
\bL & \hat{Y} \arrow{l}[swap]{\alpha}\arrow{r}{\beta} & \tilde{Y} \arrow{r}{\gamma} & Y
\end{tikzcd}
\end{equation*}
where
\begin{gather*}
\begin{aligned}
\hat{Y} &= \{(g,h) \in \bG \times \bG \mid h^{-1}gh \in \Sigma\}, \quad&\quad \tilde{Y} &= \{(g,h\bL) \in \bG \times (\bG/\bL) \mid h^{-1}gh \in \bL\},
\end{aligned}\\
\begin{aligned}
\alpha(g,h) &= h^{-1}gh, \quad&\quad \beta(g,h) &= (g,h\bL), \quad&\quad \gamma(g,h\bL) &=g.
\end{aligned}
\end{gather*}
As for parabolic induction we have $\hat{Y}$ is a variety where $\bG$ acts on the left via $x\cdot(g,h) = (xgx^{-1},xh)$ and $\bL$ acts on the right via $(g,h)\cdot l = (g,hl)$. We have $\tilde{Y}$ is the quotient of $Y$ by the right $\bL$-action. Now, there exists a unique $\bG$-equivariant local system $\tilde{\mathscr{E}}$ on $\tilde{Y}$ such that $\alpha^*\mathscr{E} = \beta^*\tilde{\mathscr{E}}$. The $\bG$-equivariant local system $\gamma_*\tilde{\mathscr{E}}$ is semisimple, see \cite[Prop.~3.5]{lusztig:1984:intersection-cohomology-complexes}, and we set $K_{\bL,\Sigma,\mathscr{E}}^{\bG} = \IC(\overline{Y},\gamma_*\tilde{\mathscr{E}})[\dim Y]$ viewed as a perverse sheaf on $\bG$ via extension by $0$.
\end{pa}

\begin{thm}[{}{Lusztig, \cite[II, 4.3.2, 8.2.3]{lusztig:1985:character-sheaves}}]\label{thm:Lusztig-iso}
The perverse sheaves $K_{\bL,\Sigma,\mathscr{E}}^{\bG}$ and $\ind_{\bL\subseteq\bP}^{\bG}(\mathscr{E}^{\sharp})$ are semisimple and canonically isomorphic. Moreover, all their simple summands are character sheaves.
\end{thm}

\begin{pa}\label{pa:induced-complex}
We have a (left) action of $\bG$ on $\Cusp(\bG)$ defined by
\begin{equation*}
g\cdot (\bL,\Sigma,\mathscr{E}) = (\boldsymbol{\iota}_g(\bL),\boldsymbol{\iota}_g(\Sigma),(\boldsymbol{\iota}_g^*)^{-1}\mathscr{E}).
\end{equation*}
The orbit of $(\bL,\Sigma,\mathscr{E})$ under this action will be denoted by $[\bL,\Sigma,\mathscr{E}]$ and the set of all orbits will be denoted by $[\Cusp(\bG)]$. By \cite[7.1.12, 7.6]{lusztig:1985:character-sheaves} and \cref{thm:Lusztig-iso} we have a decomposition
\begin{equation*}
\Irr(\CS(\bG)) = \bigsqcup_{[\bL,\Sigma,\mathscr{E}] \in [\Cusp(\bG)]} \Irr(\CS(\bG) \mid K_{\bL,\Sigma,\mathscr{E}}^{\bG}).
\end{equation*}
\end{pa}

\begin{pa}
The Frobenius endomorphism $F$ defines a permutation of the set $\Cusp(\bG)$ via the map $(\bL,\Sigma,\mathscr{E}) \mapsto (F^{-1}(\bL),F^{-1}(\Sigma),F^*\mathscr{E})$. We denote by $\Cusp(\bG)^F$ the set of fixed points. The permutation $\Cusp(\bG) \to \Cusp(\bG)$ induced by $F$ also induces a permutation of the $\bG$-orbits $[\Cusp(\bG)]$ and we again denote by $[\Cusp(\bG)]^F$ the set of fixed points.
\end{pa}

\begin{pa}
A standard argument using the Lang--Steinberg theorem shows that the canonical map $\Cusp(\bG)^F \to [\Cusp(\bG)]^F$ is surjective, see \cite[10.5]{lusztig:1985:character-sheaves}. Moreover, we have a decomposition
\begin{equation*}
\Irr(\CS(\bG))^F = \bigsqcup_{[\bL,\Sigma,\mathscr{E}] \in [\Cusp(\bG)]^F} \Irr(\CS(\bG) \mid K_{\bL,\Sigma,\mathscr{E}}^{\bG})^F.
\end{equation*}
If $\Class(\bG^F \mid [\bL,\Sigma,\mathscr{E}]) \subseteq \Class(\bG^F)$ denotes the subspace spanned by the characteristic functions of the character sheaves contained in $\Irr(\CS(\bG) \mid K_{\bL,\Sigma,\mathscr{E}}^{\bG})^F$ then we have a corresponding direct sum decomposition
\begin{equation*}
\Class(\bG^F) = \bigoplus_{[\bL,\Sigma,\mathscr{E}] \in [\Cusp(\bG)]^F} \Class(\bG^F \mid [\bL,\Sigma,\mathscr{E}]).
\end{equation*}
\end{pa}

\begin{pa}
Now assume $(\bL,\Sigma,\mathscr{E}) \in \Cusp(\bG)^F$ is $F$-fixed. Then, by definition, there exists an isomorphism $\varphi : F^*\mathscr{E} \to \mathscr{E}$ (recall that $\mathscr{E}$ is taken up to isomorphism). We will denote by $\Cusp(\bG,F)$ the set of tuples $(\bL,\Sigma,\mathscr{E},\varphi)$ with $(\bL,\Sigma,\mathscr{E}) \in \Cusp(\bG)^F$ and $\varphi : F^*\mathscr{E} \to \mathscr{E}$ an isomorphism; such tuples are called induction data in \cite[1.8]{lusztig:1990:green-functions-and-character-sheaves}. The isomorphism $\varphi$ naturally extends to an isomorphism $\varphi^{\sharp} : F^*\mathscr{E}^{\sharp} \to \mathscr{E}^{\sharp}$ by the functoriality of intersection cohomology which, in turn, extends to an isomorphism $\phi : F^*K_{\bL,\Sigma,\mathscr{E}}^{\bG} \to K_{\bL,\Sigma,\mathscr{E}}^{\bG}$, see \cite[8.2]{lusztig:1985:character-sheaves}.
\end{pa}

\begin{pa}\label{pa:property-DL-ind}
In what follows we will need the following powerful generalization of \cref{lem:ind-split-Levi} to the case of non-split Levi subgroups:
\begin{enumerate}[leftmargin=1.5cm]
	\item[($\mathcal{R}_{\bG,F}$)] For any tuple $(\bL,\Sigma,\mathscr{E},\varphi) \in \Cusp(\bG,F)$ we have $R_{\bL}^{\bG}(\mathcal{X}_{\mathscr{E}^{\sharp},\varphi^{\sharp}}) = \mathcal{X}_{K_{\bL,\Sigma,\mathscr{E}}^{\bG},\phi}$.
\end{enumerate}
Assume $p$ is good for $\bG$. If $q$ is sufficiently large then it is a result of Lusztig that ($\mathcal{R}_{\bG,F}$) holds, \cite[Prop.~9.2]{lusztig:1990:green-functions-and-character-sheaves}. If $Z(\bG)$ is connected then Shoji has shown that ($\mathcal{R}_{\bG,F}$) always holds, see \cite[Thm.~4.2]{shoji:1996:on-the-computation}.
\end{pa}

\section{Harish-Chandra Parameterization of Character Sheaves}\label{sec:harish-chandra-param-char-sheaves}

\begin{pa}
We will assume we have a fixed triple $(\bL,\Sigma,\mathscr{E}) \in \Cusp(\bG)$ and a parabolic subgroup $\bP\leqslant\bG$ with Levi complement $\bL$. We denote by $N_{\bG}(\bL,\Sigma,\mathscr{E}) \leqslant \bG$ the stabiliser of the triple under the $\bG$-action described in \cref{pa:cusp-pairs}. We clearly have $\bL \leqslant N_{\bG}(\bL,\Sigma,\mathscr{E}) \leqslant N_{\bG}(\bL)$ and so we obtain a subgroup $W_{\bG}(\bL,\Sigma,\mathscr{E}) := N_{\bG}(\bL,\Sigma,\mathscr{E})/\bL$ of the relative Weyl group $W_{\bG}(\bL) := N_{\bG}(\bL)/\bL$. For brevity we set $\mathcal{W} = W_{\bG}(\bL,\Sigma,\mathscr{E})$ and $\mathcal{N} = N_{\bG}(\bL,\Sigma,\mathscr{E})$ for the rest of this section.
\end{pa}

\begin{pa}\label{pa:Lusztig-basis}
Let us denote by $\mathcal{A}_{\bL,\Sigma,\mathscr{E}}^{\bG}$ the endomorphism algebra $\End_{\mathscr{M}_{\bG}(\bG)}(K_{\bL,\Sigma,\mathscr{E}}^{\bG})$, which is a finite dimensional $\Ql$-algebra. To each element $w \in \mathcal{W}$ Lusztig has defined an invertible endomorphism $\Theta_w^{\bG} \in \mathcal{A}_{\bL,\Sigma,\mathscr{E}}^{\bG}$, as follows. Using the notation of \cref{pa:cuspidal-induce} let $\gamma_w : \tilde{Y} \to \tilde{Y}$ be defined by $\gamma_w(g,x\bL) = (g,xn_w^{-1}\bL)$ where $n_w \in \mathcal{N}$ is a fixed representative of $w \in \mathcal{W}$. There exists an isomorphism $\theta_w : \mathscr{E} \to \boldsymbol{\iota}_{n_w}^*\mathscr{E}$ of $\bL$-equivariant local systems and this extends uniquely to an isomorphism $\tilde{\theta}_w^{\bG} : \tilde{\mathscr{E}} \to \gamma_w^*\tilde{\mathscr{E}}$ satisfying $\alpha^*\theta_w^{\bG} = \beta^*\tilde{\theta}_w^{\bG}$, see \cite[Prop.~3.5]{lusztig:1984:intersection-cohomology-complexes}. As $\gamma_*\gamma_w^* = \gamma_*$ we have $\gamma_*\tilde{\theta}_w^{\bG}$ is an automorphism of $\gamma_*\tilde{\mathscr{E}}$. Applying the fully faithful functor $\IC(\overline{Y},-)[\dim Y]$ to $\gamma_*\tilde{\theta}_w^{\bG}$ we get $\Theta_w^{\bG} \in \mathcal{A}_{\bL,\Sigma,\mathscr{E}}^{\bG}$.
\end{pa}

\begin{thm}[{}{Lusztig, \cite[10.2]{lusztig:1985:character-sheaves}}]\label{thm:twisted-algebra}
There exists a $2$-cocycle $\alpha \in Z^2(\mathcal{W},\Ql)$ such that $\Theta_w^{\bG}\Theta_{w'}^{\bG} = \alpha(w,w')\Theta_{ww'}^{\bG}$. Hence, we have an isomorphism of $\Ql$-algebras $\Theta_{\bL,\Sigma,\mathscr{E}}^{\bG} : \mathcal{A}_{\bL,\Sigma,\mathscr{E}}^{\bG} \to \Ql[\mathcal{W}]_{\alpha}$.
\end{thm}

\begin{rem}
In most cases it is known that $\alpha$ can be assumed to be trivial but this has yet to be established in general. For instance, after \cite[Thm.~9.2]{lusztig:1984:characters-of-reductive-groups-over-finite-fields}, this is the case when $\Sigma$ contains a unipotent element.
\end{rem}

\begin{pa}
One gets an isomorphism $\theta_w^{\sharp} : \mathscr{E}^{\sharp} \to \boldsymbol{\iota}_{n_w}^*\mathscr{E}^{\sharp}$ by applying the fully faithful functor $\IC(\overline{\Sigma},-)[\dim\Sigma]$ to $\theta_w$. Finally, applying induction we get an isomorphism
\begin{equation*}
\ind_{\bL\subseteq\bP}^{\bG}(\theta_w^{\sharp}) : \ind_{\bL\subseteq\bP}^{\bG}(\mathscr{E}^{\sharp}) \to \ind_{\bL\subseteq\bP}^{\bG}(\boldsymbol{\iota}_{n_w}^*\mathscr{E}^{\sharp}).
\end{equation*}
By \cite[Lem.~3.9]{taylor:2014:evaluating-characteristic-functions} we have $\ind_{\bL\subseteq\bP}^{\bG}(\boldsymbol{\iota}_{n_w}^*\mathscr{E}^{\sharp}) = \boldsymbol{\iota}_{n_w}^*\ind_{\bL\subseteq{}^{n_w}\bP}^{\bG}(\mathscr{E}^{\sharp})$ and by the $\bG$-equivariance we may identify $\boldsymbol{\iota}_{n_w}^*\ind_{\bL\subseteq{}^{n_w}\bP}^{\bG}(\mathscr{E}^{\sharp})$ with $\ind_{\bL\subseteq{}^{n_w}\bP}^{\bG}(\mathscr{E}^{\sharp})$. Hence, we may also think of the above as an isomorphism
\begin{equation*}
\ind_{\bL\subseteq\bP}^{\bG}(\theta_w^{\sharp}) : \ind_{\bL\subseteq\bP}^{\bG}(\mathscr{E}^{\sharp}) \to \ind_{\bL\subseteq {}^{n_w}\bP}^{\bG}(\mathscr{E}^{\sharp}).
\end{equation*}
We will need the following compatibility between these two constructions.
\end{pa}

\begin{prop}\label{lem:compatible-endomorphisms}
We have a commutative diagram
\begin{equation*}
\begin{tikzcd}[column sep=6em]
\ind_{\bL\subseteq\bP}^{\bG}(\mathscr{E}^{\sharp}) \arrow[r,"\ind_{\bL\subseteq\bP}^{\bG}(\theta_w^{\sharp})"]\arrow[d] & \ind_{\bL\subseteq{}^{n_w}\bP}^{\bG}(\mathscr{E}^{\sharp}) \arrow[d]\\
K_{\bL,\Sigma,\mathscr{E}}^{\bG} \arrow[r,"\Theta_w^{\bG}"] & K_{\bL,\Sigma,\mathscr{E}}^{\bG}
\end{tikzcd}
\end{equation*}
where the vertical arrows are the canonical isomorphisms of \cref{thm:Lusztig-iso}.
\end{prop}

\begin{proof}
We freely use the notation of \cref{pa:ind-char-sheaves,pa:cuspidal-induce}. Let us denote by $\imath_Y : Y \to \bG$ and $\jmath : \tau^{-1}(Y) \to \tilde{X}$ the natural inclusion morphisms. The endomorphism $\Theta_w^{\bG}$ is uniquely determined by the property that $\imath_Y^*\Theta_w^{\bG}$ is $\gamma_*\tilde{\theta}_w^{\bG}[\dim Y]$. Hence it suffices to show that the morphism corresponding to $\ind_{\bL\subseteq\bP}^{\bG}(\theta_w^{\sharp})$ has this property.

We denote by $f$ the unique morphism satisfying $\tilde{\pi}\theta_w^{\sharp} = \tilde{\sigma}f$. Then, by definition, we have $\ind_{\bL\subseteq\bP}^{\bG}(\theta_w^{\sharp}) = \tau_!f$. Lusztig has shown that we have an isomorphism $\kappa : \tilde{Y} \to \tau^{-1}(Y)$ defined by $\kappa(g,h\bL) = (g,h\bP)$, see \cite[4.3(c)]{lusztig:1984:intersection-cohomology-complexes}. The isomorphism between $\ind_{\bL\subseteq\bP}^{\bG}(\mathscr{E}^{\sharp})$ and $K_{\bL,\Sigma,\mathscr{E}}^{\bG}$ gives an isomorphism
\begin{equation*}
\imath_Y^*\ind_{\bL}^{\bG}(\mathscr{E}^{\sharp}) = \imath_Y^*\tau_!D \to \imath_Y^*K_{\bL,\Sigma,\mathscr{E}}^{\bG} = \gamma_*\tilde{\mathscr{E}},
\end{equation*}
\cite[4.4, 4.5]{lusztig:1984:intersection-cohomology-complexes}. Now $\imath_Y^*\tau_!f = \tau_!\jmath^*f$. Moreover, by \cite[4.3(b)]{lusztig:1984:intersection-cohomology-complexes} we have $\gamma$ is proper so by smooth base change $\tau_! = \gamma_* \kappa^*$ because $\gamma = \tau\circ \kappa$. Under the above isomorphism $\imath_Y^*\tau_!f$ corresponds to $\gamma_*\kappa^*\jmath^*f$. Hence, it suffices to show that $\kappa^*\jmath^*f = \tilde{\theta}_w^{\bG}[\dim Y]$.

Recall from \cite[4.3(a)]{lusztig:1984:intersection-cohomology-complexes} that we have an equality
\begin{equation*}
\dim\bG + \dim\bU - \dim\bP = \dim Y - \dim\Sigma.
\end{equation*}
As $\tilde{\sigma}f = \tilde{\pi}\theta_w^{\sharp}$ we get that $\sigma^*f = \pi^*\theta_w^{\sharp}[\dim Y - \dim \Sigma]$. If $\imath_{\hat{Y}} : \hat{Y} \to \hat{X}$ is the natural inclusion map then $\jmath\circ\kappa\circ\beta = \sigma\circ\imath_{\hat{Y}}$ which implies that
\begin{equation*}
\beta^*\kappa^*\jmath^*f = \imath_{\hat{Y}}^*\sigma^*f = \imath_{\hat{Y}}^*\pi^*\hat{\theta}_w^{\bL}[\dim Y - \dim \Sigma].
\end{equation*}
The image of the morphism $\pi\circ \imath_{\hat{Y}}$ is contained in $\Sigma$ so agrees with $\alpha$. Hence, if $\imath_{\Sigma} : \Sigma \to \bL$ is the natural inclusion map then we have $\imath_{\hat{Y}}^*\pi^*\theta_w^{\sharp}$ coincides with $\alpha^*\imath_{\Sigma}^*\theta_w^{\sharp}$ but by definition $\imath_{\Sigma}^*\theta_w^{\sharp} = \theta_w[\dim\Sigma]$. Putting things together we get that $\beta^*\kappa^*\jmath^*f = \alpha^*\theta_w[\dim Y]$ which implies that $\kappa^*\jmath^*f = \tilde{\theta}_w[\dim Y]$, as desired.
\end{proof}

\begin{pa}\label{pa:HC-param-char-sheaves}
Recall from \cref{sec:decomp-semisimple-obj} that we have a functor
\begin{equation*}
\mathfrak{F}_{\bL,\Sigma,\mathscr{E}}^{\bG} = \Hom_{\mathscr{M}_{\bG}(\bG)}(-,K_{\bL,\Sigma,\mathscr{E}}^{\bG}) : \mathscr{M}_{\bG}(\bG) \to \lmod{\mathcal{A}_{\bL,\Sigma,\mathscr{E}}^{\bG}}.
\end{equation*}
By \cref{thm:twisted-algebra} we may view this as a functor $\mathscr{M}_{\bG}(\bG) \to \lmod{\Ql[\mathcal{W}]_{\alpha}}$. For any \emph{$\alpha$-character} $\eta \in \Class_{\alpha}(\mathcal{W})$ we will denote by $K_{\eta}^{\bG} \in \CS(\bG,[\bL,\Sigma,\mathscr{E}])$ a perverse sheaf such that $\mathfrak{F}_{\bL,\Sigma,\mathscr{E}}^{\bG}(K_{\eta}^{\bG})$ affords the character $\eta$. This yields a bijection
\begin{equation*}
\Irr_{\alpha}(\mathcal{W}) \to \Irr(\CS(\bG) \mid K_{\bL,\Sigma,\mathscr{E}}^{\bG})
\end{equation*}
as in \cref{lem:bij-iso-classes}.
\end{pa}

\begin{pa}
Now assume $(\bL,\Sigma,\mathscr{E},\varphi) \in \Cusp(\bG,F)$. Recall that the isomorphism $\varphi : F^*\mathscr{E} \to \mathscr{E}$ induces an isomorphism $\phi : F^*K_{\bL,\Sigma,\mathscr{E}}^{\bG} \to K_{\bL,\Sigma,\mathscr{E}}^{\bG}$. Using $\phi$ we obtain a corresponding algebra automorphism $\sigma : \mathcal{A}_{\bL,\Sigma,\mathscr{E}}^{\bG} \to \mathcal{A}_{\bL,\Sigma,\mathscr{E}}^{\bG}$ as in \cref{pa:algebra-auto}. It is readily checked that under the isomorphism of \cref{thm:Lusztig-iso} the isomorphism $\phi$ corresponds to the isomorphism
\begin{equation*}
\ind_{\bL \subseteq \bP}^{\bG}(\varphi^{\sharp}) : F^*\ind_{\bL \subseteq F(\bP)}^{\bG}(\mathscr{E}^{\sharp}) \to \ind_{\bL \subseteq \bP}^{\bG}(\mathscr{E}^{\sharp}).
\end{equation*}
Now it is straightforward to check that for some function $s : \mathcal{W} \to \Ql$ we have $\sigma^{-1}(\Theta_w^{\bG}) = s(w)\Theta_{F(w)}^{\bG}$ for all $w \in \mathcal{W}$.
\end{pa}

\begin{lem}\label{lem:ext-2-cocycle}
Assume $F \in \Aut(\mathcal{W})$ has order $m > 0$. Then the $2$-cocycle $\alpha \in Z^2(\mathcal{W},\Ql)$ extends to a $2$-cocycle $\alpha \in Z^2(\mathcal{W}\sd F,\Ql)$ by setting
\begin{equation*}
\alpha(wF^i,xF^j) = s(x)\cdots s(F^{i-1}(x))\alpha(w,F^i(x))
\end{equation*}
for all $0 \leqslant i,j < m$ and $w,x \in \mathcal{W}$. In particular, $\alpha(wF,x) = s(x)\alpha(w,F(x))$ and $\alpha(w,xF) = \alpha(w,x)$.
\end{lem}

\begin{proof}
As $\sigma$, hence $\sigma^{-1}$, is an algebra automorphism we have
\begin{equation*}
\alpha(w,w')s(ww') = \alpha(F(w),F(w'))s(w)s(w')
\end{equation*}
for all $w,w' \in \mathcal{W}$. Using this it is a straightforward calculation to show that $\alpha$ satisfies the $2$-cocycle condition.
%
\end{proof}

\begin{assumption}
From this point forward we assume that the basis $(\Theta_w^{\bG} \mid w \in \mathcal{W})$ of $\mathcal{A}_{\bL,\Sigma,\mathscr{E}}^{\bG}$ is chosen such that the $2$-cocycle $\alpha \in Z^2(\mathcal{W}\sd F,\Ql)$ is unital.
\end{assumption}

\begin{pa}
It follows from \cref{lem:iso-correspondence} that the bijection in \cref{pa:HC-param-char-sheaves} restricts to a bijection
\begin{equation*}
\Irr_{\alpha}(\mathcal{W})^F \to \Irr(\CS(\bG) \mid K_{\bL,\Sigma,\mathscr{E}}^{\bG})^F.
\end{equation*}
Assume $\eta \in \Irr_{\alpha}(\mathcal{W})^F$ is an $F$-stable $\alpha$-character and let $A = K_{\eta}^{\bG}$ then there exists an isomorphism $\phi_A : F^*A \to A$. As in \cref{pa:iso-twisted-module} we obtain a corresponding isomorphism $\sigma_A : \mathfrak{F}_{\bL,\Sigma,\mathscr{E}}^{\bG}(A)_{\sigma} \to \mathfrak{F}_{\bL,\Sigma,\mathscr{E}}^{\bG}(A)$ . We make $\mathfrak{F}_{\bL,\Sigma,\mathscr{E}}^{\bG}(A)$ into a $\Ql[\mathcal{W}\sd F]_{\alpha}$-module by setting $\Theta_F\cdot v = \sigma_A^{-1}(v)$. This module then affords an irreducible $\alpha$-character $\widetilde{\eta} \in \Irr_{\alpha}(\mathcal{W}\sd F)$ that extends $\eta$. We will denote by $\phi_{\widetilde{\eta}} : F^*K_{\eta}^{\bG} \to K_{\eta}^{\bG}$ an isomorphism such that $\mathfrak{F}_{\bL,\Sigma,\mathscr{E}}^{\bG}(K_{\eta}^{\bG})$ affords the character $\widetilde{\eta}$ when viewed as a $\Ql[\mathcal{W}\sd F]_{\alpha}$-module.
\end{pa}

\begin{rem}\label{rem:choice-of-isos}
We may, and will, assume that $\phi_{\widetilde{\eta}}$ is part of a family of isomorphisms as in \cref{thm:char-funcs-as-basis}.
\end{rem}

\begin{pa}
For each element $w \in \mathcal{W}$ recall our choice of representative $n_w \in \mathcal{N}$ from \cref{pa:Lusztig-basis}. In addition let us choose an element $g_w \in \bG$ such that $g_w^{-1}F(g_w) = n_w^{-1}$; such an element exists by the Lang--Steinberg theorem. We then obtain a new triple $(\bL_w,\Sigma_w,\mathscr{E}_w) \in \Cusp(\bG)^F$ where
\begin{equation*}
\bL_w = \boldsymbol{\iota}_{g_w}(\bL) \qquad \Sigma_w = \boldsymbol{\iota}_{g_w}(\Sigma) \qquad \mathscr{E}_w := \boldsymbol{\iota}_{g_w^{-1}}^*\mathscr{E}.
\end{equation*}
A standard result shows that the map $w \mapsto [\bL_w,\Sigma_w,\mathscr{E}_w]$ gives a bijection between the $F$-conjugacy classes of $\mathcal{W}$ and the orbits of $\bG^F$ acting on $\Cusp(\bG)^F$. Following \cite[10.6]{lusztig:1985:character-sheaves} we get an isomorphism $\varphi_w : F^*\mathscr{E}_w \to \mathscr{E}_w$ by setting $\varphi_w = (\boldsymbol{\iota}_{g_w^{-1}}^*\varphi) \circ (F^*\boldsymbol{\iota}_{g_w^{-1}}^*\theta_w)$.
\end{pa}

\begin{pa}
Now assume $w,x,z \in \mathcal{W}$ satisfy $w^{-1} = zx^{-1}F(z^{-1})$. We can find an element $n \in \mathcal{N}$, representing $z \in \mathcal{W}$, such that $n_w^{-1} = nn_x^{-1}F(n^{-1})$. We then have $g = g_wng_x^{-1} \in \bG^F$ satisfies $\boldsymbol{\iota}_g(\bL_x) = \bL_w$ and $\boldsymbol{\iota}_g(\bP_x) = \bP_w$. Furthermore, there is an isomorphism $\psi : \boldsymbol{\iota}_g^*\mathscr{E}_w \to \mathscr{E}_x$. Now we get two isomorphisms $F^*\boldsymbol{\iota}_g^*\mathscr{E}_w \to \mathscr{E}_x$ given by $\varphi_x \circ F^*\psi$ and $\psi\circ \boldsymbol{\iota}_g^*\varphi_w$. As $\mathscr{E}_w$ and $\mathscr{E}_x$ are irreducible there exists a scalar $\omega = \omega(n,x,w) \in \Ql^{\times}$ such that $\psi\circ \boldsymbol{\iota}_g^*\varphi_w = \omega(n,x,w)(\varphi_x \circ F^*\psi)$. The scalar $\omega$ does not depend on the isomorphism $\psi$.
\end{pa}

\begin{pa}
By functoriality we get isomorphisms between ICs that satisfy the same identity, namely $(\psi^{\sharp}\circ \boldsymbol{\iota}_g^*\varphi_w^{\sharp}) = \omega(\varphi_x^{\sharp} \circ F^*\psi^{\sharp})$. So for any $l \in \bL_x^F$ we get that
\begin{equation*}
\Tr(\varphi_w^{\sharp},\mathscr{H}_{\boldsymbol{\iota}_g(l)}^i(\mathscr{E}_w^{\sharp})) = \omega\Tr(\psi^{\sharp -1}\circ\varphi_x^{\sharp} \circ \psi^{\sharp},\mathscr{H}_{\boldsymbol{\iota}_g(l)}^i(\mathscr{E}_w^{\sharp})) = \omega\Tr(\varphi_x^{\sharp},\mathscr{H}_l^i(\mathscr{E}_x^{\sharp})),
\end{equation*}
From the definition we thus have $\mathcal{X}_{\mathscr{E}_w^{\sharp},\varphi_w^{\sharp}} \circ \boldsymbol{\iota}_g = \omega(g,x,w)\mathcal{X}_{\mathscr{E}_x^{\sharp},\varphi_x^{\sharp}}$. In a special case, one can deduce the following from the brief remarks in \cite[10.6.4, 10.6.5]{lusztig:1985:character-sheaves}.
\end{pa}

\begin{prop}\label{prop:conj-scalar-calc}
Recall that $w^{-1} = zx^{-1}F(z^{-1})$ and $n \in \mathcal{N}$ is a representative of $z \in \mathcal{W}$ satisfying $n_w^{-1} = nn_x^{-1}F(n^{-1})$. We have
\begin{equation*}
\omega(n,x,w) = \alpha(w,w^{-1})\alpha(x,x^{-1})^{-1}c_{\alpha}(z,x^{-1}F).
\end{equation*}
In particular, $\omega(n,x,w)$ is independent of $n$.
\end{prop}

\begin{proof}
For brevity we set $\omega = \omega(n,x,w)$. Applying the induction construction we get an equality
\begin{equation}\label{eq:omega-identity}
\ind_{\bL_x \subseteq \bP_x}^{\bG}(\psi^{\sharp}) \circ \ind_{\bL_x \subseteq \bP_x}^{\bG}(\boldsymbol{\iota}_g^*\varphi_w^{\sharp}) = \omega(\ind_{\bL_x \subseteq \bP_x}^{\bG}(\varphi_x^{\sharp}) \circ \ind_{\bL_x \subseteq \bP_x}^{\bG}(F^*\psi^{\sharp})).
\end{equation}
We may freely choose $\psi$ to perform this calculation. If $\theta : \boldsymbol{\iota}_n^*\mathscr{E} \to \mathscr{E}$ is an isomorphism then we may take $\psi = \boldsymbol{\iota}_{g_x^{-1}}^*\theta$. As $n \in \mathcal{N}$ represents $z \in \mathcal{W}$ we have $n = ln_z$ for some $l \in \bL$. Using the $\bL$-equivariance we can identify $\mathscr{E}$ and $\boldsymbol{\iota}_l^*\mathscr{E}$. We assume that, under this identification, $\theta$ is identified with the isomorphism $\theta_z^{-1} : \boldsymbol{\iota}_{n_z}^*\mathscr{E} \to \mathscr{E}$.

Using the $\bG$-equivariance to make identifications the left hand side of \cref{eq:omega-identity} becomes
\begin{equation*}
\ind_{\bL \subseteq {}^{n_z}\bP}^{\bG}(\theta_z^{-1}) \circ \ind_{\bL \subseteq F({}^{n_z}\bP)}^{\bG}(\varphi^{\sharp}) \circ \ind_{\bL \subseteq {}^{n_w^{-1}}F({}^{n_z}\bP)}^{\bG}(F^*\theta_w)
\end{equation*}
and the right hand side becomes
\begin{equation*}
\ind_{\bL \subseteq F(\bP)}^{\bG}(\varphi) \circ \ind_{\bL \subseteq {}^{n_x^{-1}}F(\bP)}^{\bG}(F^*\theta_x) \circ \ind_{\bL \subseteq {}^{n_zn_x^{-1}}F(\bP)}(F^*\theta_z^{-1})
\end{equation*}
After \cref{lem:compatible-endomorphisms} we thus have an equality
\begin{equation*}
\Theta_z^{-1} \circ \phi \circ F^*\Theta_w = \omega(g,x,w)(\phi \circ F^*\Theta_x \circ F^*\Theta_z^{-1})
\end{equation*}
in the algebra $\mathcal{A}_{\bL,\Sigma,\mathscr{E}}^{\bG}$. Using the definition of $\sigma$ we get an equality $\sigma^{-1}(\Theta_z^{-1})\Theta_w = \omega(\Theta_x\Theta_z^{-1})$ in the algebra $\Ql[\mathcal{W}]_{\alpha}$.

As $\sigma^{-1}(\Theta_z^{-1}) = \Theta_F\Theta_z^{-1}\Theta_F^{-1}$ in the algebra $\Ql[\mathcal{W}\sd F]_{\alpha}$ we get an equality
\begin{equation*}
\Theta_z\Theta_x^{-1}\Theta_F\Theta_z^{-1} = \omega \Theta_w^{-1}\Theta_F
\end{equation*}
in $\Ql[\mathcal{W}\sd F]_{\alpha}$. The statement now follows from the definitions, in particular using the fact that $\alpha(y,F) = \alpha(y,1) = \alpha(1,1)$ for any $y \in \mathcal{W}$.
\end{proof}

\begin{prop}\label{prop:bumper-DL-ind-char-sheaves}
Fix a tuple $(\bL,\Sigma,\mathscr{E},\varphi) \in \Cusp(\bM,F)$. We define a $\Ql$-linear map $\mathscr{R}^{\bG}_{\bL,\Sigma,\mathscr{E},\varphi} : \Class_{\alpha}(W_{\bG}(\bL,\Sigma,\mathscr{E}).F) \to \Class(\bG^F \mid [\bL,\Sigma,\mathscr{E}])$ by setting
\begin{equation*}
\mathscr{R}^{\bG}_{\bL,\Sigma,\mathscr{E},\varphi}(f) = \frac{1}{|W_{\bG}(\bL,\Sigma,\mathscr{E})|}\sum_{w \in W_{\bG}(\bL,\Sigma,\mathscr{E})} \alpha(w,w^{-1})^{-1}f(w^{-1}F)R_{\bL_w}^{\bG}(\mathcal{X}_{\mathscr{E}_w^{\sharp},\varphi_w^{\sharp}}).
\end{equation*}
Assume ($\mathcal{R}_{\bG,F}$) holds, and recall our assumption that the $2$-cocyle $\alpha \in Z^2(W_{\bG}(\bL,\Sigma,\mathscr{E})\sd F,\Ql)$ is unital. Then the following hold:
\begin{enumerate}
	\item for any irreducible $\alpha$-character $\widetilde{\eta} \in \Irr_{\alpha}(W_{\bG}(\bL,\Sigma,\mathscr{E})\sd F)$ we have
\begin{equation*}
\mathscr{R}^{\bG}_{\bL,\Sigma,\mathscr{E},\varphi}(\widetilde{\eta}) = \mathcal{X}_{K_{\eta},\phi_{\widetilde{\eta}}},
\end{equation*}
	\item $\mathscr{R}^{\bG}_{\bL,\Sigma,\mathscr{E},\varphi}$ is an isometry onto its image,
	\item we have an equality of linear maps $\Class_{\alpha}(W_{\bM}(\bL,\Sigma,\mathscr{E}).F) \to \Class(\bG^F \mid [\bL,\Sigma,\mathscr{E}])$
\begin{equation*}
R_{\bM}^{\bG} \circ \mathscr{R}^{\bM}_{\bL,\Sigma,\mathscr{E},\varphi} = \mathscr{R}^{\bG}_{\bL,\Sigma,\mathscr{E},\varphi} \circ \Ind_{W_{\bM}(\bL,\Sigma,\mathscr{E}).F}^{W_{\bG}(\bL,\Sigma,\mathscr{E}).F}.
\end{equation*}
\end{enumerate}
\end{prop}

\begin{proof}
(a). This follows from \cite[10.4.5]{lusztig:1985:character-sheaves} by applying the identity $\alpha(w^{-1},F) = \alpha(w^{-1},1) = \alpha(1,1)$ and ($\mathcal{R}_{\bG,F}$) to each of the tuples $(\bL_w,\Sigma_w,\mathscr{E}_w,\varphi_w)$.

(b). This is clear as $\mathcal{R}^{\bG}_{\bL,\Sigma,\mathscr{E},\varphi}$ maps an orthonormal basis of $\Class_{\alpha}(W_{\bG}(\bL,\Sigma,\mathscr{E}).F)$, c.f., \cref{pa:ortho-basis-coset-funcs}, onto an orthonormal basis of $\Class(\bG^F \mid [\bL,\Sigma,\mathscr{E}])$, c.f., \cref{thm:char-funcs-as-basis,rem:choice-of-isos}.

(c). It is simple to check that \cref{prop:conj-scalar-calc} implies $\mathscr{R}^{\bM}_{\bL,\Sigma,\mathscr{E},\varphi}(\pi_{x^{-1}F}) = R_{\bL_x}^{\bG}(\mathcal{X}_{\mathscr{E}_x^{\sharp},\varphi_x^{\sharp}})$, where $\pi_{x^{-1}F}$ is the function defined in \cref{pa:orb-funcs}. Using (ii) of \cref{prop:fun-G-sets-bumper} we can now argue as in \cite[Prop.~15.7]{digne-michel:1991:representations-of-finite-groups-of-lie-type}.
\end{proof}

\section{A Comparison Theorem for Character Sheaves}\label{sec:comparison-thms}
\begin{pa}
Assume $\bM \leqslant \bG$ is a Levi subgroup of $\bG$ and $(\bL,\Sigma,\mathscr{E}) \in \Cusp(\bM)$ is a cuspidal triple. We choose parabolic subgroups $\bP \leqslant \bQ \leqslant \bG$ with Levi complements $\bL \leqslant \bP$ and $\bM \leqslant \bQ$. In this section we note that an analogue of Howlett--Lehrer's Comparison Theorem, see \cite[5.9]{howlett-lehrer:1983:representations-of-generic-algebras}, holds for induction of character sheaves. Namely $\ind_{\bM\subseteq\bQ}^{\bG}$ corresponds to the usual induction $\Ind_{W_{\bM}(\bL,\Sigma,\mathscr{E})}^{W_{\bG}(\bL,\Sigma,\mathscr{E})}$ under the correspondence described in \cref{pa:HC-param-char-sheaves}. In the case where $\Sigma$ contains a unipotent element this was pointed out by Lusztig in \cite[2.5]{lusztig:1986:on-the-character-values}.
\end{pa}

\begin{pa}
In general if $A \in \mathscr{M}_{\bM}(\bM)$ is an $\bM$-equivariant perverse sheaf then the complex $\ind_{\bM\subseteq\bQ}^{\bG}(A) \in \mathscr{D}_{\bG}(\bG)$ need not necessarily be perverse. However, if $A \in \CS(\bM)$ then Lusztig has shown that $\ind_{\bM\subseteq\bQ}^{\bG}(A) \in \CS(\bG)$, see \cite[4.4]{lusztig:1985:character-sheaves}. Hence, we have $\ind_{\bM\subseteq\bQ}^{\bG}$ defines a $\Ql$-linear functor $\CS(\bM) \to \CS(\bG)$ between abelian categories. In particular, we can appeal to the formalism discussed in \cref{sec:decomp-semisimple-obj}.
\end{pa}

\begin{pa}
Let $\bP \leqslant \bQ \leqslant \bG$ be parabolic subgroups of $\bG$ with Levi complements $\bL \leqslant \bP$ and $\bM \leqslant \bQ$. We set $\tilde{\bP} = \bM \cap \bP$ which is a parabolic subgroup of $\bM$ with Levi complement $\bL$. By \cite[4.2, 4.4]{lusztig:1985:character-sheaves} we have an isomorphism
\begin{equation}\label{eq:iso-ind}
\ind_{\bM \subseteq \bQ}^{\bG}(\ind_{\bL \subseteq \tilde{\bP}}^{\bM}(\mathscr{E}^{\sharp})) \cong \ind_{\bL \subseteq \bP}^{\bG}(\mathscr{E}^{\sharp}).
\end{equation}
After \cref{thm:Lusztig-iso} this yields an isomorphism
\begin{equation*}
\ind_{\bM\subseteq\bQ}^{\bG}(K_{\bL,\Sigma,\mathscr{E}}^{\bM}) \cong K_{\bL,\Sigma,\mathscr{E}}^{\bG}
\end{equation*}
and we obtain an algebra homomorphism $\ind_{\bM\subseteq\bQ}^{\bG} : \mathcal{A}_{\bL,\Sigma,\mathscr{E}}^{\bM} \to \mathcal{A}_{\bL,\Sigma,\mathscr{E}}^{\bG}$. We want to show the following compatibility.
\end{pa}

\begin{prop}\label{lem:end-alg-identifications}
We have a commutative diagram
\begin{equation*}
\begin{tikzcd}
\mathcal{A}_{\bL,\Sigma,\mathscr{E}}^{\bM} \arrow[r,"\ind_{\bM\subseteq\bQ}^{\bG}"]\arrow[d] & \mathcal{A}_{\bL,\Sigma,\mathscr{E}}^{\bG} \arrow[d]\\
\Ql[W_{\bM}(\bL,\Sigma,\mathscr{E})]_{\alpha} \arrow[r,hook] & \Ql[W_{\bG}(\bL,\Sigma,\mathscr{E})]_{\alpha}
\end{tikzcd}
\end{equation*}
where the bottom arrow is the canonical inclusion of algebras.
\end{prop}

\begin{proof}
For any object $\square$ introduced in \cref{pa:ind-char-sheaves} we affix subscripts and superscripts, such as $\square_{\bL\subseteq\bP}^{\bG}$, to indicate that it is defined with regards to $\ind_{\bL \subseteq \bP}^{\bG}$. Let $D \in \mathscr{M}(\tilde{X}_{\bL \subseteq \bP}^{\bG})$ and $D' \in \mathscr{M}(\tilde{X}_{\bL \subseteq \tilde{\bP}}^{\bM})$ be the canonical perverse sheaves satisfying $\tilde{\pi}_{\bL \subseteq \bP}^{\bG}\mathscr{E}^{\sharp} = \tilde{\sigma}_{\bL \subseteq \bP}^{\bG}D$ and $\tilde{\pi}_{\bL \subseteq \tilde{\bP}}^{\bM}\mathscr{E}^{\sharp} = \tilde{\sigma}_{\bL \subseteq \tilde{\bP}}^{\bM}D'$ respectively. By definition we have
\begin{equation*}
\ind_{\bL \subseteq \bP}^{\bG}(\mathscr{E}^{\sharp}) = (\tau_{\bL \subseteq \bP}^{\bG})_!D \qquad \ind_{\bL \subseteq \tilde{\bP}}^{\bM}(\mathscr{E}^{\sharp}) = (\tau_{\bL \subseteq \tilde{\bP}}^{\bM})_!D'.
\end{equation*}

We have a well-defined equivariant morphism $\lambda : \tilde{X}_{\bL \subseteq \bP}^{\bG} \to \tilde{X}_{\bM \subseteq \bQ}^{\bG}$ given by $\lambda(g,h\bP) = (g,h\bQ)$ because $\bP \leqslant \bQ$. In \cite[4.2(b)]{lusztig:1985:character-sheaves} Lusztig shows that $\tilde{\pi}_{\bM\subseteq\bQ}^{\bG}(\tau_{\bL\subseteq\tilde{\bP}}^{\bM})_!D' = \tilde{\sigma}_{\bM\subseteq\bQ}^{\bG}\lambda_!D$ so, again by definition, we have
\begin{equation*}
\ind_{\bM \subseteq \bQ}^{\bG}(\ind_{\bL \subseteq \tilde{\bP}}^{\bM}(\mathscr{E}^{\sharp})) = (\tau_{\bM\subseteq\bQ}^{\bG})_!\lambda_!D = (\tau_{\bM\subseteq\bQ}^{\bG}\circ \lambda)_!D = (\tau_{\bL\subseteq\bP}^{\bG})_!D = \ind_{\bL \subseteq \bP}^{\bG}(\mathscr{E}^{\sharp})
\end{equation*}
because $\tau_{\bL\subseteq\bP}^{\bG} = \tau_{\bM\subseteq\bQ}^{\bG}\circ \lambda$. Note that $\tilde{\sigma}_{\bM\subseteq\bQ}^{\bG}\lambda_!$ denotes the composition $\tilde{\sigma}_{\bM\subseteq\bQ}^{\bG} \circ \lambda_!$ and $\tilde{\pi}_{\bM\subseteq\bQ}^{\bG}(\tau_{\bL\subseteq\tilde{\bP}}^{\bM})_!$ denotes the compactly supported pushforward of $\tilde{\pi}_{\bM\subseteq\bQ}^{\bG}(\tau_{\bL\subseteq\tilde{\bP}}^{\bM})$.

Now let $\theta$ be an invertible endomorphism of $\mathscr{E}^{\sharp}$. If $f$ and $f'$ denote the unique morphisms satisfying $\tilde{\pi}_{\bL\subseteq\bP}^{\bG}\theta = \tilde{\sigma}_{\bL\subseteq\bP}^{\bG}f$ and $\tilde{\pi}_{\bL\subseteq\tilde{\bP}}^{\bM}\theta = \tilde{\sigma}_{\bL\subseteq\tilde{\bP}}^{\bM}f'$ then by definition
\begin{equation*}
\ind_{\bL \subseteq \bP}^{\bG}(\theta) = (\tau_{\bL\subseteq\bP}^{\bG})_!f \qquad \ind_{\bL \subseteq \tilde{\bP}}^{\bM}(\theta) = (\tau_{\bL\subseteq\tilde{\bP}}^{\bM})_!f'.
\end{equation*}
An identical argument to that used by Lusztig shows that $\tilde{\pi}_{\bM\subseteq\bQ}^{\bG}(\tau_{\bL\subseteq\tilde{\bP}}^{\bM})_!f' = \tilde{\sigma}_{\bM\subseteq\bQ}^{\bG}\lambda_!f$ and identically we get that $\ind_{\bM \subseteq \bQ}^{\bG}(\ind_{\bL \subseteq \tilde{\bP}}^{\bM}(\theta)) = \ind_{\bL \subseteq \bP}^{\bG}(\theta)$. The statement now follows from \cref{lem:compatible-endomorphisms}.
\end{proof}

\begin{cor}\label{prop:comparison-theorem}
Assume $\bM \leqslant \bG$ is a Levi subgroup and $(\bL,\Sigma,\mathscr{E}) \in \Cusp(\bM)$ is a cuspidal triple. For any irreducible $\alpha$-characters $\mu \in \Irr_{\alpha}(W_{\bM}(\bL,\Sigma,\mathscr{E}))$ and $\lambda \in \Irr_{\alpha}(W_{\bG}(\bL,\Sigma,\mathscr{E}))$ we have
\begin{equation*}
\dim_{\Ql} \Hom_{\mathscr{M}_{\bG}(\bG)}(K_{\lambda}^{\bG},\ind_{\bM\subseteq \bQ}^{\bG}(K_{\mu}^{\bM})) = \langle \lambda, \Ind_{W_{\bM}(\bL,\Sigma,\mathscr{E})}^{W_{\bG}(\bL,\Sigma,\mathscr{E})}(\mu)\rangle_{W_{\bG}(\bL,\Sigma,\mathscr{E})}.
\end{equation*}
In particular, we have $K_{\lambda}^{\bG} \mid \ind_{\bM\subseteq \bQ}^{\bG}(K_{\mu}^{\bM})$ if and only if $\langle \lambda, \Ind_{W_{\bM}(\bL,\Sigma,\mathscr{E})}^{W_{\bG}(\bL,\Sigma,\mathscr{E})}(\mu)\rangle_{W_{\bG}(\bL,\Sigma,\mathscr{E})} \neq 0$.
\end{cor}

\begin{proof}
By \cref{cor:decomp-ind-semisimple} we have a $\Ql$-linear isomorphism
\begin{equation*}
\Hom_{\mathscr{M}_{\bG}(\bG)}(\ind_{\bM\subseteq\bQ}^{\bG}(K_{\mu}^{\bM}),K_{\lambda}^{\bG}) \cong \Hom_{\mathcal{A}_{\bL,\Sigma,\mathscr{E}}^{\bG}}(\mathcal{A}_{\bL,\Sigma,\mathscr{E}}^{\bG} \otimes_{\mathcal{A}_{\bL,\Sigma,\mathscr{E}}^{\bM}} \mathfrak{F}_{\bL,\Sigma,\mathscr{E}}^{\bM}(K_{\mu}^{\bM}),\mathfrak{F}_{\bL,\Sigma,\mathscr{E}}^{\bG}(K_{\lambda}^{\bG})).
\end{equation*}
Using \cref{lem:end-alg-identifications,lem:char-ind-module} we see that $\mathcal{A}_{\bL,\Sigma,\mathscr{E}}^{\bG} \otimes_{\mathcal{A}_{\bL,\Sigma,\mathscr{E}}^{\bM}} \mathfrak{F}_{\bL,\Sigma,\mathscr{E}}^{\bM}(K_{\mu}^{\bM})$ affords the induced $\alpha$-character $\Ind_{W_{\bM}(\bL,\Sigma,\mathscr{E})}^{W_{\bG}(\bL,\Sigma,\mathscr{E})}(\mu)$. The statement now follows from the fact that the induced complex is semisimple.
\end{proof}

\section{The Case of a Split Levi Subgroup}\label{sec:split-levi-case}
\begin{pa}
Recall that if $p$ is a good prime for $\bG$ then to any irreducible character $\chi \in \Irr(\bG^F)$ one can associate its wave-front set $\mathcal{O}_{\chi}^* \subseteq \mathcal{U}(\bG)$, which is an $F$-stable unipotent conjugacy class in $\bG$, see \cite[11.2]{lusztig:1992:a-unipotent-support} and \cite[14.10]{taylor:2016:GGGRs-small-characteristics}. If $D_{\bG^F} : \Class(\bG^F) \to \Class(\bG^F)$ denotes Alvis--Curtis duality then we have a bijection ${}^* : \Irr(\bG^F) \to \Irr(\bG^F)$ defined by $\chi^* = \pm D_{\bG^F}(\chi)$. This follows from the fact that $D_{\bG^F}$ is an involutive isometry mapping characters to virtual characters, \cite[\S8]{digne-michel:1991:representations-of-finite-groups-of-lie-type}. By \cite[11.2]{lusztig:1992:a-unipotent-support}, see also \cite[14.15]{taylor:2016:GGGRs-small-characteristics}, the notions of wave-front set and unipotent support are related via the equality
\begin{equation}\label{eq:wave-front-unip-supp}
\mathcal{O}_{\chi^*} = \mathcal{O}_{\chi}^*.
\end{equation}
\end{pa}

\begin{lem}\label{lem:equiv-dual-statements}
Assume $p$ is a good prime for $\bG$. Then for any $F$-stable Levi subgroup $\bM \leqslant \bG$ the following statements are equivalent:
\begin{enumerate}[label=\textnormal{(\roman*)}]
	\item For any irreducible characters $\eta \in \Irr(\bM^F)$ and $\chi \in \Irr(\bG^F)$ satisfying $\langle \chi, R_{\bM}^{\bG}(\eta)\rangle_{\bG^F} \neq 0$ we have ${\mathcal{O}_{\eta}^*} \leqslant \mathcal{O}_{\chi}^*$.
	\item For any irreducible characters $\eta \in \Irr(\bM^F)$ and $\chi \in \Irr(\bG^F)$ satisfying $\langle \chi, R_{\bM}^{\bG}(\eta)\rangle_{\bG^F} \neq 0$ we have $\mathcal{O}_{\eta} \leqslant \mathcal{O}_{\chi}$.
\end{enumerate}
\end{lem}

\begin{proof}
As $p$ is a good prime for $\bG$ we have that the Mackey formula holds by \cite{bonnafe-michel:2011:mackey-formula}. Hence $R_{\bM}^{\bG} \circ D_{\bM^F} = D_{\bG^F} \circ R_{\bM}^{\bG}$ by \cite[10.13]{bonnafe:2006:sln}. As $D_{\bG^F}$ is an isometry we get that
\begin{equation*}
\langle \chi,R_{\bM}^{\bG}(\eta)\rangle_{\bG^F} = \langle D_{\bG^F}(\chi),D_{\bG^F}(R_{\bM}^{\bG}(\eta))\rangle_{\bG^F} = \pm \langle \chi^*,R_{\bM}^{\bG}(\eta^*)\rangle_{\bG^F}.
\end{equation*}
The equivalence now follows from \cref{eq:wave-front-unip-supp} and the fact that ${}^*$ is a bijection.
\end{proof}

\begin{pa}\label{pa:isotypic-morphisms}
Recall that a homomorphism of algebraic groups $\iota : \bG \to \widetilde{\bG}$ is called \emph{isotypic} if the image $\iota(\bG)$ contains the derived subgroup of $\widetilde{\bG}$ and the kernel $\Ker(\iota)$ is contained in the centre $Z(\bG)$ of $\bG$. We further assume that $\widetilde{\bG}$ is equipped with a Frobenius endomorphism $F : \widetilde{\bG} \to \widetilde{\bG}$ and $\iota$ is defined over $\mathbb{F}_q$, in the sense that $\iota\circ F = F\circ\iota$. If $\widetilde{\bG}$ is a connected reductive algebraic group then for any Levi subgroup $\bM \leqslant \bG$ of $\bG$ we have $\widetilde{\bM} = \iota(\bM)Z(\widetilde{\bG}) \leqslant \widetilde{\bG}$ is a Levi subgroup of $\widetilde{\bG}$. The assignment $\bM \mapsto \widetilde{\bM}$ is a bijection between Levi subgroups sending $F$-stable Levi subgroups to $F$-stable Levi subgroups and $(\bG,F)$-split Levi subgroups to $(\widetilde{\bG},F)$-split Levi subgroups.
\end{pa}

\begin{pa}\label{pa:reg-embedding}
We now wish to reduce checking the conditions in \cref{lem:equiv-dual-statements} to those situations covered by \cite{bezrukavnikov-liebeck-shalev-tiep:2017:character-bounds-grps-Lie-type}. For this let us recall that an isotypic morphism $\iota : \bG \to \widetilde{\bG}$ is a \emph{regular embedding} if $\widetilde{\bG}$ is a connected reductive algebraic group with $Z(\widetilde{\bG})$ connected and $\iota$ is a closed embedding. Given such a morphism we will implicitly identify $\bG$ with a subgroup of $\widetilde{\bG}$ and identify $\mathcal{U}(\bG)$ with $\mathcal{U}(\widetilde{\bG})$.
\end{pa}

\begin{prop}\label{prop:split-red-to-conn-centre}
Let $\iota : \bG \to \widetilde{\bG}$ be a regular embedding and assume $\bM \leqslant \bG$ and $\widetilde{\bM}\leqslant \widetilde{\bG}$ are corresponding $F$-stable Levi subgroups. If $\bM$ is $(\bG,F)$-split, equivalently $\widetilde{\bM}$ is $(\widetilde{\bG},F)$-split, then the conditions of  \cref{lem:equiv-dual-statements} hold for the pair $(\bM,\bG)$ if and only if they hold for the pair $(\widetilde{\bM},\widetilde{\bG})$.
\end{prop}

\begin{proof}
We first show that if (ii) of \cref{lem:equiv-dual-statements} holds for the pair $(\widetilde{\bM},\widetilde{\bG})$ then it holds for the pair $(\bM,\bG)$. Let $\eta \in \Irr(\bM^F)$ and $\chi \in \Irr(\bG^F)$ be irreducible characters satisfying $\langle \chi, R_{\bM}^{\bG}(\eta)\rangle_{\bG^F} \neq 0$. We choose an irreducible character $\widetilde{\eta} \in \Irr(\widetilde{\bM}^F)$ such that $\Res_{\bG^F}^{\widetilde{\bG}^F}(\widetilde{\eta}) = \eta + \lambda$ with $\lambda \in \Class(\bM^F)$ a \emph{character}.

According to \cite[10.10]{bonnafe:2006:sln} we have $R_{\bM}^{\bG}\circ \Res_{\bM^F}^{\widetilde{\bM}^F} = \Res_{\bG^F}^{\widetilde{\bG}^F} \circ R_{\widetilde{\bM}}^{\widetilde{\bG}}$ which implies that
\begin{equation}\label{eq:sum-of-chars}
\Res_{\bG^F}^{\widetilde{\bG}^F}(R_{\widetilde{\bM}}^{\widetilde{\bG}}(\widetilde{\eta})) = R_{\bM}^{\bG}(\eta) + R_{\bM}^{\bG}(\lambda).
\end{equation}
As $\bM$ is a $(\bG,F)$-split Levi subgroup we have $R_{\bM}^{\bG}$ is Harish-Chandra induction so the sum in \cref{eq:sum-of-chars} is a sum of characters. As $\chi$ is a constituent of $R_{\bM}^{\bG}(\eta)$ it is therefore also a constituent of $\Res_{\bG^F}^{\widetilde{\bG}^F}(R_{\widetilde{\bM}}^{\widetilde{\bG}}(\widetilde{\eta}))$. Hence we have
\begin{equation*}
\langle \chi, \Res_{\bG^F}^{\widetilde{\bG}^F}(R_{\widetilde{\bM}}^{\widetilde{\bG}}(\widetilde{\eta}))\rangle_{\bG^F} = \langle \Ind_{\bG^F}^{\widetilde{\bG}^F}(\chi), R_{\widetilde{\bM}}^{\widetilde{\bG}}(\widetilde{\eta})\rangle_{\bG^F} \neq 0.
\end{equation*}

This implies there exists an irreducible character $\widetilde{\chi} \in \Irr(\widetilde{\bG}^F \mid R_{\widetilde{\bM}}^{\widetilde{\bG}}(\widetilde{\eta}))$ such that $\chi \in \Irr(\bG^F \mid \Res_{\bG^F}^{\widetilde{\bG}^F}(\widetilde{\chi}))$. In particular, we have $\langle \widetilde{\chi},R_{\widetilde{\bM}}^{\widetilde{\bG}}(\widetilde{\eta})\rangle_{\widetilde{\bG}^F} \neq 0$ so $\mathcal{O}_{\widetilde{\eta}} \leqslant \mathcal{O}_{\widetilde{\chi}}$ by assumption. The statement now follows from the fact that $\mathcal{O}_{\chi} = \mathcal{O}_{\widetilde{\chi}}$ and $\mathcal{O}_{\eta} = \mathcal{O}_{\widetilde{\eta}}$, see the proof of \cite[Lem.~5.1]{geck:1996:on-the-average-values}.

Now assume (ii) of \cref{lem:equiv-dual-statements} holds for the pair $(\bM,\bG)$. Let $\widetilde{\eta} \in \Irr(\widetilde{\bM}^F)$ and $\widetilde{\chi} \in \Irr(\widetilde{\bG}^F)$ be irreducible characters satisfying $\langle \widetilde{\chi}, R_{\widetilde{\bM}}^{\widetilde{\bG}}(\widetilde{\eta})\rangle_{\widetilde{\bG}^F} \neq 0$. As $\widetilde{\bM}$ is $(\widetilde{\bG},F)$-split we have $R_{\widetilde{\bM}}^{\widetilde{\bG}}$ is Harish-Chandra induction and there exists a \emph{character} $\widetilde{\lambda} \in \Class(\widetilde{\bG}^F)$ such that $R_{\widetilde{\bM}}^{\widetilde{\bG}}(\widetilde{\eta}) = \widetilde{\chi} + \widetilde{\lambda}$. Restricting we get
\begin{equation*}
R_{\widetilde{\bM}}^{\widetilde{\bG}}(\Res_{\bM^F}^{\widetilde{\bM}^F}(\widetilde{\eta})) = \Res_{\bG^F}^{\widetilde{\bG}^F}(R_{\widetilde{\bM}}^{\widetilde{\bG}}(\widetilde{\eta})) = \Res_{\bG^F}^{\widetilde{\bG}^F}(\widetilde{\chi}) + \Res_{\bG^F}^{\widetilde{\bG}^F}(\widetilde{\lambda}).
\end{equation*}
Let $\chi \in \Irr(\bG^F \mid \Res_{\bG^F}^{\widetilde{\bG}^F}(\widetilde{\chi}))$ be an irreducible constituent of the restriction. Then there must exist an irreducible character $\eta \in \Irr(\bM^F \mid \Res_{\bM^F}^{\widetilde{\bM}^F}(\widetilde{\eta}))$ such that $\langle \chi, R_{\bM}^{\bG}(\eta) \rangle_{\bG^F} \neq 0$. We now conclude that $\mathcal{O}_{\widetilde{\eta}} = \mathcal{O}_{\eta} \leqslant \mathcal{O}_{\chi} = \mathcal{O}_{\widetilde{\chi}}$.
\end{proof}

\begin{lem}\label{lem:red-smooth-covering}
Let $\iota : \widetilde{\bG} \to \bG$ be a surjective isotypic morphism with connected kernel. If $\bM \leqslant \bG$ and $\widetilde{\bM} \leqslant \widetilde{\bG}$ are any corresponding $F$-stable Levi subgroups then the conditions of \cref{lem:equiv-dual-statements} hold for the pair $(\bM,\bG)$ if they hold for the pair $(\widetilde{\bM},\widetilde{\bG})$.
\end{lem}

\begin{proof}
As $\Ker(\iota)$ is connected we have $\iota$ restricts to surjective homomorphisms $\iota : \widetilde{\bG}^F \to \bG^F$ and $\iota : \widetilde{\bM}^F \to \bM^F$ by the Lang--Steinberg theorem. We then have an isometry $\Class(\bG^F) \to \Class(\widetilde{\bG}^F)$ given by $\chi \mapsto \chi \circ \iota$. In particular, for any irreducible characters $\chi \in \Irr(\bG^F)$ and $\eta \in \Irr(\bM^F)$ we have
\begin{equation*}
\langle R_{\bM}^{\bG}(\eta),\chi\rangle_{\bG^F} = \langle R_{\bM}^{\bG}(\eta)\circ\iota,\chi\circ\iota\rangle_{\widetilde{\bG}^F} = \langle R_{\widetilde{\bM}}^{\widetilde{\bG}}(\eta\circ \iota),\chi\circ\iota\rangle_{\widetilde{\bG}^F},
\end{equation*}
where the last equality follows from \cite[13.22]{digne-michel:1991:representations-of-finite-groups-of-lie-type}. Now $\iota$ restricts to a homeomorphism $\mathcal{U}(\widetilde{\bG}) \to \mathcal{U}(\bG)$ between the unipotent varieties and from the proof of \cite[Lem.~5.2]{geck:1996:on-the-average-values} we see that $\iota(\mathcal{O}_{\chi\circ\iota}) = \mathcal{O}_{\chi}$ and $\iota(\mathcal{O}_{\eta\circ\iota}) = \mathcal{O}_{\eta}$. Hence, if (ii) of \cref{lem:equiv-dual-statements} holds for the pair $(\widetilde{\bM},\widetilde{\bG})$ then it holds for the pair $(\bM,\bG)$.
\end{proof}

\begin{thm}[{}{Bezrukavnikov--Liebeck--Shalev--Tiep, \cite[2.6]{bezrukavnikov-liebeck-shalev-tiep:2017:character-bounds-grps-Lie-type}}]\label{thm:BLST}
Assume that $p$ is a good prime for a connected reductive group $\bG$, $F : \bG \to \bG$ is a Frobenius endomorphism and that $\bM \leqslant \bG$ is a $(\bG,F)$-split Levi subgroup. Then the equivalent conditions of \cref{lem:equiv-dual-statements} hold.
\end{thm}

\begin{proof}
We show that (i) of \cref{lem:equiv-dual-statements} holds. By a result of Asai, there exists a surjective isotypic morphism $\iota : \widetilde{\bG} \to \bG$ such that $\Ker(\iota)$ is connected and $\widetilde{\bG}$ has a simply connected derived subgroup, see \cite[1.21]{taylor:2017:arxiv-structure-of-root-data}. By \cref{lem:red-smooth-covering} we can thus assume that the derived subgroup $[\bG,\bG] \leqslant \bG$ is simply connected. Applying \cref{prop:split-red-to-conn-centre} to a regular embedding $\bG \to \widetilde{\bG}$, which does not change the isomorphism class of the derived subgroup, we can assume $Z(\bG)$ is connected and $[\bG,\bG]$ is still simply connected. 
Now, applying \cref{prop:split-red-to-conn-centre} to the canonical regular embedding $[\bG,\bG] \to \bG$ we can assume that $\bG$ is semisimple and simply connected.

With this assumption we have $\bG = \bG_1 \times \cdots \times \bG_r$ with each $\bG_i \leqslant \bG$ simple and simply connected. The Frobenius endomorphism $F$ permutes the subgroups $\bG_i$ and we may clearly assume that it does so transitively. Using \cite[8.3]{taylor:2017:arxiv-structure-of-root-data} it suffices to consider the case where $\bG$ is simple and simply connected. Moreover, choosing a regular embedding $\bG \to \widetilde{\bG}$ we can assume that $Z(\bG)$ is connected and $[\bG,\bG]$ is simple and simply connected. If $[\bG,\bG] = \SL_n(\mathbb{F})$ then we may, and will, assume that $\bG = \GL_n(\mathbb{F})$.

With this in place we have $p$ is an acceptable prime for $\bG$, in the sense of \cite[6.1]{taylor:2016:GGGRs-small-characteristics}, and the results of Lusztig \cite{lusztig:1992:a-unipotent-support} are available to us, see \cite[13.6]{taylor:2016:GGGRs-small-characteristics}. We may now proceed as in the proof of \cite[2.6]{bezrukavnikov-liebeck-shalev-tiep:2017:character-bounds-grps-Lie-type}.
\end{proof}

\section{Unipotent Supports of Characters and Character Sheaves}\label{sec:unip-supp}
\begin{pa}
Let $\Fam(\bG^{\star},\bT_0^{\star})$ denote the set of all pairs $(s,\mathfrak{C})$ where $s \in \bT_0^{\star}$ is a semisimple element and $\mathfrak{C}$ is a two-sided cell of $W_{\bG^{\star}}(s) := C_{N_{\bG^{\star}}(\bT_0^{\star})}(s)/\bT_0^{\star}$. We refer to the elements of $\Fam(\bG^{\star},\bT_0^{\star})$ as \emph{families}. The Weyl group $W_{\bG^{\star}}$ acts naturally on $\Fam(\bG^{\star},\bT_0^{\star})$ by conjugation and we denote by $[\Fam(\bG^{\star},\bT_0^{\star})]$ the set of orbits under this action. To each family $\mathcal{F} \in \Fam(\bG^{\star},\bT_0^{\star})$ we have a corresponding unipotent class $\mathcal{O_F} \subseteq \mathcal{U}(\bG)/\bG$ of $\bG$, see \cite[10.5]{lusztig:1992:a-unipotent-support} and \cite[12.9]{taylor:2016:GGGRs-small-characteristics}. This assignment is invariant under the action of $W_{\bG^{\star}}$.
\end{pa}

\begin{pa}
The Frobenius defines a permutation of $\Fam(\bG^{\star},\bT_0^{\star})$ and $[\Fam(\bG^{\star},\bT_0^{\star})]$. We denote by $\Fam(\bG^{\star},\bT_0^{\star})^F$ and $[\Fam(\bG^{\star},\bT_0^{\star})]^F$ the respective set of fixed points. By work of Lusztig we have decompositions
\begin{equation*}
\Irr(\bG^F) = \bigsqcup_{\mathcal{F} \in [\Fam(\bG^{\star},\bT_0^{\star})]^F} \mathcal{E}(\bG^F,\mathcal{F}) \quad \text{and} \quad \Irr(\CS(\bG))^F = \bigsqcup_{\mathcal{F} \in [\Fam(\bG^{\star},\bT_0^{\star})]^F} \Irr(\CS(\bG),\mathcal{F}),
\end{equation*}
see \cite[10.6, 11.1]{lusztig:1992:a-unipotent-support}, \cite[16.7]{lusztig:1985:character-sheaves}, and \cite[13.1, 14.7]{taylor:2016:GGGRs-small-characteristics}. We note that if $\mathcal{F} = (s,\mathfrak{C})$ then $\mathcal{E}(\bG^F,\mathcal{F}) \subseteq \mathcal{E}(\bG^F,(s))$, where $\mathcal{E}(\bG^F,(s))$ is the corresponding geometric Lusztig series.
\end{pa}

\begin{pa}\label{pa:weak-Lusztig-conj}
We will need the following compatibility between these decompositions. This property is a consequence of a more precise conjecture of Lusztig which relates the irreducible characters of $\bG^F$ to the characteristic functions of character sheaves. Hence, one may view this as a weak form of Lusztig's conjecture. Lusztig's conjecture has been shown to hold in many important cases but remains open in general.
\begin{enumerate}[leftmargin=1.5cm]
	\item[($\mathcal{W}_{\bG,F}$)] For any family $\mathcal{F} \in [\Fam(\bG^{\star},\bT_0^{\star})]^F$ the subspace of $\Class(\bG^F)$ spanned by $\mathcal{E}(\bG^F,\mathcal{F})$ coincides with that spanned by $\{\mathcal{X}_A \mid A \in \Irr(\CS(\bG),\mathcal{F})\}$.
\end{enumerate}
With this in place we may state the following.
\end{pa}

\begin{lem}\label{lem:equal-supports}
Assume $p$ is a good prime and ($\mathcal{W}_{\bG,F}$) holds. If $A \in \Irr(\CS(\bG))^F$ is an $F$-stable character sheaf and $\chi \in \Irr(\bG^F)$ is an irreducible character satisfying $\langle \chi, \mathcal{X}_A \rangle_{\bG^F} \neq 0$ then $\mathcal{O}_A = \mathcal{O}_{\chi}$.
\end{lem}

\begin{proof}
The statement follows from ($\mathcal{W}_{\bG,F}$) and the fact that if $\chi \in \mathcal{E}(\bG^F,\mathcal{F})$ then $\mathcal{O}_{\chi} = \mathcal{O_F}$, see \cite[11.2]{lusztig:1992:a-unipotent-support} and \cite[\S3.C]{geck-malle:2000:existence-of-a-unipotent-support}, and if $A \in \Irr(\CS(\bG),\mathcal{F})$ then $\mathcal{O}_A = \mathcal{O_F}$, see \cite[10.7]{lusztig:1992:a-unipotent-support} and \cite[13.8]{taylor:2016:GGGRs-small-characteristics}.
\end{proof}

\section{Proof of \texorpdfstring{\cref{thm:main-char-sheaves,thm:main-characters}}{main theorems}}\label{sec:proof-of-closure}

\begin{proof}[of \cref{thm:main-char-sheaves}]
Let $\bQ \leqslant \bG$ be an $F_1$-stable parabolic subgroup with $\bM$ an $F_1$-stable Levi complement and let $K = \ind_{\bM\subseteq\bQ}^{\bG}(A)$. By assumption there exist isomorphisms $F_1^*A \to A$ and $F_1^*B \to B$. In particular, we have an induced isomorphism $F_1^*K \to K$ as in \cref{lem:ind-split-Levi}. As $K \in \CS(\bG)$ is semisimple we have by \cref{lem:ind-split-Levi,lem:char-func-semisimple} that
\begin{equation*}
\langle \mathcal{X}_B, R_{\bM\subseteq\bQ}^{\bG}(\mathcal{X}_A) \rangle_{\bG^F} = \Tr(\sigma_B,\mathfrak{F}_K(B)) \neq 0.
\end{equation*}
Here $R_{\bM}^{\bG} = R_{\bM\subseteq\bQ}^{\bG}$ denotes Harish-Chandra induction with respect to the Frobenius endomorphism $F_1$.

As the irreducible characters form an orthonormal basis of the space of class functions we have decompositions
\begin{align*}
\mathcal{X}_A = \sum_{\eta \in \Irr(\bM^{F_1})} \langle \eta, \mathcal{X}_A\rangle_{\bM^{F_1}}\eta, &&  R_{\bM}^{\bG}(\eta) = \sum_{\chi \in \Irr(\bG^{F_1})} \langle \chi, R_{\bM}^{\bG}(\eta)\rangle_{\bG^{F_1}}\chi,
\end{align*}
where $\eta \in \Irr(\bM^{F_1})$. In particular, we have
\begin{equation*}
\langle \mathcal{X}_B, R_{\bM}^{\bG}(\mathcal{X}_A)\rangle_{\bG^{F_1}} = \sum_{\eta \in \Irr(\bM^{F_1})}\sum_{\chi \in \Irr(\bG^{F_1})} \overline{\langle \eta, \mathcal{X}_A\rangle}_{\bM^{F_1}}\cdot\overline{\langle \chi, R_{\bM}^{\bG}(\eta)\rangle}_{\bG^{F_1}}\cdot\langle \mathcal{X}_B, \chi\rangle_{\bG^{F_1}}.
\end{equation*}
As $\langle \mathcal{X}_B, R_{\bM}^{\bG}(\mathcal{X}_A)\rangle_{\bG^{F_1}} \neq 0$ there must exist irreducible characters $\eta \in \Irr(\bM^{F_1})$ and $\chi \in \Irr(\bG^{F_1})$ such that
\begin{equation*}
\langle \eta, \mathcal{X}_A\rangle_{\bM^{F_1}} \neq 0, \qquad \langle \chi, R_{\bM}^{\bG}(\eta)\rangle_{\bG^{F_1}} \neq 0, \qquad\text{and}\qquad \langle \chi,\mathcal{X}_B\rangle_{\bG^{F_1}} \neq 0.
\end{equation*}
By \cref{thm:BLST,lem:equal-supports} we must therefore have $\mathcal{O}_A = \mathcal{O}_{\eta} \leqslant \mathcal{O}_{\chi} = \mathcal{O}_B$.
\end{proof}

\begin{proof}[of \cref{thm:main-characters}]
Interchanging the roles of the two bases in the above argument we get a decomposition
\begin{equation*}
\langle \chi, R_{\bM}^{\bG}(\eta)\rangle_{\bG^F} = \sum_{A \in \Irr(\CS(\bM))^F}\sum_{B \in \Irr(\CS(\bG))^F} \overline{\langle \mathcal{X}_A,\eta\rangle}_{\bM^F} \cdot \overline{\langle \mathcal{X}_B, R_{\bM}^{\bG}(\mathcal{X}_A)\rangle}_{\bG^F} \cdot \langle \chi, \mathcal{X}_B\rangle_{\bG^F}.
\end{equation*}
Hence, if $\langle \chi, R_{\bM}^{\bG}(\eta)\rangle_{\bG^F} \neq 0$ then there exist $F$-stable character sheaves $A \in \Irr(\CS(\bM))^F$ and $B \in \Irr(\CS(\bG))^F$ such that
\begin{equation*}
\langle \eta, \mathcal{X}_A\rangle_{\bL^F} \neq 0, \qquad \langle \mathcal{X}_B, R_{\bM}^{\bG}(\mathcal{X}_A)\rangle_{\bG^F}\neq 0, \qquad\text{and}\qquad \langle \chi, \mathcal{X}_B\rangle_{\bG^F} \neq 0.
\end{equation*}

We can assume that $(\bL,\Sigma,\mathscr{E},\varphi) \in \Cusp(\bM,F)$ is a cuspidal tuple such that $A \in \Irr(\CS(\bM) \mid K_{\bL,\Sigma,\mathscr{E}}^{\bM})^F$. By \cref{prop:bumper-DL-ind-char-sheaves} there exist irreducible $\alpha$-characters $\widetilde{\lambda} \in \Irr_{\alpha}(W_{\bM}(\bL,\Sigma,\mathscr{E})\sd F)$ and $\widetilde{\mu} \in \Irr_{\alpha}(W_{\bG}(\bL,\Sigma,\mathscr{E})\sd F)$ such that $\mathcal{X}_A = \mathscr{R}_{\bL,\Sigma,\mathscr{E},\varphi}^{\bM}(\widetilde{\lambda})$ and $\mathcal{X}_B = \mathscr{R}_{\bL,\Sigma,\mathscr{E},\varphi}^{\bG}(\widetilde{\mu})$ and
\begin{equation*}
0 \neq \langle \mathcal{X}_B, R_{\bM}^{\bG}(\mathcal{X}_A)\rangle = \langle \widetilde{\mu}, \Ind_{W_{\bM}(\bL,\Sigma,\mathscr{E}).F}^{W_{\bG}(\bL,\Sigma,\mathscr{E}).F}(\widetilde{\lambda})\rangle_{W_{\bG}(\bL,\Sigma,\mathscr{E}).F}.
\end{equation*}
Moreover, $\widetilde{\lambda}$ has irreducible restriction $\lambda \in \Irr_{\alpha}(W_{\bM}(\bL,\Sigma,\mathscr{E}))^F$ and $A \cong K_{\lambda}^{\bM}$. Similarly, $\widetilde{\mu}$ has irreducible restriction $\mu \in \Irr_{\alpha}(W_{\bG}(\bL,\Sigma,\mathscr{E}))^F$ and $B \cong K_{\mu}^{\bG}$.

By \cref{lem:coset-ind-reg-ind} we must have $\langle \mu, \Ind_{W_{\bM}(\bL,\Sigma,\mathscr{E})}^{W_{\bG}(\bL,\Sigma,\mathscr{E})}(\lambda)\rangle_{W_{\bG}(\bL,\Sigma,\mathscr{E})} \neq 0$ which implies that $B \mid \ind_{\bM}^{\bG}(A)$ by \cref{prop:comparison-theorem}. However, by assumption we may apply \cref{thm:main-char-sheaves}, with $F_1 = F^n$, and \cref{lem:equal-supports} to get that $\mathcal{O}_{\eta} = \mathcal{O}_A \leqslant \mathcal{O}_B = \mathcal{O}_{\chi}$ so we are done by Lemma \ref{lem:equiv-dual-statements}.
\end{proof}

\section{Bounding the Multiplicities}\label{sec:multiplicities}
\begin{lem}\label{prop:inner-prod-char-sheaves}
Assume ($\mathcal{R}_{\bG,F}$) holds and $\bM \leqslant \bG$ is an $F$-stable Levi subgroup. If $B_1,B_2 \in \Irr(\CS(\bM))^F$ are $F$-stable character sheaves and $(\bL,\Sigma,\mathscr{E}) \in \Cusp(\bM)$ is a cuspidal triple such that the set $\Irr(\CS(\bM) \mid K_{\bL,\Sigma,\mathscr{E}}^{\bM})$ contains either $B_1$ or $B_2$ then we have
\begin{equation*}
|\langle R_{\bM}^{\bG}(\mathcal{X}_{B_1}), R_{\bM}^{\bG}(\mathcal{X}_{B_2})\rangle_{\bG^F}| \leqslant |W_{\bG}(\bL,\Sigma,\mathscr{E})|.
\end{equation*}
\end{lem}

\begin{proof}
Let us fix a cuspidal tuple $(\bL_i,\Sigma_i,\mathscr{E}_i,\varphi_i) \in \Cusp(\bM,F)$ such that $B_i \in \Irr(\CS(\bM)^F \mid K_{\bL_i,\Sigma_i,\mathscr{E}_i}^{\bM})$. After \cref{prop:bumper-DL-ind-char-sheaves} we have $\mathcal{X}_{B_i} = \mathcal{R}_{\bL_i,\Sigma_i,\mathscr{E}_i,\varphi_i}^{\bM}(\widetilde{\lambda}_i)$ for some irreducible $\alpha$-character $\widetilde{\lambda}_i \in \Irr_{\alpha}(W_{\bM}(\bL_i,\Sigma_i,\mathscr{E}_i)\sd F \downarrow W_{\bM}(\bL_i,\Sigma_i,\mathscr{E}_i))$. If $\langle R_{\bM}^{\bG}(\mathcal{X}_{B_1}), R_{\bM}^{\bG}(\mathcal{X}_{B_2})\rangle_{\bG^F}$ is zero then there is nothing to show so we will assume that this inner product is non-zero. If this is the case then we must have the tuples $(\bL_1,\Sigma_1,\mathscr{E}_1)$ and $(\bL_2,\Sigma_2,\mathscr{E}_2)$ are in the same $\bG$-orbit. We fix a representative $(\bL,\Sigma,\mathscr{E}) \in \Cusp(\bM)^F$ of that $\bG$-orbit.

For brevity let us set $\mathcal{W}_{\bG} = W_{\bG}(\bL,\Sigma,\mathscr{E})$ and $\mathcal{W}_{\bM} = W_{\bM}(\bL,\Sigma,\mathscr{E})$. There exists an element $g_i \in \bG$ such that $(\bL_i,\Sigma_i,\mathscr{E}_i) = (\boldsymbol{\iota}_{g_i}^{-1}(\bL),\boldsymbol{\iota}_{g_i}^{-1}(\Sigma),\boldsymbol{\iota}_{g_i}^*(\mathscr{E}))$. As all the triples are $F$-stable we must have $F(g_i)g_i^{-1} \in N_{\bG}(\bL,\Sigma,\mathscr{E})$ represents an element $w_i^{-1} \in \mathcal{W}_{\bG}$. Conjugating by $g_i$ identifies the pair $(W_{\bM}(\bL_i,\Sigma_i,\mathscr{E}_i),F)$ with $(\mathcal{W}_{\bM},\boldsymbol{\iota}_{w_i}F)$. Identifying $\widetilde{\lambda}_i$ as an irreducible character of $\mathcal{W}_{\bM} \sd w_iF$, c.f., \cref{sec:coset-multiplicities}, we get that
\begin{equation*}
\langle R_{\bM}^{\bG}(\mathcal{X}_{B_1}), R_{\bM}^{\bG}(\mathcal{X}_{B_2})\rangle_{\bG^F} = \langle \Ind_{\mathcal{W}_{\bM}.w_1F}^{\mathcal{W}_{\bG}.F}(\widetilde{\lambda}_1),\Ind_{\mathcal{W}_{\bM}.w_2F}^{\mathcal{W}_{\bG}.F}(\widetilde{\lambda}_2)\rangle_{\mathcal{W}_{\bG}.F}.
\end{equation*}
We now apply \cref{lem:triv-bound-inner-prod}.
\end{proof}

\begin{prop}\label{prop:multiplicities}
Assume ($\mathcal{R}_{\bG,F}$) and ($\mathcal{W}_{\bG,F}$) hold then for any irreducible character $\eta \in \Irr(\bM^F)$ we have
\begin{equation*}
|\langle R_{\bM}^{\bG}(\eta), R_{\bM}^{\bG}(\eta)\rangle_{\bG^F}| \leqslant B(\bM)^4\cdot |W_{\bG}| \leqslant B(\bG)^4\cdot |W_{\bG}|.
\end{equation*}
\end{prop}

\begin{proof}
Expanding $\eta$ in terms of characteristic functions of character sheaves we obtain a decomposition
\begin{equation*}
\langle R_{\bM}^{\bG}(\eta), R_{\bM}^{\bG}(\eta)\rangle_{\bG^F} = \sum_{B_1,B_2 \in \Irr(\CS(\bM))^F)}\langle \mathcal{X}_{B_1},\eta \rangle_{\bM^F}\cdot\overline{\langle \mathcal{X}_{B_2}, \eta\rangle}_{\bM^F}\cdot\langle R_{\bM}^{\bG}(\mathcal{X}_{B_1}), R_{\bM}^{\bG}(\mathcal{X}_{B_2})\rangle_{\bG^F}.
\end{equation*}
Let us assume that $\mathcal{F} \in \Fam(\bM^{\star},\bT_0^{\star})^F$ is a family such that $\eta \in \mathcal{E}(\bM^F,\mathcal{F})$, c.f., the proof of \cref{lem:equal-supports}. Moreover, let us denote by $\Class(\bM^F,\mathcal{F}) \subseteq \Class(\bM^F)$ the subspace spanned by the irreducible characters in $\mathcal{E}(\bM^F,\mathcal{F})$.

Note that for any character sheaf $B \in \Irr(\CS(\bM)^F)$ we have
\begin{equation*}
1 = \langle \mathcal{X}_B,\mathcal{X}_B\rangle_{\bM^F} = \sum_{\eta' \in \Irr(\bM^F)} \langle \eta',\mathcal{X}_B\rangle_{\bM^F}\cdot\overline{\langle \eta',\mathcal{X}_B\rangle}_{\bM^F} = \sum_{\eta' \in \Irr(\bM^F)} |\langle \eta',\mathcal{X}_B\rangle_{\bM^F}|
\end{equation*}
by the orthonormality of the irreducible characters of $\bM^F$ and the characteristic functions of the character sheaves, c.f., \cref{thm:char-funcs-as-basis}. This implies that $|\langle \eta,\mathcal{X}_{B_i}\rangle| \leqslant 1$.

In \cite[Chapter 4]{lusztig:1984:characters-of-reductive-groups} and \cite{lusztig:1984:characters-of-reductive-groups-over-finite-fields}, see also \cite{lusztig:2018:on-the-definitions-of-almost-characters}, Lusztig has associated to the family $\mathcal{F}$ a pair $(\mathcal{G_F},\phi)$ consisting of a finite group $\mathcal{G_F}$ and an automorphism $\phi \in \Aut(\mathcal{G_F})$. Moreover, he has defined a corresponding set $\overline{\mathcal{M}}(\mathcal{G_F},\phi)$ consisting of pairs $(x,\sigma)$, where $x \in \mathcal{G_F}.\phi$ and $\sigma \in \Irr(C_{\mathcal{G}_F}(x))$, taken up to equivalence modulo the action of $\mathcal{G_F}\sd\phi$ defined by $g\cdot(x,\sigma) = ({}^gx,\sigma\circ\boldsymbol{\iota}_g^{-1})$. Note that $\mathcal{G_F}.\phi$ is a coset as in \cref{sec:coset-multiplicities} and $C_{\mathcal{G}_F}(x)$ denotes the stabiliser of $x$ under the natural conjugation action of $\mathcal{G_F}$ on $\mathcal{G_F}.\phi$. The main result of \cite{lusztig:1984:characters-of-reductive-groups} and \cite{lusztig:1988:reductive-groups-with-a-disconnected-centre}, see also \cite{lusztig:1984:characters-of-reductive-groups-over-finite-fields}, shows that there is a bijection
\begin{equation*}
\overline{\mathcal{M}}(\mathcal{G_F},\phi) \to \mathcal{E}(\bG^F,\mathcal{F}).
\end{equation*}

As we assume ($\mathcal{W}_{\bG,F}$) holds it must be the case that if $B \in \Irr(\CS(\bM)^F)$ is a character sheaf satisfying $\langle \eta,\mathcal{X}_B\rangle \neq 0$ then $\mathcal{X}_B \in \Class(\bG^F,\mathcal{F})$. Moreover, as the characteristic functions form a basis there can be at most
\begin{equation*}
\dim (\Class(\bG^F,\mathcal{F})) = |\overline{\mathcal{M}}(\mathcal{G_F},\phi)| \leqslant |\mathcal{G_F}|^2
\end{equation*}
character sheaves satisfying $\langle \eta,\mathcal{X}_B\rangle \neq 0$. Combining this with \cref{prop:inner-prod-char-sheaves} we get that
\begin{equation*}
\langle R_{\bM}^{\bG}(\eta), R_{\bM}^{\bG}(\eta)\rangle \leqslant |\mathcal{G_F}|^4\cdot |W_{\bG}(\bL,\Sigma,\mathscr{E})|
\end{equation*}
where $(\bL,\Sigma,\mathscr{E}) \in \Cusp(\bM)$ is a tuple as in the statement of \cref{prop:inner-prod-char-sheaves}.

Now $W_{\bG}(\bL,\Sigma,\mathscr{E})$ is a subgroup of the relative Weyl group $W_{\bG}(\bL)$ which may be identified with a section of the Weyl group $W_{\bG}$. In particular we have $|W_{\bG}(\bL,\Sigma,\mathscr{E})| \leqslant |W_{\bG}|$. To prove the first stated inequality it suffices to show that $|\mathcal{G_F}| \leqslant |B(\bM)|$. It is known that there exists a (special) element $g \in \bM^{\star}$ in a group dual to $\bM$ such that the group $\mathcal{G_F}$ can be identified with a quotient of the component group $A_{\bM^{\star}}(g) = C_{\bM^{\star}}(g)/C_{\bM^{\star}}^{\circ}(g)$, see \cite[Chapter 13]{lusztig:1984:characters-of-reductive-groups}, \cite{lusztig:1984:characters-of-reductive-groups-over-finite-fields}, and \cite{lusztig:2014:families-and-springers-correspondence}.

If $\bL = C_{\bM^{\star}}(s)$ then $u \in \bL$ and we have $A_{\bL}(u) \cong A_{\bM^{\star}}(g)$. Now we have an injection $A_{\bL^{\circ}}(u) \to A_{\bL}(u)$ because $C_{\bL^{\circ}}^{\circ}(u) = C_{\bL}^{\circ}(u)$. Indeed $C_{\bL}^{\circ}(u) \leqslant C_{\bL}(u) \cap \bL^{\circ} = C_{\bL^{\circ}}(u)$ and the reverse inclusion is clear. Hence, we have an injection
\begin{equation*}
A_{\bL}(u)/A_{\bL^{\circ}}(u) \cong C_{\bL}(u)/C_{\bL^{\circ}}(u) \hookrightarrow \bL/\bL^{\circ}.
\end{equation*}
This means that $|A_{\bL}(u)| \leqslant |A_{\bL^{\circ}}(u)|\cdot |\bL/\bL^{\circ}|$. After \cite[13.14(iii)]{digne-michel:1991:representations-of-finite-groups-of-lie-type} we have $|\bL/\bL^{\circ}| \leqslant |Z(\bM)/Z^{\circ}(\bM)|$.

Now let $\pi : \bH \to \bL^{\circ}$ be a simply connected cover of the derived subgroup of $\bL^{\circ}$. There exists a unique unipotent element $v \in \bH$ such that $\pi(v) = u$. The map $\pi$ defines a surjective homomorphism $A_{\bH}(v) \to A_{\bL^{\circ}}(u)$. Putting things together we have shown that $|\mathcal{G_F}| \leqslant |A_{\bH}(v)|\cdot |Z(\bM)/Z^{\circ}(\bM)| \leqslant B(\bM)$ as desired.

Finally, by \cite[4.2]{bonnafe:2006:sln} we have $|Z(\bM)/Z^{\circ}(\bM)| \leqslant |Z(\bG)/Z^{\circ}(\bG)|$ so $B(\bM) \leqslant B(\bG)$ by definition.
\end{proof}

Next we record a lemma that is used in the proof of \cref{thm:character-bound}, which is essentially observed in \cite[12.22]{digne-michel:1991:representations-of-finite-groups-of-lie-type} and \cite[\S25.A]{bonnafe:2006:sln}. We include a proof of this result for the convenience of the reader.

\begin{lem}\label{lem:restrict-equal-val}
Let $\bM \leqslant \bG$ be an $F$-stable Levi subgroup. If $g \in \bM^F$ is any element satisfying $C_{\bG}^{\circ}(g) \leqslant \bM$ then $\chi(g) = {}^*R_{\bM}^{\bG}(\chi)(g)$ for any irreducible character $\chi \in \Irr(\bG^F)$.
\end{lem}

\begin{proof}
Let $g = su = us$ be the Jordan decomposition of the element. As $C_{\bG}^{\circ}(g) \leqslant \bM$ we have by \cite[1.3]{bonnafe:2004:actions-of-rel-Weyl-grps-I} that $C_{\bG}^{\circ}(s) \leqslant \bM$ and $C_{\bG}^{\circ}(s) = C_{\bM}^{\circ}(s)$. Thus, by the formula in \cite[12.5]{digne-michel:1991:representations-of-finite-groups-of-lie-type}
\begin{equation*}
\Res_{C_{\bM}^{\circ}(s)^F}^{\bM^F}\circ {}^*R_{\bM}^{\bG}
= {}^*R_{C_\bM^\circ(s)}^{C_\bG^\circ(s)} \circ \Res_{C_{\bG}^{\circ}(s)^F}^{\bG^F} = \Res_{C_{\bG}^{\circ}(s)^F}^{\bG^F}.
\end{equation*}
Note that $u \in C_{\bG}^{\circ}(s)$, as $C_{\bG}(s)/C_{\bG}^{\circ}(s)$ is a $p'$-group, so $g \in C_{\bG}^{\circ}(s)$. Hence, the statement follows by evaluating this formula at $\chi$ and then further at $g$.
\end{proof}

\begin{proof}[of \cref{cor:bound}]
We assume $\iota : \bG \to \widetilde{\bG}$ is a regular embedding and $\widetilde{\bM} \leqslant \widetilde{\bG}$ is the $F$-stable Levi subgroup corresponding to $\bM \leqslant \bG$, c.f., \cref{pa:isotypic-morphisms}. Consider an irreducible character $\widetilde{\chi} \in \Irr(\widetilde{\bG}^F)$ such that $\chi$ is a constituent of the restriction $\Res_{\bG^F}^{\widetilde{\bG}^F}(\widetilde{\chi})$. By \cref{thm:character-bound} we have
\begin{equation*}
|\widetilde{\chi}(g)| \leqslant f(r)\cdot \widetilde{\chi}(1)^{\alpha_{\widetilde{\bG}}(\widetilde{\bM},F)}.
\end{equation*}

As $\iota$ defines a bijection between the unipotent classes of $\bG$ and $\widetilde{\bG}$ which preserves the dimension of each class, and similarly for $\bM$ and $\widetilde{\bM}$, we have $\alpha_{\widetilde{\bG}}(\widetilde{\bM},F) = \alpha_{\bG}(\bM,F)$. By a result of Lusztig \cite{lusztig:1988:reductive-groups-with-a-disconnected-centre} the restriction
\begin{equation*}
\Res_{\bG^F}^{\widetilde{\bG}^F}(\widetilde{\chi}) = \chi_1 + \ldots + \chi_m,
\end{equation*}
with each $\chi_i \in \Irr(\bG^F)$, is multiplicity free. Hence, if $\chi$ is $\widetilde{\bG}^F$-invariant, then $\chi = \Res_{\bG^F}^{\widetilde{\bG}^F}(\widetilde{\chi})$ and we are done in this case.
 
Next assume that $g^{\widetilde{\bG}^F} = g^{\bG^F}$. This implies that $\widetilde{\bG}^F = C_{\widetilde{\bG}^F}(g)\bG^F$, and so
we can find $x_i \in C_{\widetilde{\bG}^F}(g)$ such that $\chi_i = \chi^{x_i}$. It follows that $\chi_i(g) = \chi(g)$ and
$\widetilde{\chi}(g) = m\chi(g)$. Since $\widetilde{\chi}(1) = m\chi(1)$, we now have
\begin{equation*}
|\chi(g)| = |\widetilde{\chi}(g)|/m \leq f(r)\widetilde{\chi}(1)^{\alpha_{\bG}(\bM,F)}/m \leq f(r)\chi(1)^{\alpha_\bG(\bM,F)},
\end{equation*}
as claimed.

Suppose now that $C_\bG(g)$ is connected, and consider any $h \in g^{\widetilde{\bG}^F} \subseteq \bG^F$. 
As $\widetilde{\bG} = Z(\widetilde{\bG})\bG$, we have that 
$h \in \bG^F$ is $\bG$-conjugate to $g$. By the Lang-Steinberg theorem, $h$ is $\bG^F$-conjugate to $g$. Thus 
$g^{\widetilde{\bG}^F} = g^{\bG^F}$, and we are done by the previous paragraph.

Finally, if $[\bG,\bG]$ is simply connected and $g$ is semisimple, then $C_\bG(g)$ is connected by Steinberg's theorem, and we 
can apply the previous result.
\end{proof}

We end this section by recording the following observation.

\begin{lem}\label{centralisers}
For an $F$-stable Levi subgroup $\bM \leqslant \bG$ and an element $g \in \bM^F$ with semisimple part $s$, consider the following four conditions:
\begin{enumerate}[label=\textnormal{(\roman*)}]
	\item $C_{\bG}^{\circ}(s) \leqslant \bM$,
	\item $C_{\bG}^\circ(g) \leqslant \bM$,
	\item $C_{\bG}^{\circ}(s)^F \leqslant \bM^F$,
	\item $C_{\bG}^{\circ}(g)^F \leqslant \bM^F$.
\end{enumerate}
Then \textnormal{(i)} and \textnormal{(ii)} are equivalent. Moreover, if $\bM$ is $(\bG,F)$-split then all four conditions are equivalent.
\end{lem}

\begin{proof}
The equivalence of (i) and (ii) is part of \cite[Lem.~1.2]{bonnafe:2004:actions-of-rel-Weyl-grps-I}.
Assume now that $\bM$ is $(\bG,F)$-split. Certainly, (i) $\Rightarrow$ (iii) $\Rightarrow$ (iv). We now prove that (iv) $\Rightarrow$ (i). Let $\bQ \leqslant \bG$ be an $F$-stable parabolic subgroup with Levi complement $\bM$. We denote by $\bV$ the unipotent radical of $\bQ$ so that $\bQ = \bV \rtimes \bM$; note that $\bV$ is $F$-stable because $\bQ$ is. By a result of Spaltenstein $C_{\bV}(g)$ is connected, see \cite[1.2]{bonnafe:2004:actions-of-rel-Weyl-grps-I}, so $C_{\bV}(g) \leqslant C_{\bG}^{\circ}(g)$. Hence, 
(iv) implies that $C_{\bV}(g)^F \leqslant \bM \cap \bV = \{1\}$. However, $C_{\bV}(g)$ is a connected group all of whose elements are unipotent, so by Rosenlicht's Theorem we have that
\begin{equation*}
|C_{\bV}(g)^F| = q^{\dim C_{\bV}(g)}.
\end{equation*}
This implies that $\dim C_{\bV}(g) = 0$, so by \cite[1.3]{bonnafe:2004:actions-of-rel-Weyl-grps-I} we have that 
$C_{\bG}^{\circ}(s) \leqslant \bM$.
\end{proof}

\section{\texorpdfstring{Split groups of type $A$}{Split groups of type A}}\label{sec:type-A}

In this section, we prove our main results on finite general and special linear groups. First we prove the following result, which improves 
\cite[Thm.~3.3]{bezrukavnikov-liebeck-shalev-tiep:2017:character-bounds-grps-Lie-type}. We follow the same proof, but make the involved 
bounds explicit:

\begin{thm}\label{gl-uni}
There is a function $g:\N \to \N$ such that the following statement holds. For any $n \geq 2$, any prime
power $q$, $\ell = 0$ or any prime not dividing $q$, any irreducible $\ell$-Brauer character $\varphi$ of
$G:= \GL_n(q)$, and any unipotent element $1 \neq u \in G$,
\begin{equation*}
|\varphi(u)| \leq g(n) \cdot \varphi(1)^{\frac{n-2}{n-1}}.
\end{equation*}
Furthermore, if $q > 3n^2$ then one can take $g(n) = 3(f(n-1)+1)$, where $f$ is the function in 
\cite[Thm.~1.1]{bezrukavnikov-liebeck-shalev-tiep:2017:character-bounds-grps-Lie-type}.
\end{thm}

\begin{proof}
(a) Note that the statement holds when $n=2$ (choosing $g(2) = 1$) as in this case we have $|\varphi(u)| \leq 1$. So in what follows
we may assume $n \geq 3$.

Recall the partial order $\leq$ on the set of unipotent classes of $\GC = \GL_n(\KK)$:
$x^\GC \leq y^\GC$ precisely when $x^\GC \subseteq \overline{y^\GC}$, and we consider
$G = \GC^F$ for a suitable Frobenius endomorphism $F$.
We will prove by induction using
the partial order $\leq$ that, if $u$ is parametrized by a partition $\lam \vdash n$ then
\begin{equation*}
|\varphi(u)| \leq g(n) \cdot \varphi(1)^{\frac{n-2}{n-1}}
\end{equation*}
for some positive constant $g(n)$ depending only on $n$, and moreover one can take
\begin{equation}\label{for-g}
g(n) = 3(f(n-1)+1)
\end{equation}
if $q > 3n^2$ and $f$ is the function in \cite[Thm.~1.1]{bezrukavnikov-liebeck-shalev-tiep:2017:character-bounds-grps-Lie-type} (or the function
$f$ in \cref{thm:character-bound}).

Also, recall that $u$ is a {\it Richardson} unipotent element, that is, we can find an $F$-stable parabolic subgroup
$\PC$ with unipotent radical $\UC$ such that $u^\GC \cap \UC$ is an open dense subset of $\UC$ that forms
a single $\PC$-orbit. As in the proof of \cite[Thm.~3.3]{bezrukavnikov-liebeck-shalev-tiep:2017:character-bounds-grps-Lie-type},
we have $u^\GC \cap U$ is a single $\PC^F$-orbit, where $U :=\UC^F$, and so
\begin{equation*}
|u^\GC \cap U| = [\PC^F:C],
\end{equation*}
where $C := \CB_\PC(u)^F = \CB_\GC(u)^F$. The structure of the connected algebraic group $\CB_\GC(u)$ is given in
\cite[Thm.~3.1]{LS}; in particular, its quotient by the unipotent radical $R_u(\CB_\GC(u))$ is a product of $\GL$-factors. As $\dim\CB_\GC(u) = \dim\PC-\dim\UC$, 
it follows that
\begin{equation*}
|C| \leq q^{\dim \PC-\dim \UC}.
\end{equation*}
On the other hand,
\begin{equation*}
|\PC^F| \geq (q-1)^{\dim \PC} = q^{\dim \PC}\left(1-\frac{1}{q}\right)^{\dim \PC}.
\end{equation*}
Note that when $q > 3n^2$, we have that
\begin{equation*}
\left(1-\frac{1}{q}\right)^{\dim \PC} > \left(1-\frac{1}{q}\right)^{n^2} > 1- \frac{n^2}{q} > \frac{2}{3}.
\end{equation*}
As $|U| = q^{\dim \UC}$, it follows that
\begin{equation}\label{for-u}
  |u^\GC \cap U| \geq \frac{2}{3}|U|, \qquad |U \smallsetminus u^\GC| \leq \frac{1}{3}|U|
\end{equation}
for all $q > 3n^2$ and all $\lam \vdash n$. By taking $g(n)$ large enough, say
\begin{equation}\label{for-g2}
  g(n) \geq \max_{q'= p^r < 3n^2} \left\{ \frac{|\psi(w)|}{\psi(1)^{\frac{n-2}{n-1}}} \mid 1 \neq w \in \GL_n(q'),~w \mbox{ unipotent},~\psi \in \IBR_\ell(\GL_n(q'))
    \right\},
\end{equation}
we may assume that the condition $q > 3n^2$ is indeed satisfied.

(b) Now we will assume that $q > 3n^2$ and choose $g(n) = 3(f(n-1)+1)$ as in \eqref{for-g}. 
Let $1 \neq w \in U \smallsetminus u^\GC$ be labeled by $\nu \vdash n$. Then
\begin{equation*}
w \in \UC = \overline{u^\GC \cap \UC},
\end{equation*}
and so $w^\GC \leq u^\GC$. In particular, if $u^\GC$ is minimal with respect to $\leq$, then
no such $w$ exists. If $u^\GC$ is not minimal, then by the induction hypothesis applied to $w^\GC$ we have
\begin{equation}\label{for-w}
  |\varphi(w)| \leq g(n) \cdot \varphi(1)^{\frac{n-2}{n-1}}.
\end{equation}
Let $\rho := {}^*R^\GC_\LC(\varphi)$, where $\LC$ is an $F$-stable Levi subgroup of $\PC$. Then
\begin{equation*}
\rho(1) = [\varphi|_U,1_U]_U = \frac{1}{|U|}\left( \varphi(1) + \sum_{1 \neq w \in U \smallsetminus u^\GC}\varphi(w) + \sum_{u' \in u^\GC \cap U}\varphi(u')\right).
\end{equation*}
As $u^\GC \cap U = u^G$, we then get
\begin{equation*}
|u^\GC \cap U|\cdot|\varphi(u)|  \leq |U|\rho(1) + \sum_{1 \neq w \in U \smallsetminus u^\GC}|\varphi(w)| + \varphi(1).
\end{equation*}
It now follows from \eqref{for-u} and \eqref{for-w} that
\begin{equation*}
|\varphi(u)| \leq \frac{3}{2}\rho(1) + \frac{1}{2}g(n)\varphi(1)^{\frac{n-2}{n-1}} + \frac{3}{2|U|}\varphi(1).
\end{equation*}
Next, \cite[Thm.~1.1]{bezrukavnikov-liebeck-shalev-tiep:2017:character-bounds-grps-Lie-type} and the bound $\al \le \frac{n-2}{n-1}$ in 
\cite[Prop.~4.5]{bezrukavnikov-liebeck-shalev-tiep:2017:character-bounds-grps-Lie-type} imply that
\begin{equation*}
\rho(1) \leq f(n-1)\varphi(1)^{\frac{n-2}{n-1}}.
\end{equation*}
On the other hand,
$|U| \geq q^{n-1}$ and $\varphi(1) < q^{n^2/2}$, whence for $n \geq 4$ we have
\begin{equation*}
\frac{\varphi(1)}{|U|} < \varphi(1)^{\frac{n-2}{n-1}}.
\end{equation*}
The same conclusion holds for $n = 3$ since $\varphi(1) < q^4$ in this case. Consequently,
\begin{equation*}
|\varphi(u)| \leq \left(\frac{3}{2}(f(n-1)+1)+\frac{1}{2}g(n)\right)\varphi(1)^{\frac{n-2}{n-1}} = g(n)\varphi(1)^{\frac{n-2}{n-1}},
\end{equation*}
and the induction step is completed.
\end{proof}

\begin{proof}[of \cref{thm:glsl}]
Let $g=su=us$ with $s$ semisimple and $u$ unipotent. Also view $G = \bG^F$ with $\bG = \GL_n(\mathbb{F})$ or $\SL_n(\mathbb{F})$ and $F$ a suitable Frobenius 
endomorphism.

(a) First we consider the case $G = \GL_n(q)$. If $s \in Z(G)$ then the statement follows from \cref{gl-uni}. If $s \notin Z(G)$, then $C_{\bG}(g)$ is contained in an $F$-stable proper Levi subgroup, and
so we are done by  \cref{thm:character-bound} and \cite[Prop.~4.3]{bezrukavnikov-liebeck-shalev-tiep:2017:character-bounds-grps-Lie-type}.

(b) Now consider the case $G = \SL_n(q)$ and view $G$ as $[\tilde G,\tilde G]$, where $\tilde G \cong \GL_n(q)$. Arguing as in the proof of
\cite[Thm.~1.5]{bezrukavnikov-liebeck-shalev-tiep:2017:character-bounds-grps-Lie-type}, we are done if $\chi$ is $\tilde G$-invariant, or if 
$g^G = g^{\tilde G}$. 

From now on we may assume that $\chi$ is not $\tilde G$-invariant, and that $g^G \neq g^{\tilde G}$.
Suppose in addition that $s \in Z(G)$. Then the proof of \cite[Thm.~1.5]{bezrukavnikov-liebeck-shalev-tiep:2017:character-bounds-grps-Lie-type},
shows that 
\begin{equation*}
|\chi(g)| \leq (f(n-1)+1)\chi(1)^{\frac{n-2}{n-1}},
\end{equation*}
unless possibly $n=6$, $|\chi(g)| \leq q^9$, and $\chi(1) \geq q^{21/2}(q-1)/6$. Since $g(6) = 3(h(5)+1) > 7$, one can check in the exceptional case
that
\begin{equation*} 
|\chi(g)| \leq 3(f(n-1)+1)\chi(1)^{\frac{n-2}{n-1}}
\end{equation*}
as well.

Thus we may now assume in addition that 
$s \notin Z(G)$. As in the proof of \cref{cor:bound}, we have that $u \neq 1$. Furthermore, as in the proof of 
\cite[Thm.~1.5]{bezrukavnikov-liebeck-shalev-tiep:2017:character-bounds-grps-Lie-type}, we also have that either $\chi(1) \geq q^{(n^2+n)/4}$,
or $2 \mid n$ and
\begin{equation*}
\chi(1) \geq \frac{1}{2}\prod^{n/2}_{j=1}(q^{2j}-1) > \frac{q^{n^2/4}}{2}\left(1 - \sum^{n/2}_{j=1}q^{1-2j} \right) > q^{n^2/4-2.6}.
\end{equation*}
Thus in either case we have that
\begin{equation}\label{eq:deg-gl}
  \chi(1) > q^{n^2/4-2.6}.
\end{equation}  
Now, if $C_{\tilde G}(s) \not\cong \GL_{n/k}(q^k)$ for some $k > 1$, then $C_\bG(g) \leq C_\bG(s)$ is contained in a proper split Levi subgroup of $\bG$,
and we are done again by applying Theorem 1.1 and Proposition 4.3 of \cite{bezrukavnikov-liebeck-shalev-tiep:2017:character-bounds-grps-Lie-type}.
In the remaining case $C_{\tilde G}(s) \cong \GL_{n/k}(q^k)$ for some $k > 1$,  the assumption $u \neq 1$ implies that
$|C_G(g)| \leq q^{n^2/2-2n+4}$ (see part (ii) of the proof of \cite[Thm.~1.5]{bezrukavnikov-liebeck-shalev-tiep:2017:character-bounds-grps-Lie-type}),
and so $|\chi(g)| \leq q^{n^2/4-n+2}$, whence $|\chi(g)| < \chi(1)^{(n-2)/(n-1)}$ by \eqref{eq:deg-gl}.
\end{proof}

\section{\texorpdfstring{Twisted Groups of Type $\A$}{Twisted Groups of Type A}}\label{sec:twisted-type-A}
\begin{pa}
In this section we assume that $\bG = \GL_n(\mathbb{F})$ and $\bT_0 \leqslant \bB_0$ are the maximal torus and Borel subgroup of diagonal and upper triangular matrices respectively. Let $W = W_{\bG}(\bT_0)$ and $\mathcal{S}_{\bG}(\bT_0,F) \subseteq \bT_0\times W$ be the set of pairs $(s,w)$ satisfying ${}^wF(s) = s$. Recall that, as $\bG = \GL_n(\mathbb{F})$, we have $\bG = \bG^{\star}$ is self-dual. Now, to each pair $(s,w) \in \mathcal{S}_{\bG}(\bT_0,F)$ we have a corresponding Deligne--Lusztig character $R_w^{\bG,F}(s) \in \Class(\bG^F)$. Moreover, we have a Green function
\begin{equation*}
Q_w^{\bG,F} = R_w^{\bG,F}(s)|_{\mathcal{U}(\bG)^F} : \mathcal{U}(\bG)^F \to \Ql
\end{equation*}
which is well known to be independent of $s$.
\end{pa}

\begin{pa}
If $s \in \bT_0$ then $W(s) := W_{\bG}(s)$ is a product of symmetric groups so the $2$-sided cells of $W_{\bG}(s)$ are in bijection with $\Irr(W_{\bG}(s))$. Hence we may, and will, identify the families $\Fam(\bG^{\star},\bT_0^{\star}) = \Fam(\bG,\bT_0)$ with pairs $(s,\lambda)$ such that $s \in \bT_0$ and $\lambda \in \Irr(W_{\bG}(s))$. Now in our setting with $\bG = \GL_n(\mathbb{F})$ we have a bijection $[\Fam(\bG,\bT_0)]^F \to \Irr(\bG^F)$, which we denote by $[s,\lambda] \mapsto \chi_{[s,\lambda]}$.
\end{pa}

\begin{pa}
We can construct $\chi_{[s,\lambda]}$ as follows. As $[s,\lambda]$ is $F$-stable there exists a $w \in W$ such that $(s,w) \in \mathcal{S}_{\bG}(\bT_0,F)$. Clearly $\boldsymbol{\iota}_wF(W(s)) = W(s)$ so we can define a map $\mathcal{R}_{s,w}^{\bG,F} : \Class(W(s).wF) \to \Class(\bG^F, s)$ by setting
\begin{equation*}
\mathcal{R}_{s,w}^{\bG,F}(f) = \frac{1}{|W(s)|}\sum_{x \in W(s)} f(xwF)R_{xw}^{\bG,F}(s).
\end{equation*}
It is well known that, as $\bG = \GL_n(\mathbb{F})$, if $\widetilde{\varphi} \in \Irr(W(s).wF)$ then there exists a root of unity $\varepsilon_{s,\widetilde{\varphi}} \in \Ql^{\times}$ such that $\chi_{[s,\varphi]} = \varepsilon_{s,\widetilde{\varphi}}\mathcal{R}_{s,w}^{\bG,F}(\widetilde{\varphi})$.
\end{pa}

\begin{pa}
For the rest of this section $F : \bG \to \bG$ denotes the Frobenius endomorphism defined by $F(a_{ij}) = (a_{ij}^q)$ and $F' : \bG \to \bG$ denotes the Frobenius endomorphism defined by $F'(A) = {}^{n_0}F(A^{-T})$. Here $n_0 \in N_{\bG}(\bT_0)$ represents the longest element $w_0 \in W = W_{\bG}(\bT_0)$. In particular, we have $\bG^F = \GL_n(q)$ and $\bG^{F'} = \GU_n(q)$.
\end{pa}

\begin{lem}\label{lem:ennola-duality-ss-classes}
Assume $(s,w) \in \mathcal{S}_{\bG}(\bT_0,F')$. If $q > n$ then there exists an element $t \in \bT_0$ such that $(t,ww_0) \in \mathcal{S}_{\bG}(\bT_0,F)$ and $W(s) = W(t)$.
\end{lem}

\begin{proof}
If $n=1$ the statement is trivial so we assume that $n \geqslant 2$. For any integers $r,a >0$ and $\zeta \in \mathbb{F}^{\times}$ we set $B_r^a(\zeta) = (\zeta,\zeta^r,\dots,\zeta^{r^{a-1}}) \in (\mathbb{F}^{\times})^a$. Moreover, for any integer $m > 0$ we denote by $\mu_m \leqslant \mathbb{F}^{\times}$ the subgroup of elements $\zeta \in \mathbb{F}^{\times}$ such that $\zeta^m = 1$. We assume $\mu_0$ is the trivial subgroup. It suffices to prove the statement for any element in the same $W$-orbit as $s$. Hence, we can assume that $s = \diag(B_1^{m_1}(\zeta_1),\dots,B_1^{m_k}(\zeta_k))$ and $(\zeta_1,\dots,\zeta_k) = (B_{-q}^{a_1}(\eta_1),\dots,B_{-q}^{a_{\ell}}(\eta_{\ell}))$ where: $m_i,a_i > 0$ are integers, $\eta_i \in \mu_{(-q)^{a_i}-1}$, and the $\zeta_i \in \mathbb{F}^{\times}$ are pairwise distinct. Thus $W(s) \cong \mathfrak{S}_{m_1}\times\cdots\times \mathfrak{S}_{m_k}$.

For each $1 \leqslant i \leqslant \ell$ choose $\nu_i \in \mu_{q^{a_i}-1}$ and set $(\xi_1,\dots,\xi_k) = (B_q^{a_1}(\nu_1),\dots,B_q^{a_1}(\nu_1))$. The element $t = \diag(B_1^{m_1}(\xi_1),\dots,B_k^{m_k}(\xi_k))$ certainly satisfies ${}^{ww_0}F(t)=t$. We claim our assumption implies that we may choose the $\nu_i$ so that the $\xi_i \in \mathbb{F}^{\times}$ are pairwise distinct. We will then have $W(s) = W(t)$ as desired.

For this, assume $\nu_i^{q^a} = \nu_j^{q^b}$. Clearly $\nu_i^{q^a(q^{a_j}-1)} = 1$ and as $\gcd(q^a,q^{a_i}-1) = 1$ we must have the order of $\nu_i$ divides $q^{a_j}-1$ so $\nu_i \in \mu_{q^{a_j}-1}$. Applying the same argument to $\nu_j$ gives $\nu_i,\nu_j \in \mu_{q^{\gcd(a_i,a_j)}-1}$. For any integer $m > 0$ let $\bar{\mu}_{q^m-1}$ be the set difference $\mu_{q^m-1} \setminus \mu_{q^{m-1}-1}$ then note that for any $\nu \in \bar{\mu}_{q^m-1}$ we have $\nu^{q^a} \in \bar{\mu}_{q^m-1}$ for any integer $a \geqslant 0$. Moreover, $|\bar{\mu}_{q^m-1}| = q^m-q^{m-1} = q^{m-1}(q-1) \geqslant n$ by assumption. We leave it to the reader to conclude that we may choose the $\nu_i \in \bar{\mu}_{q^{a_i}-1}$ such that the resulting $\xi_i$ are pairwise distinct.
\end{proof}

\begin{rem}
If we assume $s = \diag(\zeta_1,\dots,\zeta_n)$ with $\zeta_i \in \mu_{q+1}$ all pairwise distinct, then $q+1\geqslant n$. We thus have ${}^{w_0}F'(s) = s$ and $W(s)$ is trivial. Any element $t$, as in \cref{lem:ennola-duality-ss-classes}, would need to be of the form $\diag(\nu_1,\dots,\nu_n)$ with $\nu_i \in \mu_{q-1}$ pairwise distinct so cannot exist if $q \leqslant n$. Hence, the assumption that $q>n$ is necessary in general.
\end{rem}

We will also need the following easy observation.

\begin{lem}\label{lem:ratio}
Let $f(t) \in \R[t]$ and $g(t) \in \R[t]$ be two polynomials with real coefficients. Suppose that $|f(t_i)| \leq |g(t_i)|$ for an infinite sequence 
$t_1 < t_2 < t_3 < \cdots$ tending to infinity. Then, for any $\epsilon > 0$  there is a constant $C=C(f,g,\epsilon)$ such that $|f(t)| \leq (1+\epsilon)|g(t)|$ whenever $|t| \geq C$.
\end{lem}

\begin{proof}
Without any loss we may assume that the leading coefficient $a$ of $f$ and the leading coefficient $b$ of $g$ are both positive, and write
\begin{equation*}
f(t) = at^m + \sum^{m-1}_{i=0}a_it^i, \quad g(t) = bt^n + \sum^{n-1}_{i=0}b_it^i.
\end{equation*}
The hypothesis now implies that either $m < n$, or $m=n$ and $a \leq b$. Now let
\begin{equation*}
A = \max_{0 \leq i \leq m-1}|a_i|, \quad B = \max_{0 \leq i \leq n-1}|b_i|,
\end{equation*}
and choose
\begin{equation*}
C= 1+\frac{A+(1+\epsilon)B}{(1+\epsilon)b-a}.
\end{equation*}
Then for any $t$ with $|t| \geq C$, we have
\begin{equation*}
|f(t)| \leq |t|^m(a+A/(C-1)), \qquad |g(t)| \geq |t|^n(b-B/(C-1)),
\end{equation*}
whence $|f(t)| \leq (1+\epsilon)|g(t)|$.
\end{proof}

\begin{rem}
The example of $(f(t),g(t),t_i) = (t-1,t,i)$ shows that Lemma \ref{lem:ratio} is false when we set $\epsilon=0$.
\end{rem}

\begin{pa}
For each partition $\lambda \vdash n$ we denote by $u_{\lambda}^+ \in \bG^F$ and $u_{\lambda}^- \in \bG^{F'}$ a unipotent element for which the sizes of the Jordan blocks in the Jordan normal form of the element are given by $\lambda$. Recall that for each $w \in W$ and partition $\lambda \vdash n$ there exist polynomials $\mathcal{Q}_{w,\lambda}^{\pm} \in \mathbb{Z}[t]$ such that $Q_w^{\bG,F}(u_{\lambda}^+) = \mathcal{Q}_{w,\lambda}^+(q)$ and $Q_w^{\bG,F'}(u_{\lambda}^-) = \mathcal{Q}_{w,\lambda}^-(q)$. Now Ennola duality states that
\begin{equation*}
\mathcal{Q}_{w,\lambda}^- = \mathcal{Q}_{ww_0,\lambda}^+(-t).
\end{equation*}
With this we may prove the following.
\end{pa}

\begin{prop}\label{gu-uni}
For any $n \geq 2$, there is a constant $C'=C'(n)$ such that the following statement holds.
Assume $\chi \in \Irr(\GU_n(q))$, with $q \geq C'$. Then for any unipotent element $u \in \GU_n(q)$ we have that
\begin{equation*}
|\chi(u)| \leqslant (g(n)+1) \cdot \chi(1)^{\frac{n-2}{n-1}}
\end{equation*}
where $g$ is the function defined in \cref{gl-uni}.
\end{prop}

\begin{proof}
Consider the character $\chi_{[s,\varphi]} = \varepsilon_{s,\widetilde{\varphi}}\mathcal{R}_{s,w}^{\bG,F'}(\widetilde{\varphi}) \in \Irr(\bG^{F'})$ for some family $[s,\varphi] \in \Fam(\bG,\bT_0)^{F'}$, pair $(s,w) \in \mathcal{S}_{\bG}(\bT_0,F')$, and extension $\widetilde{\varphi} \in \Irr(H\sd wF')$ of $\varphi$ where $H = W(s)$. We will assume that $\widetilde{\varphi}$ is one of the two extensions defined over $\mathbb{Q}$. Assuming $q > n$ there exists an element $t \in \bT_0$ such that $(t,ww_0) \in \mathcal{S}_{\bG}(\bT_0,F)$ and $H = W(t)$, see \cref{lem:ennola-duality-ss-classes}. As $F' = \boldsymbol{\iota}_{w_0}$ and $F$ is the identity, we have $\boldsymbol{\iota}_wF' = \boldsymbol{\iota}_{ww_0} = \boldsymbol{\iota}_{ww_0}F \in \Aut(H)$. In particular, $H \sd wF' = H \sd ww_0 = H \sd ww_0F$; and we have a family $[t,\varphi] \in \Fam(\bG,\bT_0)^F$. Hence we have a corresponding character $\chi_{[t,\varphi]} = \varepsilon_{t,\widetilde{\varphi}}\mathcal{R}_{t,ww_0}^{\bG,F}(\widetilde{\varphi}) \in \Irr(\bG^F)$.

Now if $\lambda\vdash n$ is a partition we have
\begin{align*}
\mathcal{R}_{s,w}^{\bG,F'}(\widetilde{\varphi})|_{u_{\lambda}^-} &= \frac{1}{|H|}\sum_{x \in H}\widetilde{\varphi}(xww_0)\mathcal{Q}_{xw,\lambda}^-(q),\\
\mathcal{R}_{t,ww_0}^{\bG,F}(\widetilde{\varphi})|_{u_{\lambda}^+} &= \frac{1}{|H|}\sum_{x \in H}\widetilde{\varphi}(xww_0)\mathcal{Q}_{xww_0,\lambda}^+(q).
%
\end{align*}
By Ennola duality, and our choice of extension $\widetilde{\varphi}$, it follows that there exists a polynomial $\mathcal{P}_{w,\widetilde{\varphi},\lambda}^H \in \mathbb{Q}[t]$ such that $\mathcal{R}_{s,w}^{\bG,F'}(\widetilde{\varphi})|_{u_{\lambda}^-} = \mathcal{P}_{w,\widetilde{\varphi},\lambda}^H(-q)$ and $\mathcal{R}_{t,ww_0}^{\bG,F}(\widetilde{\varphi})|_{u_{\lambda}^+} = \mathcal{P}_{w,\widetilde{\varphi},\lambda}^H(q)$. After \cref{gl-uni} we have for any prime power $q$ that
\begin{equation*}
|\chi_{[t,\varphi]}(u_{\lambda}^+)| = |\mathcal{P}_{w,\widetilde{\varphi},\lambda}^H(q)| \leqslant g(n)\cdot |\mathcal{P}_{w,\widetilde{\varphi},(1^n)}^H(q)|^{\frac{n-2}{n-1}} = g(n) \cdot \chi_{[t,\varphi]}(1)^{\frac{n-2}{n-1}}
\end{equation*}
Thus, for the pair
\begin{equation*}
f_{w,\widetilde{\varphi},\lambda}^H(t) = \bigl(\mathcal{P}_{w,\widetilde{\varphi},\lambda}^H(t)\bigr)^{n-1}, \quad
    h_{w,\widetilde{\varphi}}^H(t) = g(n)^{n-1}\bigl(\mathcal{P}_{w,\widetilde{\varphi},(1^n)}^H(t)\bigr)^{n-2}
\end{equation*}
of polynomials in $t$ we have that $|f_{w,\widetilde{\varphi},\lambda}^H(q)| \leq |h_{w,\widetilde{\varphi}}^H(q)|$ for all prime powers $q$.

Now we can apply Lemma \ref{lem:ratio} and choose
\begin{equation*}
\epsilon = (1+1/g(n))^{n-1}-1, \qquad C' = \max\left(n+1,\max_{(H,w,\widetilde{\varphi},\lambda)}C(f_{w,\widetilde{\varphi},\lambda}^H,h_{w,\widetilde{\varphi}}^H,\epsilon)\right).
\end{equation*}
Note that the number of possible tuples $(H,w,\widetilde{\varphi},\lambda)$ we need to consider to cover all irreducible characters of $\GU_n(q)$ is bounded in terms of $n$. Indeed, we can take the first term over all standard parabolic subgroups of $W$, the second over $N_W(H)w_0$, the third is identified with a subset of $\Irr(H)^F \subseteq \Irr(H)$, and the last is over all partitions of $n$.

Let $\mathcal{P}_{\lambda} = \mathcal{P}_{w,\widetilde{\varphi},\lambda}^H$, $f_{\lambda} = f_{w,\widetilde{\varphi},\lambda}^H$, and $h = h_{w,\widetilde{\varphi}}^H$. Then for all $q \geq C'$ we have
\begin{equation*}
|\mathcal{P}_{\lambda}(-q)|^{n-1}=|f_{\lambda}(-q)| \leq (1+1/g(n))^{n-1}|h(-q)| = (g(n)+1)^{n-1}|\mathcal{P}_{(1^n)}(-q)|^{n-2},
\end{equation*}
and so
\begin{equation*}
|\chi_{[s,\varphi]}(u_{\lambda}^-)| = |\mathcal{P}_{\lambda}(-q)| \leqslant (g(n)+1)\cdot |\mathcal{P}_{(1^n)}(-q)|^{\frac{n-2}{n-1}} = (g(n)+1) \cdot \chi_{[s,\varphi]}(1)^{\frac{n-2}{n-1}}
\end{equation*}
as desired.
\end{proof}

\begin{proof}[of \cref{thm:gusu}]
Let $x = su = us$ with $s$ semisimple and $u$ unipotent. Also view $G = \bG^F$ with $\bG = \GL_n(\mathbb{F})$ or $\bG = \SL_n(\mathbb{F})$, and $F$ a suitable Frobenius 
endomorphism. 

(a) First we consider the case $G = \GU_n(q)$. If $s \in Z(G)$ then the statement follows from \cref{gu-uni}, by choosing
$C^*(n) \geq C'(n)$ and $h^*(n) \geq g(n)+1$. If $s \notin Z(G)$, then $C_{\bG}(g)$ is contained in an $F$-stable proper Levi subgroup, and
so we are done by  \cref{thm:character-bound} and \cite[Prop.~4.3]{bezrukavnikov-liebeck-shalev-tiep:2017:character-bounds-grps-Lie-type},
by choosing $h^*(n) \geq f(n-1)$.

(b) Now consider the case $G = \SU_n(q)$ and view $G$ as $[\tilde G,\tilde G]$, where $\tilde G \cong \GU_n(q)$. Arguing as in the proof of
\cite[Thm.~1.5]{bezrukavnikov-liebeck-shalev-tiep:2017:character-bounds-grps-Lie-type}, we are done if $\chi$ is $\tilde G$-invariant, or if 
$x^G = x^{\tilde G}$. We will also choose $C^*(n) \geq g_A(n)$, where $g_A(n)$ is the function mentioned in \cref{types A/C}. Hence, if $s \notin Z(G)$, then $C_{\bG}(x)$ is contained in an $F$-stable proper Levi subgroup, and
so we are done by  \cref{thm:character-bound2}
and \cite[Prop.~4.3]{bezrukavnikov-liebeck-shalev-tiep:2017:character-bounds-grps-Lie-type}.

Thus we may now assume that $\chi$ is not $\tilde G$-invariant, $x^G \neq x^{\tilde G}$, and that $s = 1$ so $x = u$. Let $\tilde\chi \in \Irr(\tilde G)$ lying above
$\chi$. Since $\chi$ is not $\tilde G$-invariant, $\tilde\chi$ is reducible over $G$, whence $\tilde\chi(1) > q^{n^2/4-2}$ by \cite[Thm.~3.4]{LBST}, and so
\begin{equation}\label{eq:deg-gu}
  \chi(1) \geq \tilde\chi(1)/(q+1) > q^{n^2/4-3.6}.
\end{equation}  

Let $r_i$ denote the number of Jordan blocks of size $i$ in the Jordan canonical form of $u$ for each $i \geq 1$;
in particular, $\sum_i ir_i = n$. It is easy to see that the condition $u^G \neq u^{\tilde G}$ implies
$\gcd(i \mid r_i \geq 1) > 1$, in particular, $r_1 = 0$. We claim (for $n \geq 5$) that either
\begin{equation}\label{su-4}
  |\CB_G(u)| \leq q^{(n^2-3n+6)/2} \cdot 1.5^{n/2}
\end{equation}
or $u$ has type $J_2^{n/2}$, i.e., $r_2 = n/2$. Indeed, the number $t$ of indices $i$ with $r_i > 0$ is at most $n/2$ since $r_1 = 0$.
Furthermore, $|\GU_{r_i}(q)| \leq (1+1/q)q^{r_i^2} \leq (1.5)q^{r_i^2}$.
It follows from Theorems 3.1 and 7.1 of \cite{LS} that
\begin{equation*}
|\CB_G(u)| < q^N \cdot 1.5^t \leq q^N \cdot 1.5^{n/2},
\end{equation*}
where
\begin{equation*}
N = \dim \CB_{\tilde \bG}(u) = \sum_i ir_i^2 + 2\sum_{i < j}ir_ir_j.
\end{equation*}
As shown in the proof of \cite[Thm.~1.5]{bezrukavnikov-liebeck-shalev-tiep:2017:character-bounds-grps-Lie-type}, $N \leq (n^2-3n+6)/2$
unless $r_2 = n/2$, whence the claim follows.

In the case of \eqref{su-4}, $|\chi(u)| \leq q^{(n^2-3n+6)/4} \cdot 1.5^{n/4}$. Taking $h^*(n) \geq 1.5^{n/4}$ and noting that
\begin{equation*}
(n^2/4 - 3.6) \cdot \frac{n-2}{n-1} > (n^2-3n+6)/4
\end{equation*}
when $n \geq 10$, we are done because of \eqref{eq:deg-gu}.

It remains to consider the case $u = J_2^{n/2}$. Write $n =2m$ and let $W = \mathbb{F}_{q^2}^n = \langle e_1, \ldots ,e_m,f_1, \ldots,f_m \rangle_{\mathbb{F}_{q^2}}$
denote the natural module for $G$, where $(e_1, \ldots,e_m,f_1, \ldots,f_m)$ is a Witt basis for the Hermitian form on $W$. 
We can find nonzero scalars $a_1, \ldots,a_m \in \mathbb{F}_{q^2}$ such that $u$ is represented by the matrix
\begin{equation*}
\diag \left( \begin{pmatrix}1 & a_1\\0 & 1 \end{pmatrix}, \begin{pmatrix}1 & a_2\\0 & 1 \end{pmatrix}, \ldots, \begin{pmatrix}1 & a_m\\0 & 1 \end{pmatrix}\right)
\end{equation*}
in the basis $(e_1,f_1,e_2,f_2, \ldots,e_m,f_m)$ of $W$. As above, we have that $|\CB_G(u)| < (1.5)q^{n^2/2}$, whence
\begin{equation}\label{su-5}
|\chi(u)| < (1.3)q^{n^2/4}.
\end{equation}
Suppose first that
\begin{equation}\label{su-6}
\chi(1) > q^{(n-1)(n-2)/2}.
\end{equation}
As $n \geq 7$, then \eqref{su-5} and \eqref{su-6} immediately imply that $|\chi(u)| < (1.3)\chi(1)^{(n-2)/(n-1)}$.

It remains to consider the case where \eqref{su-6} does not hold.
Let $\chi$ be afforded by a $\Ql G$-module $V$ and let $P := \Stab_G(\langle e_1 \rangle_{\mathbb{F}_{q^2}}) = UL$, with $U$ being the unipotent radical
and $L$ being a Levi subgroup. Note that $u=tv$, where
\begin{equation*}
t = \diag \left( \begin{pmatrix}1 & a_1\\0 & 1 \end{pmatrix}, I_{2m-2}\right) \in \bZ(U)
\end{equation*}
and
\begin{equation*}
v = \diag \left( I_2, \begin{pmatrix}1 & a_2\\0 & 1 \end{pmatrix}, \ldots, \begin{pmatrix}1 & a_m\\0 & 1 \end{pmatrix}\right) \in \SU_{2m-2}(q) = [L,L]
\end{equation*}
in the basis $(e_1,f_1,e_2,f_2, \ldots,e_m,f_m)$.

We decompose the $P$-module $V$ as $\CB_V(U) \oplus V_1 \oplus V_2$, where 
$V_1 := [U,\CB_{\bZ(U)}(V)]$ and $V_2:= [\bZ(U),V]$, and let $\chi_0$, respectively $\chi_1$ and $\chi_2$, denote the $P$-character of
$\CB_V(U)$, respectively of $V_1$ and $V_2$. In particular, $\chi_0 = \,^*R^G_L(\chi)$, and so, arguing as in part (ii) of the proof of
\cite[Thm.~1.4]{bezrukavnikov-liebeck-shalev-tiep:2017:character-bounds-grps-Lie-type}
we get
\begin{equation}\label{su-7}
|\chi_0(u)| \leq \chi_0(1) \leq f(n-1)\chi(1)^{\frac{n-2}{n-1}}.
\end{equation}
Next, we decompose
\begin{equation*}
V_1 = \sum_{1_{U/\bZ(U)} \neq \lambda \in \Irr(U/\bZ(U))}V_\lambda,
\end{equation*}
as a direct sum of $U$-eigenspaces, which are permuted by $L \cong \SU_{n-2}(q) \cdot C_{q^2-1}$.

Note that $u$ has prime order $p \mid q$, and it acts on $\Irr(U/\bZ(U)) \smallsetminus \{1_{U/\bZ(U)} \}$ with exactly $q^{m-1}-1$ fixed points. 
Certainly, the trace of
$u$ in its action on $\sum_{\lambda \in \mathcal{O}'}V_{\lambda}$ for any nontrivial $u$-orbit $\mathcal{O}'$ on $\Irr(U/\bZ(U)) \smallsetminus \{1_{U/\bZ(U)} \}$ is zero. On the other hand, each $L$-orbit $\mathcal{O}$ on $\Irr(U/\bZ(U)) \smallsetminus \{1_{U/\bZ(U)} \}$ has length $q^{2m-3}(q^{2m-2}-1)$ or
$(q^{2m-3}+1)(q^{2m-2}-1)$. Writing $\chi_{\mathcal {O}}$ for the $P$-character of the nonzero submodule of $V_1$ corresponding to 
$\mathcal{O} \ni \lambda$, we then have
\begin{equation*} 
|\chi_{\mathcal{O}}(u)| \leq (q^{m-1}-1)\dim(V_{\lambda}) = (q^{m-1}-1) \cdot \frac{\chi_{\mathcal{O}}(1)}{|\mathcal{O}|} < 
\frac{\chi_{\mathcal{O}}(1)}{q^{m-1}}.
\end{equation*}
Summing over all $\mathcal{O}$ occuring in $V_1$, we get
\begin{equation}\label{su-8}
  |\chi_1(u)| \leq \frac{\chi_1(1)}{q^{m-1}}.
\end{equation}  

Finally, as explained in \cite[\S5]{GMST}, we can decompose
\begin{equation*}
V_2 = \sum_{1_{\bZ(U)} \neq \beta \in \Irr(\bZ(U))}E_\beta \otimes V'_\beta,
\end{equation*}
as a direct sum of $\bZ(U)$-eigenspaces. Here, $E_\beta$ lies over $\beta \in  \Irr(\bZ(U)) \smallsetminus \{1_{\bZ(U)}\}$, and its restriction to 
$[L,L] \cong \SU_{n-2}(q) \ni v$ yields a reducible Weil module of dimension $q^{n-2}$, and, according to \cite[\S4]{TZ}, the trace of 
$v$ on $E_\beta$ is $q^{m-1}$. Furthermore, $U$ acts trivially on $V'_\beta$. Thus the absolute value of the trace of $u$ on 
$E_\beta \otimes V'_\beta$ is at most $\dim(E_\beta \otimes V'_\beta)/q^{m-1}$, and so
\begin{equation}\label{su-9}
|\chi_2(u)| \leq \frac{\chi_3(1)}{q^{m-1}}.
\end{equation}   
Since $\chi(1) \leq q^{(n-1)(n-2)/2}$, we have that
\begin{equation*}
|(\chi_1+\chi_2)(u)| \leq \frac{\chi(1)}{q^{n/2-1}} \leq \chi(1)^{\frac{n-2}{n-1}}.
\end{equation*}
Together with \eqref{su-7}, this completes the proof, if we choose $h^*(n) \geq f(n-1)+1$.
\end{proof}

\begin{proof}[of \cref{cor:slsu}]
Note that Theorem 1.3 and Proposition 4.3 of \cite{bezrukavnikov-liebeck-shalev-tiep:2017:character-bounds-grps-Lie-type} show that 
the exponent $\dfrac{n-2}{n-1}$ in \cref{thm:glsl,thm:gusu} is best possible. Next, the proofs of Corollary 1.14 of \cite{bezrukavnikov-liebeck-shalev-tiep:2017:character-bounds-grps-Lie-type} can be repeated verbatim, but using \cref{thm:glsl}, respectively \cref{thm:gusu}, instead of \cite[Thm.~1.5]{bezrukavnikov-liebeck-shalev-tiep:2017:character-bounds-grps-Lie-type} to yield \cref{cor:slsu}.
\end{proof}

\appendix
\section{Decomposing Semisimple Objects in Abelian Categories}\label{sec:decomp-semisimple-obj}
\begin{assumption}
From now until the end of this section we assume that $\mathscr{A}$ is a locally finite $k$-linear abelian category, where $k = \overline{k}$ is an algebraically closed field, and $K \in \mathscr{A}$ is a fixed semisimple object.
\end{assumption}

\begin{pa}\label{pa:semisimple-obj}
We refer to \cite[Chapter 1]{etingof-gelaki-nikshych-ostrik:2015:tensor-categories} for the basic definitions concerning abelian categories. Recall that an object $A \in \mathscr{A}$ is said to be a \emph{summand} of $K$ if there exists a pair of morphisms
\begin{equation*}
\begin{tikzcd}
A \arrow[r,bend left,"m"] & K \arrow[l,bend left,"p"]
\end{tikzcd}
\end{equation*}
such that $pm = \Id_A$. Note we necessarily have $p$ is an epimorphism, $m$ is a monomorphism, and $u = mp$ is an idempotent. Moreover, if $B = \Ker(u)$ then there exist morphisms
\begin{equation*}
\begin{tikzcd}
B \arrow[r,bend left,"\iota"] & K \arrow[l,bend left,"q"]
\end{tikzcd}
\end{equation*}
such that $q\iota = \Id_B$ and $mp + \iota q = \Id_K$. In other words, we have $K \cong A \oplus B$. We write $A \mid K$ to indicate that $A$ is a summand of $K$.
\end{pa}

\begin{pa}
As $K$ is assumed to be semisimple there exist finitely many simple objects $A_1,\dots,A_r \in \mathscr{A}$ and morphisms
\begin{equation*}
\begin{tikzcd}
A_j \arrow[r,bend left,pos=0.45,"m_j"] & K \arrow[l,bend left,pos=0.51,"p_j"]
\end{tikzcd}
\end{equation*}
such that $m_1p_1+\cdots+m_rp_r = \Id_K$ and $p_im_j = \delta_{i,j}\Id_{A_j}$, where $\delta_{i,j}$ is the Kronecker delta. We usually write $K = A_1\oplus\cdots\oplus A_r$ to indicate it is a direct sum. As the Krull-Schmidt theorem holds in $\mathscr{A}$ we have the following.
\end{pa}

\begin{lem}\label{lem:semisimple-summands}
If $A \mid K$ is a summand of $K$ then $A$ is semisimple and we have an injection $\Irr(\mathscr{A}\mid A) \to \Irr(\mathscr{A}\mid K)$.
\end{lem}

\begin{pa}
We will denote by $\mathcal{A} = \End_{\mathscr{A}}(K)$ the endomorphism algebra of $K$, which is a finite dimensional $k$-algebra. We have a contravariant $k$-linear functor
\begin{equation*}
\mathfrak{F}_K = \Hom_{\mathscr{A}}(-,K) : \mathscr{A} \to \lmod{\mathcal{A}}
\end{equation*}
where $\Hom_{\mathscr{A}}(A,K)$ is naturally a left $\mathcal{A}$-module via left composition.
\end{pa}

\begin{lem}\label{lem:bij-iso-classes}
Recall our assumption that $K$ is semisimple. The algebra $\mathcal{A}$ is semisimple and the functor $\mathfrak{F}_K$ defines a bijection $\Irr(\mathscr{A}\mid K) \to \Irr(\mathcal{A})$.
\end{lem}

\begin{proof}
We have $\mathcal{A} = \oplus_{j=1}^r \mathcal{A}e_j$, where $e_j = m_jp_j \in \mathcal{A}$ is an idempotent. For any $1 \leqslant j,k\leqslant r$ we have a $k$-linear isomorphism $\Hom_{\mathscr{A}}(A_j,A_k) \to e_k\mathcal{A}e_j$ defined by $f \mapsto m_jfp_k$ so $\mathcal{A}e_j$ is simple by Schur's Lemma, see \cite[1.8.4]{etingof-gelaki-nikshych-ostrik:2015:tensor-categories} and \cite[3.18]{curtis-reiner:1981:methods-vol-I}. Thus $\mathcal{A}$ is semisimple and $\mathfrak{F}_K(A_j)$ is simple because right multiplication by $p_j$ defines an isomorphism $\mathfrak{F}_K(A_j) \to \mathcal{A}e_j$ of $\mathcal{A}$-modules.

The resulting map on isomorphism classes is surjective because every simple $\mathcal{A}$-module is a submodule of $\mathcal{A}$. Moreover, this is injective by Schur's Lemma because we have standard $k$-linear isomorphisms
\begin{equation*}
\Hom_{\mathcal{A}}(\mathfrak{F}_K(A_j),\mathfrak{F}_K(A_k)) \cong \Hom_{\mathcal{A}}(\mathcal{A}e_j,\mathcal{A}e_k) \cong e_j\mathcal{A}e_k \cong \Hom_{\mathcal{A}}(A_k,A_j)
\end{equation*}
where the second isomorphism is given by $f \mapsto e_jf(e_j)e_k$.
\end{proof}

\begin{lem}\label{lem:hom-to-end-alg}
For any summand $A \mid K$ and object $B \in \mathscr{A}$ we have a $k$-linear isomorphism
\begin{equation*}
\mathfrak{F}_K : \Hom_{\mathscr{A}}(B,A) \to \Hom_{\mathcal{A}}(\mathfrak{F}_K(A),\mathfrak{F}_K(B)).
\end{equation*}
\end{lem}

\begin{proof}
By \cref{lem:semisimple-summands} we can assume $A = A_j$ for some $1 \leqslant j \leqslant r$. Now assume $\varphi \in \Hom_{\mathscr{A}}(B,A_j)$ satisfies $\mathfrak{F}_K(\varphi) = 0$ so $m_j\varphi = 0$. As $m_j$ is a monomorphism this implies $\varphi = 0$, hence the map is injective. Counting dimensions, exactly as in the proof of \cite[4.1.2]{geck-jacon:2011:hecke-algebras}, we get the map is surjective.
\end{proof}

\begin{pa}
We now assume that $\mathscr{B}$ is another locally finite $k$-linear abelian category and $\mathfrak{I} : \mathscr{A} \to \mathscr{B}$ is a $k$-linear functor. If $L := \mathfrak{I}(K)$ and $\mathcal{B} := \End_{\mathscr{B}}(L)$ then $\mathfrak{I}$ defines a $k$-algebra homomorphism $\mathfrak{I} : \mathcal{A} \to \mathcal{B}$. In particular, we may view $\mathcal{B}$ as an $(\mathcal{A},\mathcal{A})$-bimodule by restricting through $\mathfrak{I}$. With this we have a corresponding functor $\mathcal{B}\otimes_{\mathcal{A}}- : \lmod{\mathcal{A}} \to \lmod{\mathcal{B}}$ where $\mathcal{B}\otimes_{\mathcal{A}} M$, for any $\mathcal{A}$-module $M \in \lmod{\mathcal{A}}$, is a left $\mathcal{B}$-module in the usual way.
\end{pa}

\begin{lem}\label{lem:I-func-is-ind}
If $A \mid K$ is a summand of $K$ then we have a $\mathcal{B}$-module isomorphism
\begin{equation*}
\phi : \mathcal{B} \otimes_{\mathcal{A}} \mathfrak{F}_K(A) \to \mathfrak{F}_L(\mathfrak{I}(A))
\end{equation*}
satisfying $\phi(b \otimes f) = b\mathfrak{I}(f)$.
\end{lem}

\begin{proof}
Let $m : A \to K$ and $p : K \to A$ be morphisms such that $pm = \Id_A$ and $u = mp \in \mathcal{A}$ is an idempotent. As
\begin{equation*}
\mathfrak{I}(p)\mathfrak{I}(m) = \mathfrak{I}(pm) = \mathfrak{I}(\Id_A) = \Id_{\mathfrak{I}(A)}
\end{equation*}
we have $\mathfrak{I}(A) \mid \mathfrak{I}(K)$ and $\mathfrak{I}(u)$ is an idempotent. By the universal property of the tensor product we have a $\mathcal{B}$-module homomorphism $\phi : \mathcal{B} \otimes_{\mathcal{A}} \mathcal{A}u \to \mathcal{B}\mathfrak{I}(u)$ satisfying $\phi(b \otimes a) = b\mathfrak{I}(a)$. Moreover, we have a $\mathcal{B}$-module homomorphism $\psi : \mathcal{B}\mathfrak{I}(u) \to \mathcal{B} \otimes_{\mathcal{A}} \mathcal{A}u$ satisfying $\psi(x) = x \otimes u$. As $x\mathfrak{I}(u) = x$ for any $x \in \mathcal{B}\mathfrak{I}(u)$ we clearly have $\phi\psi$ is the identity, hence $\phi$ is an isomorphism. The statement now follows because right multiplication by $p$ defines an isomorphism of $\mathcal{A}$-modules $\mathfrak{F}_K(A) \to \mathcal{A}u$ and left multiplication by $\mathfrak{I}(m)$ defines an isomorphism of $\mathcal{B}$-modules $\mathcal{B}\mathfrak{I}(u) \to \mathfrak{F}_L(\mathfrak{I}(A))$.
\end{proof}

\begin{cor}\label{cor:decomp-ind-semisimple}
If $L = \mathfrak{I}(K)$ is semisimple and $B \mid L$ then for any $A \mid K$ we have a $k$-linear isomorphism
\begin{equation*}
\Hom_{\mathscr{B}}(\mathfrak{I}(A),B) \to \Hom_{\mathcal{B}}(\mathfrak{F}_L(B),\mathcal{B} \otimes_{\mathcal{A}}\mathfrak{F}_K(A)).
\end{equation*}
\end{cor}

\begin{proof}
As $L$ is semisimple we have a $k$-linear isomorphism
\begin{equation*}
\Hom_{\mathscr{B}}(\mathfrak{I}(A),B) \cong \Hom_{\mathcal{B}}(\mathfrak{F}_L(B),\mathfrak{F}_L(\mathfrak{I}(A)))
\end{equation*}
by \cref{lem:hom-to-end-alg}. The statement now follows from \cref{lem:I-func-is-ind}.
\end{proof}

\setstretch{0.96}

\end{document}